\documentclass[12pt,a4paper]{amsart}
\usepackage[utf8]{inputenc}
\usepackage{amssymb}
\usepackage{amsmath,amsthm}
\usepackage{color}
\usepackage[mathscr]{eucal}
\usepackage{mathrsfs}
\usepackage{mathdots}

\usepackage[top=30truemm,bottom=30truemm,left=30truemm,right=30truemm]{geometry}

\usepackage{color,fullpage}
\usepackage{amsfonts}
\usepackage{hyperref}
\usepackage{cleveref}
\usepackage{enumitem}
\usepackage{ulem}
\usepackage{comment}

\newtheorem{theorem}{Theorem}[section]
\newtheorem{proposition}[theorem]{Proposition}
\newtheorem{definition}[theorem]{Definition}
\newtheorem{remark}[theorem]{Remark}
\newtheorem{lemma}[theorem]{Lemma}
\newtheorem{corollary}[theorem]{Corollary}
\newtheorem{example}[theorem]{Example}

\numberwithin{equation}{section}

\def\fa{{\mathfrak a}}
\def\fb{{\mathfrak b}}
\def\fc{{\mathfrak c}}
\def\fd{{\mathfrak d}}
\def\fL{{\mathfrak L}}
\def\fG{{\mathfrak G}}
\def\Z{{\mathbb Z}}
\def\C{{\mathbb C}}
\def\Q{{\mathbb Q}}
\def\R{{\mathbb R}}
\def\Re{{\rm Re}}
\def\leq{\leqslant}
\def\geq{\geqslant}

\title{Counting Lattices with Local Hecke Series}
\author[Gautami Bhowmik] {Gautami Bhowmik }

\address{Laboratoire Paul Painlev\'e LABEX-C2EMPI , Universit\'e de Lille, Batiment M2, 59655 Villeneuve-d'Ascq Cedex, France}
\email{gautami.bhowmik@univ-lille.fr}
\author[Masao Tsuzuki] {Masao Tsuzuki}

\address{ Faculty of Science and Technology, Sophia University, Kioi-cho 7-1 Chiyoda-ku Tokyo, 102-8554, Japan}
\email{m-tsuduk@sophia.ac.jp}


\begin{document}

\begin{abstract} 
 We count the maximal lattices over $p$-adic fields and the rational number field. For this, we use the theory of Hecke series for a reductive group over nonarchimedean local fields, which was developed by Andrianov and Hina-Sugano. By treating the Euler factors of the counting Dirichlet series for lattices, we obtain zeta functions of classical groups, which were earlier studied with $p$-adic cone integrals. When our counting series equals the existing zeta functions of groups, we recover the known results in a simple way. Further we obtain some new zeta functions for non-split even orthogonal and odd orthogonal groups.
\end{abstract}

\maketitle

\section{Introduction}

In this paper we count  
lattices over the $p$-adic field $\Q_p$ and over the rational number field $\Q$, a topic that fits in a broad family of the
enumeration of the finite-index subgroups of a finitely generated  group. We start with the ${\rm GL}$-case, a classical example, which can be seen as a prototype of our study. A very simple problem would be to count the number of subgroups $L$ of the free abelian group $\Z^n$, $n$  a natural number, with index $m\,(\geq 1)$, i.e., to understand the generating series of the arithmetic function
$$
\nu(m):=\#\{L \subset \Z^n \mid [\Z^n:L]=m\}. 
$$
\begin{definition}
The lattice counting series of $\Z^n$ is defined by
\begin{equation}  
\zeta_{\Z^n}(s):=\sum_{m=1}^{\infty}
\frac{\nu(m)}{m^{s}}
=\sum_{L\subset \Z^n} [\Z^n:L]^{-s}.
\label{generalzeta}
\end{equation}
\end{definition}
We easily see that the series $\zeta_{\Z^n}(s)$ is absolutely convergent on $\Re(s)>n$ and can be explicitly written as the product of shifts of the Riemann zeta function 
\begin{align}
\zeta_{\Z^n}(s)=\zeta(s)\zeta(s-1)\cdots \zeta(s-n+1),\quad \Re(s)>n, 
\label{HeyZeta}
\end{align}
a folkloric result, see for example 
\cite{Tamagawa1963}. This expression yields a meromorphic continuation of $\zeta_{\Z^n}(s)$ to 
$\C$, with the residue 
$\zeta(n)\zeta(n-1)\cdots \zeta(2)$ at its right-most pole $s=n$, which gives an elementary 
asymptotic result 
$$
\sum_{m<X} \nu_L(m) = \big(\prod_{j=2}^{n}\zeta(j)\big)\,X^{n} 
+
\mathcal{O}(X^{n-1}\log X), \quad X\rightarrow \infty 
$$
for the number of sublattices of $\Z^n$ with index at most $X$.

The lattice counting function (\ref{generalzeta}) introduced above can 
be seen as a special case of a more general subgroup counting function for groups,  which we recall now:
\begin{definition}\label{Intro-Def1}
Let $\Gamma$ be a finitely generated group and ${\mathfrak h}$ a family of its finite-index subgroups. The subgroup growth zeta function of $\Gamma$ is defined by
 \[
\zeta_{\Gamma}(s):=\sum_{H \in {\mathfrak h}}\frac{1}{[\Gamma:H]^s}=\sum_{n=1}^{\infty}\frac{\nu_\Gamma^{\mathfrak h}(n)}{n^s}, 
\]
where \(\nu_\Gamma^{\mathfrak h}(n)\) denotes the number of subgroups in $\mathfrak h$ of index \(n\) in \(\Gamma\).
\end{definition}
This includes the following well-known examples of the zeta functions for the lattice counting with additional structures: 
\begin{itemize}
\item[(i)] the Dedekind zeta-function, where $\Gamma$ is the maximal order of a number field and ${\mathfrak h}$ is the set of all invertible ideals of $\Gamma$.  
\item[(ii)] the zeta-function of orders, where $\Gamma$ is the maximal order of a number field and ${\mathfrak h}$ is the set of all orders (\cite{Nakagawa}, \cite{KaplanMarcinekTaklooBighash}).
\end{itemize}

In a by-now classical work, Grunewald, Segal and Smith  \cite{GSS1988} showed that for a finitely generated 
torsion-free nilpotent group and the family of its all finite-index subgroups, the Euler $p$-factors are 
rational in $\ p^{-s}$.
Under certain nice properties of $\Gamma$, for example
for nilpotent groups, the above series is known to 
decompose locally into a convergent Euler product of
$\zeta_{\Gamma_p}(s)$ for
the pro $p$-completion $\Gamma_p$ of $\Gamma$ over all prime numbers $p$. 
In this paper we do not exploit the subgroup growth zeta function in 
its full generality, but keep track of it to establish its links 
with the counting of lattices. A first observation is that \eqref{generalzeta} is indeed the subgroup-growth zeta function of $\Z^n$; in this case, the pro $p$-completion $\Gamma_p$ is the $p$-adic integer lattice $\Z_p^n$, for
which $\zeta_{\Gamma_p}(s)$ is a product of $p$-components of the Riemann-zeta functions,
\[
\zeta_{\Z_p^n}(s)=(1-p^{-s})^{-1} (s-p^{-(s-1)})^{-1} \cdots 
(1-p^{-(s-n+1)})^{-1}.
\]
We now extend the setting of the ${\rm GL}$-case by introducing a
bilinear form on the ambient vector space $\Q^n$ of $\Z^n$.
 Let $S\in {\bf GL}_n(\Q)$ ($n\geq 2)$ be a symmetric matrix 
(i.e., ${}^t S=S$) or an alternating matrix (i.e., ${}^tS=-S$), and
let 
us introduce a bilinear form 
$\beta({\bf x},{\bf y})={}^t {\bf x}S {\bf y}$ on $\Q^{n}$ to invoke 
the dual  $\widehat L:=\{{\bf x}\in \Q^n\mid \beta({\bf x}, L)
\subset \Z\}$ of a lattice $L\subset \Q^{n}$. Then we recall the
notion of maximal lattices, which is viewed as a generalization of
fractional ideals over the maximal order of a quadratic field.

\begin{definition} Given a $\Z$-lattice $L\subset \Q^{n}$, when $S$
is symmetric \textup{$($}resp. alternating\textup{$)$}, we define its {\it norm} $\mu(L)$ to
be the fractional $\Z$-ideal generated by $2^{-1}\beta({\bf x}, {\bf x})\,({\bf x}\in L)$ \textup{(}resp. $\beta({\bf x}, {\bf y})\,( {\bf x}, {\bf y} \in L)$\textup{)}. Then a lattice $L$ is called a maximal lattice if it is maximal in the set of all lattices with the same norm. 
\end{definition}
Maximal lattices  occur naturally
 in the theory of automorphic forms on classical groups and 
related Hecke algebras.  We give more details  in 
\S\ref{sec:MaxLattice} and 
\S\ref{sec:GlobalCTN}. Now let $\fL_{\beta}^{{\rm max}}$ denote the set of all maximal lattices in $\Q^n$. Assume $\Z^n\in \fL_\beta^{\rm max}$. We are interested in counting the number of 
maximal lattices contained in $\Z^n$ with a given index 
$m\geq 1$:
$$
{\rm N}_{\beta}(m):=\#\{L \in \fL_\beta^{\rm max}\mid L\subset \Z^n,\,[\Z^n:L]=m\}.
$$
\begin{definition}
The maximal lattice counting series of $\Z^n$ is given by
\begin{equation}
\zeta_{\beta, \Z^n}(s):=\sum_{m=1}^{\infty}
\frac{{\rm N}_{\beta}(m)}{m^{s}}
=\sum_{\substack{L\in \fL_{\beta}^{\rm max} \\ L\subset \Z^n}} [\Z^n:L]^{-s}.
\label{maximalzeta}
\end{equation}
\end{definition}
This series is again in the realm of Definition \ref{Intro-Def1} with $\Gamma=\Z^n$ and ${\mathfrak h}=\fL_{\beta}^{\rm max}$. There is another class of lattices, the polarized lattices, that is defined below to include both symmetric and alternating forms.

\begin{definition} A $\Z$-lattice $L\subset \Q^n$ is said to
be {\it even $\beta$-polarized} if there exists $c\in \Q^\times$ such 
that $\widehat L=c L$ and $\beta({\bf x}, {\bf x}) \in 
2c^{-1}\Z$ for all ${\bf x}\in L$, i.e., if $L$ is an even unimodular lattice with respect to the bilinear form $c\beta$.
\end{definition}
Let $\fL_{\beta}^{{\rm pl}}$ denote the set of all even $\beta$-polarized lattices in $\Q^n$ and assume that $\Z^n\in \fL_\beta^{\rm pl}$. We now evoke the {\it polarized lattice counting series}, which was introduced in \cite{BG2005}:
\begin{equation} 
\zeta^{\rm pl}_{\beta, \Z^n}(s):=\sum_{m=1}^{\infty}
\frac{{\rm N}_{\beta}^{\rm pl}(m)}{m^{s}}
=\sum_{\substack{L\in \fL_{\beta}^{\rm pl} \\ L\subset \Z^n}} [\Z^n:L]^{-s}.
\label{polarizedzeta}
\end{equation}
Since an even $\beta$-polarized lattice is a maximal lattice, 
$\zeta_{\beta,\Z^n}^{\rm pl}(s)$ is a sub-series of $\zeta_{\beta, 
\Z^n}(s)$; in the split cases, the two zeta functions 
coincide.\footnote{This will be elaborated later. In this paper,
$\zeta_{\beta,\Z^n}^{\rm pl}(s)$ will appear only as the maximal 
lattice counting series $\zeta_{\beta,\Z^n}(s)$ with a split $\beta$ 
and the superscript is omitted. }

The  study of group zeta functions over the last thirty years in 
particular by du Sautoy, Grunewald, Lubotzky, 
Segal \cite{dSL1996, dSG2000, LS2003},  often use an Euler
product to obtain analytic information through the $p$-factors  by 
writing them as $p$-adic cone integrals. In their work we encounter
the \textit{$p$-adic} zeta function, which occurs in earlier works of 
Hey\cite{H1929} and Weil \cite{W1982} and was particularly brought in the forefront through the work of Igusa \cite{I1989}. Let us recall the definition of this zeta function :
\begin{definition} \label{Intro-Def2} 
The p-adic zeta function of \(\mathfrak G\), a linear algebraic group over \(\Q_p\),
is defined by the Haar integral
\[Z_{\fG(\Q_p),\rho}(s)= \int_{\fG_p^+} |\det 
\rho(g)|_p^{s}\,d\mu(g) \qquad s\in \C,\]
where \(\fG_p^+=\rho^{-1}(\rho (\fG(\Q_p))\cap {\bf Mat}_n (\Z_p)) , \ \mid\cdot\mid_p\) denotes
the p-adic valuation, and $\mu$ a 
normalized Haar 
measure on ${\mathfrak G}(\Q_p)$. Further, \(\rho:G \rightarrow \fG\) is a rational representation defined over $\Q_p$.
\end{definition} 
The above two definitions \ref{Intro-Def1} and \ref{Intro-Def2} can be linked in the local case. For example, in the ${\rm GL}$-case, i.e., when $\mathfrak G={\bf GL}_n$,
and $\rho$ is the natural representation, 
$Z_{{\bf GL}_n(\Q_p),\rho}(s)=\zeta_{\Z_p^{n}}(s)=\zeta_p(s)\zeta_p(s-1)\cdots \zeta_p(s-n+1)$, so that $\zeta_{\Z^n}(s)=\prod_{p}Z_{{\bf GL}_n(\Q_p),\rho}(s)$ between the generating 
series and the Euler product of $p$-adic zeta functions of ${\bf GL}_n$. 
Tamagawa studied the ${\rm GL}$-case in detail and included the situation when ${\mathfrak G}$ is the multiplicative group of a division algebra 
over a number field (\cite{Tamagawa1963}). 

In what follows, the $p$-adic zeta function is considered only for ${\mathfrak G}:=\{g\in {\bf GL}_{n}\mid {}^t g S g=\mu(g)\,S\,(\exists \mu(g)\in {\bf GL}_1\}$ with $\rho$ being the standard inclusion ${\mathfrak G}\hookrightarrow {\bf GL}_n$; this case is closely related to our maximal lattice counting series $\zeta_{\beta,\Z^n}(s)$. However, we note, once again, that we do not use the $p$-adic zeta function except to compare results obtained earlier in this context.

\subsection{Counting problem of the number of maximal lattices}\label{sec:Intro-1}
We work in the  general setting of 
maximal lattices. To study the analytic behavior of $\sum_{m<X} {\rm N}_{\beta}(m)$ as 
$X \rightarrow \infty$, as is commonly done for 
Dirichlet series (for example in the ${\rm GL}$-case), the following issues are addressed for \eqref{maximalzeta}: 
\begin{itemize}
\item[(a)] Determine a region ${\mathscr D}$ where the series converges absolutely. 
\item[(b)] Construct a meromorphic continuation beyond ${\mathscr D}$ to find the rightmost pole $s_0$ if any. 
\item[(c)] Find the leading Laurent coefficient at $s_0$ explicitly. 
\end{itemize}
In \cite[\S6]{dSW2008}, du Sautoy and Woodward conduct a study of the group zeta-function 
 defined as the Euler product 
$Z_{{\mathfrak G}, \rho}(s):=\prod_{p}Z_{{\mathfrak G}(\Q_p),\rho}(s)$ of
the $p$-adic zeta-functions $Z_{{\mathfrak G}(\Q_p),\rho}(s)$ for   the following $S$ :  
\begin{align}
&\left[\begin{smallmatrix} 0 & J_\ell \\ -J_\ell & 0 \end{smallmatrix}\right] \quad \text{(${\rm C}_\ell$-type)}, \qquad \left[\begin{smallmatrix} 0 & J_\ell \\J_\ell & 0 \end{smallmatrix}\right] \quad \text{(${\rm D}_\ell$-type)}, \qquad \left[\begin{smallmatrix} 0 & 0 & J_\ell \\ 0 & 2 & 0 \\ J_\ell & 0 & 0 \end{smallmatrix}\right] \quad \text{(${\rm B}_\ell$-type)}
\label{labelCDB}
\end{align} with $J_{\ell}= \left[\begin{smallmatrix}
    0 & & 1 \\
    & \iddots &  \\
    1 & & 0
\end{smallmatrix}\right]$. 
For each $p$, the explicit formula of the form 
\begin{align}
Z_{{\mathfrak G}(\Q_p), \rho}(s)=\frac{P_{{\mathfrak G}}(p,p^{-s})}{Q_{{\mathfrak G}}(p,p^{-s})}
\label{dSW-f}
\end{align}
can be found  in (\cite{I1989}, \cite{dSL1996}), where $Q_{{\mathfrak G}}(X,Y)$ and $P_{\mathfrak G}(X,Y)$ are two variable
polynomials with integer coefficients described in terms of the root 
system of ${\mathfrak G}$. With the help of the explicit formula, the above authors studied the analytic issues 
listed above and found an abscissa of absolute convergence and the natural boundary. 
In fact, \eqref{dSW-f} is obtained by linking $Z_{{\mathfrak G}
(\Q_p),\rho}(s)$ to Igusa's $p$-adic cone zeta-integrals. 
Although the denominator $Q_{{\mathfrak G}}(X,Y)$ is simple, the 
available formula of the numerator $P_{{\mathfrak G}}(X,Y)$ 
(\cite[Pg.156]{dSW2008}) is expressed as an alternating sum over the 
Weyl group of the root system.

As we shall see in Proposition \ref{ML-L5}, the group-zeta function 
$Z_{{\mathfrak G},\rho}(s)$ is the same as our 
$\zeta_{\beta,\Z^n}(s)$, when $S$ is of type ${\rm C}_\ell$ or ${\rm D}_\ell$, i.e., $S=\left[\begin{smallmatrix} 0 & \varepsilon J_\ell \\ J_\ell & 0 
\end{smallmatrix}\right]\,(\varepsilon \in \{\pm\})$. Our aim is to 
use a different method to provide 
a simple and  uniform expression of $P_{{\mathfrak G}}(X,Y)$ 
which enables us to express the Laurent coefficient in (c) concisely.
In what follows we shall obtain an explicit expression of 
$\zeta_{\beta,\Z^n}(s)$ for a general $S$ including forms listed in \eqref{labelCDB} through its Euler $p$-factors, without invoking 
the theory of Igusa's $p$-adic zeta-integral. 
Instead we  use the theory of local Hecke series for $\fG(\Q_p)$ on the lines of Tamagawa and Satake and apply the
theory developed by Andrianov and Hina-Sugano (\S\ref{sec:HeckeSeries} for detail). As far as the authors know, such an 
approach has not been taken earlier to study subgroup growth zeta functions for a lattice. 

 We would like to mention that our result is new for the type 
 ${\rm B}_\ell$-case and more generally for a case when $S$ is
 symmetric and there exists a prime number $p$ over which the matrix
 $S$ is not equivalent to the split form $\left[\begin{smallmatrix} 0 
 & J_\ell \\ J_\ell & 0 \end{smallmatrix}\right] $. For such a $p$,
 the natural action of the group ${\mathfrak G}(\Q_p)$ 
on the set of maximal lattices of $\Q_p^n$ is {\it not transitive} but has 2 orbits. Thus, the local $p$-factor of the counting series 
$\zeta_{\beta,\Z^n}(s)$ is expressed as a sum of both $Z_{{\mathfrak G}
(\Q_p),\rho}(s)$, which is indeed the $p$-adic zeta-function 
treated in \cite[\S6]{dSW2008}, but also has a second summand $p^{Bs} 
Z_{{\mathfrak G}(\Q_p),\rho}^{(1)}(s)$ with the constant $B\in \frac{1}
{2}\Z$ specified in Table \eqref{TableAB}, with 
$Z^{(1)}_{{\mathfrak G}(\Q_p),\rho}(s)$ being a variant of the $p$-adic
zeta-function as in Lemma \ref{Remarkf2}. 
We calculate not only  $Z_{{\mathfrak G}(\Q_p),\rho}(s)$ but also 
$Z^{(1)}_{{\mathfrak G}(\Q_p),\rho}(s)$ explicitly 
(Theorem \ref{IdxFtSrf2-T1}) by 
extending an argument of Hina-Sugano \cite{HS1983}. The explicit 
formula of the last function is new. As an example, this result can be 
applied to all 
$p$ for {\it symmetric matrices of odd degree} including the type ${\rm B}_\ell$-case in \eqref{labelCDB}.
 Another new example  is {\it the 
Witt tower over a norm form of a quadratic field}.
This case exhibits an interesting new 
phenomenon that the zeta-function $\zeta_{\beta,\Z^n}(s)$ is meromorphic
on a certain half-plane only after being squared. For a
general statement, we refer the reader to Theorem \ref{Thm2(EvenGO)}. 

\subsection{Description of main results}\label{sec:Intor-2}
In this section, rigorous statements are made only when $S$ is of type ${\rm C}_\ell$ or of type ${\rm D}_\ell$ in \eqref{labelCDB}. This is only for the sake of simplicity. Note that in these two cases the notions of
$\beta$-polarized and maximal lattices are the same. Before stating the results we need a polynomial, which already occurred, though not explicitly, in \cite{HS1983}; here we resurrect it with a more explicit description and name it the ASH (Andrianov-Hina-Sugano) polynomial.
 In the definition, the following notations are used. For $\ell\geq 2$, let $S_\ell$ be the symmetric group of $\ell$ letters $\{1,\dots,\ell\}$. For $\pi \in 
S_\ell$, let ${\rm I}(\pi)$ denote the set of all pairs $(i,j)$ of the 
letters such that $i<j$ and $\pi(i)>\pi(j)$, and ${\rm inv}(\pi):=\#{\rm 
I}(\pi)$. Set $D(\pi):=\{i\in \{1,\dots,\ell\}\mid (i,i+1)\in {\rm I}
(\pi)\}$.

\begin{definition}\label{Intro-Def3}
For $\ell \geq 1$, the ASH polynomial $W_\ell^{(\varepsilon)}(X,T)\in \Z[X,T]$ for $\varepsilon\in \{\pm 1\}$ is defined as 
$$
W_{\ell}^{(\varepsilon)} (X,T):=1+\sum_{d=1}^{\ell-1}\Bigl(\sum^{d\beta_{\ell,d}^{(\varepsilon)}}_{j=d\alpha_{\ell
,d}^{(\varepsilon)}} \Gamma_{\ell,d}^{(\varepsilon)}(j)X^{j}\Bigr)\,T^{d}
$$
with  
$$
\Gamma_{\ell,d}^{(\varepsilon)}(j):=\#\Bigl\{\pi \in S_\ell\mid d=\#D(\pi), \,j=\tfrac{\ell(\ell-\varepsilon)}{2}d-\sum_{\nu \in D(\pi)} \tfrac{\nu(\nu-\varepsilon)}{2}-{\rm inv}(\pi)\Bigr\},
$$
and, for $\ell \geq 2$ and $d\in [1,\ell-1]$,  
\begin{align*}
\alpha_{\ell,d}^{(\varepsilon)}&:=\tfrac{1}{2}\ell(d-1)-\tfrac{1}{6}(d+1)(d+\tfrac{3\varepsilon-5}{2}), \\
\beta_{\ell,d}^{(\varepsilon)}&:=\tfrac{1}{2}\ell(\ell-\varepsilon)-\tfrac{1}{6}(d+1)(d-\tfrac{3\varepsilon-7}{2}).
\end{align*}
\end{definition}
Note that $d\beta_{\ell,d}^{(\varepsilon)}$ and $d\alpha_{\ell,d}^{(\varepsilon)}$ are non-negative integers. 
\begin{theorem}\label{MAINTHM1}
The ASH polynomial satisfies the functional equation
$$
W_\ell^{(\varepsilon)}(X^{-1},T^{-1})=X^{\frac{\ell(\ell-1)}{2}-\frac{1}{3}(\ell-1)\ell(\ell+\frac{1-3\varepsilon}{4}))}\,T^{1-\ell}\,W_\ell^{(\varepsilon)}(X,T).
$$
For $j=d\alpha_{\ell,d}^{(\varepsilon)}$ and for $j=d\beta_{\ell,d}^{(\varepsilon)}$, we have that $\Gamma_{\ell,d}^{(\varepsilon)}(j)=1$
\end{theorem}

\begin{remark} Subsequently, to describe our results for a general maximal lattices, we will introduce the total ASH polynomial in \eqref{WellBtotal-Def}.
\end{remark}

We now define :

\begin{align*}
A_\ell^{(\varepsilon)}:=\prod_{r=1}^{\ell}\zeta\left(\tfrac{r(r-\varepsilon )}{2}+1\right) \times \prod_{p : \text{primes}}W^{(\varepsilon)}_\ell(p, p^{-\frac{\ell(\ell-\varepsilon)}{2}-\varepsilon }). 
\end{align*}

Now we present the following uniform description of 
$\zeta_{\beta,\Z^n}(s)$ to cover both cases of type ${\rm C}_\ell$ and ${\rm D}_\ell$ by letting $\varepsilon=1$ or $\varepsilon=-1$ respectively in the matrix $S_{\varepsilon}=\left[\begin{smallmatrix} 0 & J_\ell \\
\varepsilon J_{\ell} & 0 \end{smallmatrix} \right]$.

\begin{theorem}\label{MAINTHM2}
Let $S_{\varepsilon}=\left[\begin{smallmatrix} 0 & J_\ell \\
\varepsilon J_{\ell} & 0 \end{smallmatrix} \right]$ with $\varepsilon\in \{\pm 1\}$. Then, 
\begin{align}
\zeta_{\beta,\Z^{2\ell}}(s)&=\prod_{r=0}^{\ell}\zeta\left({\ell s}+\tfrac{r(r-2\ell-\varepsilon)}{2} \right)\times \prod_{p:{\rm primes}}W_\ell^{(\varepsilon)}(p, p^{-{\ell s}
})
 \label{MAINTHM2-f}
\end{align}
when $\Re(s)$ is large and admits a meromorphic continuation on some half-plane; its right-most pole is at $s_0=\frac{1}{\ell}(\frac{\ell(\ell-\varepsilon)}{2}+1)$. The Euler product in $A_\ell^{(\varepsilon)}$ converges and $A_{\ell}^{(\varepsilon)}>0$. If $\varepsilon=-1$, $\sigma_n$ is a simple pole with $
{\rm Res}_{s=s_0}\zeta_{\beta,\Z^{2\ell}}(s)=A_\ell^{(-1)}.
$
If $\varepsilon=+1$, then $s_0$ is a double pole with 
$\zeta_{\beta,\Z^{2\ell}}(s)= \tfrac{1}{(s-s_0)^{2}}\,\Bigl\{A_\ell^{(+1)}
+{\mathcal O}(1) \Bigr\}. 
$
\end{theorem}

We write out  explicit expressions of $\zeta_{\beta,\Z^{2\ell}}(s)$ for 
$\ell\le 4$.
For $S$ of type ${\rm B}_\ell$ and for a general $S$, we still obtain an explicit formula for $\zeta_{\beta,\Z^n}(s)$; though it is not as neat.

In the setting of Theorem \ref{MAINTHM2}, 
the series $\zeta_{\beta,\Z^n}(s)$ is shown to be the same as the group zeta function 
studied in \cite[\S6]{dSW2008}. Moreover, $Q_{{\mathfrak G}}(p,p^{-z})\,(z=ns)$ 
(see \cite[Pg.156 and Table 6.1]{dSW2008}) considered by earlier 
authors after being multiplied by $\prod_{j=0}^{\ell-\frac{\varepsilon+3}{2}}
(1+p^{\frac{j(2\ell-j+\varepsilon)}{2}-\frac{z}{2}})^{-1}$, coincides with the inverse of the Euler $p$-factor of the Riemann zeta factors in \eqref{MAINTHM2-f}. In this way we end up with the equality
\begin{align}
P_{{\mathfrak G}}(p,p^{-z})=W_\ell^{(\varepsilon)}(p,p^{-z}) \times \prod_{j=0}^{\ell-\frac{\varepsilon+3}{2}}(1+p^{\frac{j(2\ell-j+\varepsilon)}{2}-{z}}).
\label{Intro-WP}
\end{align}
Thus, in our setting, \cite[Corollaries 6.4, 6.11 and 6.13]
{dSW2008} yield the following :
\begin{theorem}[\cite{dSW2008}] \label{dSW-Intro}
The abscissa of convergence of the Euler product 
$\prod_{p}W_\ell^{(\varepsilon)}(p,p^{-{z}})$ is 
$\frac{\ell(\ell-\varepsilon)}{2}$. The natural boundary of 
$\zeta_{\beta,\Z^{2\ell}}(s/\ell)$ for 
$S=\left[\begin{smallmatrix} 0 & J_\ell \\
\varepsilon J_{\ell} & 0 \end{smallmatrix} \right]$ with $\varepsilon\in 
\{\pm 1\}$ is equal to $\beta_{\ell,1}^{(-1)}=\frac{\ell(\ell+1)}{2}-2$ if 
$\varepsilon=-1$, and to $\beta_{\ell,1}^{(+1)}=\frac{\ell(\ell-1)}{2}-1$ 
if $\varepsilon=+1$.  
\end{theorem}
Our description of $P_{{\mathfrak G},\rho}(X,Y)$ as in \eqref{Intro-WP} simplifies the arguments in 
\cite[\S6]{dSW2008}. However, 
for type ${\rm B}_\ell$-case, 
the counting series $\zeta_{S,\Z^n}(s)$ is not the group zeta-function of ${\mathfrak G}$ studied in \cite[\S6]{dSW2008}. 

It is worth noting that the ASH polynomial $W_\ell^{(\varepsilon)}(p,T)\in 
\Z[T]$ is related to the $p$-Euler factor $f(s_1,\dots,s_{\ell-1})$ of a 
multi-variable zeta function in \cite[Theorem 3.1]{Petrogradsky2007} with the equality 
\begin{align}
W_\ell^{(\varepsilon)}(p,p^{-z})=f\left(z-\tfrac{(\ell-1)(\ell-\varepsilon-1)}
{2}, s_2,\dots, s_\ell\right), \quad s_j:=\ell-j+\tfrac{1-\varepsilon}{2}\,
(j\in [2,\ell])\quad z\in \C,\label{Intro-f1}
\end{align}
Moreover, for type ${\rm C}_\ell$-case, we have the following relation 
between our counting-series  $\zeta_{\beta,\Z^{n}}(s)$ and a one-variable specialization of the Petrogradsky's multivariable Dirichlet 
series $\zeta_{\Z^\ell}(s_1,\dots,s_{\ell})$:  
\begin{align}
\zeta_{\beta,\Z^n}(z/\ell)=\zeta(z)\,\zeta_{\Z^{\ell}}\left(z-
\tfrac{\ell(\ell-1)}{2},s_2,\dots,s_{\ell-1},
, s_{\ell}\right), \quad s_j:=\ell-j+\tfrac{1-\varepsilon}{2}\,(j\in 
[2,\ell]). 
\label{Intro-f2}
\end{align}
 By comparing the coefficients of the Dirichlet series 
in \eqref{Intro-f2}, we obtain an exact formula representing the number of
$\beta$-polarized lattices in $\Z^{2\ell}$ of index $m^\ell$ in terms of 
the counting function of sub-lattices in $\Z^\ell$ with a given type of 
elementary divisors:
\begin{align*}
{\rm N}_{\beta}(m^\ell)=\sum_{d\mid m}d^{\frac{\ell(\ell-1)}{2}}\sum_{(\delta_1,\dots,\delta_{\ell-1})}\nu_{\Z^\ell}(d,\delta)\,\prod_{j=1}^{\ell-1} \delta_j^{j-\ell},
\end{align*}
where $\delta=(\delta_1,\dots,\delta_{\ell-1})$ runs over all $(\ell-1)$-tuples of divisors of $d$ such that $\delta_{j}\mid \delta_{j-1}\,(j\in [2,\ell])$, and $\nu_{\Z^\ell}(d,\delta)$ denotes the number of lattices $L\subset \Z^{\ell}$ such that $\Z^\ell/L\cong \bigoplus_{j=0}^{\ell-1}\Z/\delta_j\Z$ with $\delta_0=d$. 

From Theorem \ref{MAINTHM2}, by the simple use of Tauberian theorems, 
we get asymptotic results for the number of $\beta$-polarized lattices in $\Z^{n}$ as :

\begin{corollary}\label{MAINTHM3} Let $S$ and $\varepsilon \in \{\pm 1\}$ be as in Theorem \ref{MAINTHM2}. Set $\sigma:=\frac{1}{\ell}(\frac{\ell(\ell-\varepsilon)}{2}+1)$. Then, 
\begin{align*}
{\rm N}_{\beta}(X) \sim \sigma^{-1} A_{\ell}^{(\varepsilon)} X^{\sigma}(\log X)^{\frac{\varepsilon+1}{2}}, \quad X\rightarrow \infty.
\end{align*}  
\end{corollary}

\begin{remark} For $\ell=3$ and $\varepsilon=-1$, this agrees with the
main term $c_1X^{7/3}$, $c_1 = 2.830...$ in the asymptotic result corresponding to $\Z^{6}$ \textup{(\cite{BSP2007})}.  
\end{remark}

For the more general context of maximal lattices, see Corollaries \ref{Asymptotic(EvenNS)}, \ref{Asymptotic(EvenSP)} and \ref{AsymptoticOdd}.

\subsection{Hecke series}\label{sec:HeckeSeries}
The local Hecke series was originally defined as a multi-variable formal 
power series in an indeterminate $T$ and other independent variables $X_1,\dots,X_\ell$ for 
the Satake parameters of the class $1$ representations of ${\mathfrak G}(\Q_p)$. 
The rationality of the local Hecke series of a symplectic group was a 
conjecture due to Shimura (\cite{Sh1963}) and the explicit form of its 
denominator was a conjecture due to Satake (\cite{S1963}), both of which were
solved by Andrianov (\cite{A1968}, \cite{A1969}, \cite{A1970}); these studies
are later used by Hina-Sugano \cite{HS1983} to obtain similar results for 
general classical groups. Actually, in their work, the denominator of 
the local Hecke series is completely worked out and the numerator has a 
rather complicated explicit description . 
The local index function series, which is related to the Euler factors of our $\zeta_{S,\Z^{n}}(s)$, is the specialization of $X_1,\dots,X_\ell$ in the 
local Hecke series at the Satake parameter of the trivial representation of 
${\mathfrak G}(\Q_p)$. Starting 
from Hina-Sugano's formula of the local index function series, we work
out the 
formula further in terms of the $q$-multinomial 
coefficients.  
As mentioned above, for the odd orthogonal case or the non-split even orthogonal case, we need to handle another kind of index-function series that is not calculated in 
\cite{HS1983}. We include a detailed account for this in \S\ref{sec:IdxFtSrf2}.   

Finally, we mention some related work on the
numerator of the Hecke series. The symplectic cases 
of ranks 1 and 2 already occur in the works of Shimura and Andrianov and 
specialize to give $\zeta(s)\zeta(s-1)$ 
and $\zeta(2s)\zeta(2s-2)\zeta(2s-3)\prod_p({1+p^{1-2s}})$ respectively. 
For the symplectic groups of rank $3$ and $4$, the numerator of the Hecke series as well as its specialization is explicitly worked out using computers in \cite{PV2007},\cite{Vankov2011} and a functional equation is obtained. The polynomials $W_{\ell}^{(-1)}(p,Y)$ for $\ell=3,4$, explicitly worked out in Example \ref{ExampleSymplectic} (3),  are indeed factors of the specialization of the numerator of the full Hecke series.
A conjecture on the functional equation for the symplectic groups of general rank posed therein  was solved recently in \cite{AMV2024}.

\subsection{Detailed organization of paper}\label{sec:Intro-4}
In \S\ref{sec:ORTH}, we study local index series for the orthogonal similitude group over local fields. In \S\ref{sec:LTTQS}, we collect basics on the maximal lattices as described in \cite{S1963} and \cite{Sh2010}. For the convenience of readers we include proofs for some specially relevant results. In \S\ref{sec:LCLIDXS} we introduce the index-function series as a counting series of maximal lattices relating it to the $p$-adic zeta function (Proposition \ref{ML-L5}). We recall the result of Hina-Sugano on the explicit formula of the index-function series, then introduce permutation descent polynomials $w_{\ell,K}(q)$ 
to express the above result in a concise form in \eqref{MainFormula}. In \S\ref{EFPell-P1}, we provide a simple proof of the functional equation (Theorem 
\ref{MAINTHM1}) by using a result due to Stanley(\cite{S1976}). Our results on the 
numerator of Hina-Sugano's local-index series is stated uniformly in Theorem 
\ref{EFPell-P3}. The technical contribution of the paper is in \S\ref{sec:IdxFtSrf2} where we study an index-function series that is not treated by Hina-Sugano. We introduce  another main 
polynomial $W_{\ell,A,B}^{\rm total}(X,Y)$ together with an auxiliary one $U_{\ell,B}^{\bullet}(X,u;T)$ that depend on the local invariants $(\ell,A,B)$ of the maximal lattice and specialize to the ASH case. To handle the counting zeta-function for the ${\rm GO}(2\ell+1)$-case and non-split ${\rm GO}(2\ell)$-case, we need results obtained in this subsection, in
particular Theorems \ref{IdxFtSrf2-T1} and \ref{IdxFtSrf2-T2}. Having local inputs 
from 
the previous section, 
we study Euler products related to the index-counting series of maximal lattices in a quadratic space over $\Q$ in \S\ref{sec:GlobalGO}. The global results including Theorem \ref{MAINTHM2} and Corollary \ref{MAINTHM3} 
 are proven in \S\ref{sec:GlobalCTN}. Since the theory of the symplectic case is 
 parallel to the ${\rm GO}(n)$-case, it is treated only briefly in \S\ref{sec:Symp}.   

\medskip
\noindent
{\bf Notation} : For integers $n,m$, set $[m,n]:=\{x\in \Z\mid m\leq 
x\leq n\}$. Thus, $[m,n]=\emptyset$ if $m>n$. Any non-empty subset of
$\Z$ of the form $[a,b]$ is called an interval of $\Z$. For $n\in 
\Z_{>0}$, let $1_n$ denote the identity matrix of degree $n$, and 
$J_n:=(\delta_{i,n-j+1})_{1\leq i, j\leq n}$. The set of all prime numbers is denoted by ${\mathbb P}$. 
 
\section{The \texorpdfstring{${\rm GO}(n)$}{}-case: the local theory}\label{sec:ORTH}

In this section, we fix a prime number $p$ and let $F$ be the $p$-adic field $\Q_p$. We denote by $O$  the ring $\Z_p$ of $p$-adic integers. Let ${\rm ord}_p:\Q_p^\times \rightarrow \Z$ denote the additive valuation so that $|\alpha|_p=p^{-{\rm ord}_p\alpha}$ for 
$\alpha \in \Q_p^\times$.  

\subsection{Lattices in a quadratic space over \texorpdfstring{$p$}
{p}-adic fields} \label{sec:LTTQS}

The notions of quadratic forms used here are standard. 
We include a few for the sake of completion and refer the reader to classical texts like \cite[Chapitre IV]{SerreBook} and \cite[\S22]{Sh2010} 
for details. Let $V\cong F^n$ be a $F$-vector space of 
dimension $n$. Let $\beta:V\times V\rightarrow F$ be a non-degenerate 
symmetric $F$-bilinear form on $V$, so that 
$$Q(v)=\tfrac{1}{2}\beta(v,v)\qquad (v\in V)$$ is a quadratic form on
$V$ related to $\beta$ as  
$$Q(v+w)=Q(v)+Q(w)+\beta(v,w).$$
There exist systems of vectors 
$\{v_j\}_{j=1}^{\ell}, \{v^{*}\}_{j=1}^{\ell}$ satisfying
$\beta(v_i,v_j^*)=\delta_{ij}$ such that $T=\sum_{j=1}^{\ell} F v_j$
and 
$T^*=\sum_{j=1}^{\ell}F v_j^*$ are totally isotropic subspaces of maximal 
dimension. Thus $W:=T^\bot \cap (T^*)^{\bot}$ is an $F$-anisotropic subspace, 
i.e., $w\in W$, $Q(w)=0$ if and only if $w=0$. We have the direct 
sum decomposition 
 \begin{align}
V=T\oplus T^*\oplus W=\sum_{j=1}^{\ell}(Fv_j+F v_j^*)+W.
\label{WittDec}
 \end{align} 

 The integer $\ell:=\dim_F(T)=\dim_F(T^*)$ is called the Witt index of $(V,Q)$, which is an isometry invariant of the quadratic space(\cite[Lemma 22.4]{Sh2010}). If $n=2\ell$, i.e., $\dim_F(W)=0$, the $\beta$ (or $Q$) is said to be 
 $F$-split. 
 Let $D$ be the quaternion division algebra over $F$, which is unique
 up to isomorphism,
 and ${\rm tr}_{D}:F\rightarrow F$ (resp. ${\rm N}_{D}:D \rightarrow
 F$) the reduced trace map (resp. the reduced norm map). 
 Let $D^0:=\{\xi \in D\mid {\rm tr}_{D}
 (\xi)=0\}$; then $\dim_{F}(D)=4$ and $\dim_{F}(D^0)=3$. 
 
 \begin{lemma}\label{ANISO-L} We have exactly one of the following cases: 
 \begin{itemize}
\item $n_0=0$, i.e., $W=(0)$.
\item $n_0=1$, and there exists $u\in F^\times$ such that $(W,Q|_{W})\cong (F,u q_1)$, where $q_1(x)=x^2\,(x\in F)$. 
\item $n_0=2$, and there exists a quadratic field extension $E/F$ and $u\in F^\times$, such that $(W,Q|_W) \cong(E, u{\rm N}_{E/\Q})$.  
\item $n_0=3$, and there exists $u\in F^\times$ such that $(W,Q|_{W})\cong (D^0, u{\rm N}_{D})$.
\item $n_0=4$, and there exists $u\in F^\times$ such that $(W,Q|_W)\cong (D, u{\rm N}_D)$.
\end{itemize} 
 \end{lemma}
\begin{proof} {\it cf}. \cite[\S25]{Sh2010}. \end{proof}
From this lemma, we have $\dim_F(W)\in \{0,1,2,3,4\}$ which means that 
$n\geq 5$ implies $\ell\geq 1$. 

  A subset $L\subset V$
 is called an $O$-lattice, if it is a finitely generated sub $O$-module
 of $V$ of rank $n=\dim_{F}(V)$. Let ${\mathfrak L}_V$ denote the set 
 of all the $O$-lattices in $V$.
 For $L \in \fL_V$, set
$$
\widehat L:=\{v\in V\mid \beta(v,L)\subset O\},
$$
which is an element of $\fL_V$ called the dual lattice of $L$.
The operation $L\mapsto \widehat L$ is an involution on $\fL_V$. 
We now recall some special lattices relevant to our work.
The first two are easier to state.

\begin{definition} 
When $\widehat L=L$, we say that $L$ is a {\underline { unimodular lattice}}. If a unimodular 
lattice $L$ further satisfies the condition 
$\beta(v,v)\in 2O\,(\forall v\in L)$, then $L$ is called an 
{\underline {even unimodular lattice}}.  
\end{definition}

\begin{definition} 
An $O$-lattice $L\subset V$ is called a {\underline{$\beta$-polarized lattice}} 
(resp. {\underline{even $\beta$-polarized}}) if $L$ is unimodular (resp. even 
unimodular) with respect to $c\beta$ for some $c\in F^\times$.
\end{definition}

Explicitly, $L$ is $\beta$-polarized (resp. even $\beta$-polarized) if 
$\widehat L=cL$ for some $c\in F^\times$ (resp. $\widehat L=cL$ and
$\beta(v,v)\in 2c^{-1}O$ for all $v\in L$ for some $c\in F^\times$).

\begin{lemma}\label{ML-L0}
Suppose $p\not=2$. Then, $L\in \fL_V$ is even $\beta$-polarized if and 
only if it is $\beta$-polarized. 
\end{lemma}
\begin{proof} It suffices to show that if $p\not=2$, $\widehat L=cL$ 
implies $\beta(v,v)\in 2c^{-1}O\,(\forall v\in L)$. Let $v\in L$; then,
since $cv\in cL=\widehat L$ and $O=2O$ because $p\not=2$, we have
\begin{align}
c\beta(v,v)=\beta(cv, v)\in \beta(\widehat L, L)\subset O=2O. 
\notag
\end{align}
\end{proof}

\subsubsection {Maximal lattices} \label{sec:MaxLattice}
To introduce this  we need the notion of the norm of a lattice (\cite{S1963}).  
For an $O$-lattice 
$L\subset V$ (of rank $n$), its norm $\mu(L)$ is the fractional  $O$-ideal in
$F$ generated by the values $Q(v)\,(v\in L)$. The following properties are 
evident: 
\begin{itemize}
\item[(a)] $L,L'\in \fL_V,\,L\subset L' \quad \Longrightarrow \quad
\mu(L)\subset \mu(L')$, 
\item[(b)] $\mu(cL)=c^2 \mu(L), \quad L\in \fL_V,\,c\in F^\times$. 
\end{itemize}
Let $\mathfrak a$ be an invertible fractional ideal of $F$. 
\begin{definition} An $O$-lattice $L$ 
is called a {\underline{maximal $\fa$-(integral) lattice}} if $\mu(L)=\mathfrak a$, and it is 
maximal among all the $O$-lattices in $V$ of norm ${\mathfrak a}$. A lattice 
is called  maximal if it is a maximal $\fa$-lattice for some fractional ideal $\fa$. 
\end{definition}
Note that \cite[\S29]{Sh2010} only deals with maximal $O$-integral lattices. 
\begin{lemma} \label{LM-L00}
 For any $L\in \fL_V$, there exists a $\mu(L)$-maximal lattice containing $L$.
 \end{lemma}
 \begin{proof} {\it cf}. \cite[Lemma 29-2]{Sh2010}
 .
 \end{proof}
For a maximal $\fa$-integral lattice $L$. There exists a Witt decomposition \eqref{WittDec} of $V$ such that 
$$ 
L=\sum_{j=1}^{\ell}(Ov_j+\fa v_j^{*})+L_0
$$
with $L_0:=\{w\in W\mid Q(w)\in \fa\}$. (It is known that $L_0$ is a
maximal lattice in $W$ such that $\mu(L_0)$ is $\fa$ or $p\fa$.) 
For $v=\sum_{j=1}^\ell(x_jv_j+y_j v_j^*)+w$ with $x_j\in O,y_j \in \fa, w\in L_0)$, we have $Q(v)=\sum_{j=1}^{\ell}x_j y_j+Q(w)$. The dual of $L$ is given as  
\begin{align}
\widehat L=\sum_{j=1}^{\ell}(\fa^{-1}v_j+O v_j^{*})+\widehat {L_0},
\label{WittLat}
\end{align}
where $\widehat{L_0}$ is the dual lattice of $L_0\in \fL_{W}$ with respect to $\beta_W=\beta|W\times W$. Therefore, 
\begin{lemma}\label{ML-L1}
An $O$-lattice $L$ is even $\beta$-polarized if and only if $\widehat L_0=\mu(L)^{-1}L_0$; if this is the 
case, then $\widehat L=\mu(L)^{-1}L$. 
In particular, if $(V,Q)$ is $F$-split, then any maximal lattice is even 
$\beta$-polarized. 
\end{lemma}
We have 
\begin{align}
L\subset \mu(L)\,\widehat L, \quad L\in \fL_V.
\label{ML-f0}
\end{align}
Indeed, if $\mu(L)=p^{m}O$, then for any $v,u\in L$, 
$$
\beta(v,u)=Q(v+u)-Q(v)-Q(u) \in \mu(L)=p^{m}O
$$
because $Q(u+v),Q(u),Q(v)$ all belong to $\mu(L)$; hence, 
$\beta(p^{-m}L,L)\subset O$, which is equivalent to 
$p^{-m}L\subset\widehat L$, i.e., $L\subset \mu(L)\,\widehat L$. 

\begin{lemma}\label{ML-L2}
Any even $\beta$-polarized lattice $L\in \fL_V$ is a maximal lattice and $\widehat L=\mu(L)^{-1}L$. 
\end{lemma}
\begin{proof} Suppose $L$ is even $\beta$-polarized, so that $\widehat L=cL$ and $Q(v)\in c^{-1}O\,(\forall v\in L)$; note that the latter condition is equivalent to $\mu(L)c\subset O$. These conditions, combined with \eqref{ML-f0}, yields $L\subset \mu(L)c L\subset OL=L$. Thus, $L=\mu(L)cL$. This implies $\mu(L)c=O$, or equivalently $\widehat L=\mu(L)^{-1}L$. It remains to show that $L$ is a maximal lattice. Let $L\subset M$ and $\mu(L)=\mu(M)$. then
$$\widehat M\subset \widehat L=\mu(L)^{-1}L=\mu(M)^{-1}L\subset \mu(M)^{-1}M.$$   
Hence $\mu(M)\widehat M\subset M$, which combined with \eqref{ML-f0} applied to $M$ implies $M=\mu(M) \widehat M$. Then, all the inclusions above become equalities; in particular $\widehat L=\widehat M$, or equivalently $L=M$. 
\end{proof}

\begin{lemma} \label{CSTMaxLatt} Suppose $\ell \geq 1$. Let $\alpha \in \Z$, and $\fc_j,\fd_j\,(1\leq j\leq \ell)$ be a set of fractional $O$-ideals in $F$ such that $\fd_j\fc_j=p^{\alpha}O\,(1\leq j \leq \ell)$. For any Witt decomposition \eqref{WittDec}, set $L_0^{(\alpha)}:=\{w\in W \mid Q(w)\in p^{\alpha}O\}$; then, the $O$-lattice $L:=\sum_{j=1}^{\ell}(\fc_j v_j+\fd_j v_j^*)+L^{(\alpha)}_0$ is a maximal $p^{\alpha}O$-integral lattice.     
\end{lemma}
\begin{proof} Indeed, $\mu(L)=p^{\alpha}O$ is easy to confirm due to $\ell\geq 1$. Then, by Lemma \ref{LM-L00}, we can find a maximal $p^{\alpha}O$-integral lattice $M$ containing $L$. As such, it has a Witt decomposition $M=\sum_{j=1}^{\ell}(O u_j+p^{\alpha+1}Ou^*_j)+M_0$ with $\beta(u_j,u_i)=\beta(u_j^*,u^*_i)=0$, $\beta(u_i,u_j^*)=\delta_{ij}$ and $M_0$ is an $O$-lattice in the orthogonal $U:=(\sum_{j=1}^{\ell}F u_j +F u_j^*)^\bot$ given as $M_0=\{u \in U\mid Q(u)\in p^{\alpha}O\}$. Set $T:=\sum_{j=1}^{\ell} Fv_j$ and $T^{*}:=\sum_{j=1}^{\ell}F v_j^*$, and $S:=\sum_{j=1}^{\ell}F u_j$, $S^{*}:=\sum_{j=1}^{\ell} F u_j^*$. Define a bijective $F$-endomorphism $g:T\oplus T^* \rightarrow S\oplus S^*$ by $g(v_j)=p^{\gamma_j} v_j,\,g(v_j^*)=p^{-\gamma_j}u_j^*\,(j \in [1,\ell])$, where $\gamma_j:={\rm ord}_p\fc_j$; then it preserves the bilinear form $\beta$, so that it can be extended to an isometry $g:V\rightarrow V$ by Witt's extension theorem. Note that $\mu(g)=1$. Since $W=(T\oplus T^*)^{\bot}$ and $U=(S\oplus S^*)^{\bot}$, we have $g(W)=U$ and $g(L_0^{(\alpha)})=M_0$. Noting $\fc_j\fd_j=p^{\alpha}O$, we get $g(L)=M$. Since $M$ is maximal, so is its isometric image $L=g^{-1}(M)$. Thus, $L=M$. 
\end{proof}

Let 
$$
\widetilde G:=\{g\in {\rm GL}_F(V)\mid (\exists \mu(g)\in F^\times)(\forall v\in V)( Q(g(v))=\mu(g)\,Q(v))\}={\rm GO}(V,Q)$$
be the orthogonal similitude group of $Q$. The character of similitude $\mu:\widetilde G\rightarrow F^\times$ is related to the determinant by the relation   
\begin{align}
(\det g)^2=\mu(g)^{n}, \quad g\in \widetilde G.
\label{ML-f1}
\end{align}
Let $G$ be the identity connected component of $\widetilde G$ with respect to the Zariski topology; $\widetilde G=G$ if $n$ is odd and $G$ coincides with the locus $\det(g)=\mu(g)^{n/2}$ if $n$ is even. Given a Witt decomposition \eqref{WittDec}, we write $\widetilde G_0$ for the orthogonal similitude group of $(W,Q|_{W})$ and $G_0$ the identity connected component of $\widetilde G_0$ and for $\mu_0:\widetilde G_0 \rightarrow F^\times$ the character of similitude. The group $G$ acts on the set $\fL_V$ by $G\times \fL_V\ni (g,L)\mapsto gL \in \fL_V$. For any $L \in \fL_V$ and $g\in G$, we easily have
$$ \mu(gL)=\mu(g)\,\mu(L), \quad \widehat{(gL)}=\mu(g)^{-1}\, g \widehat L.
$$
By using this, it is easy to show that the set of maximal lattices, say $\fL_\beta^{\rm {max}}$, and the set of even $\beta$-polarized lattices, say $\fL_\beta^{{\rm pl}}$, are preserved by the action of $G$.
Given a Witt decomposition \eqref{WittDec}, we define elements of $G$ as follows. For any $(t_1,\dots, t_\ell)\in (F^\times)^\ell$, $\lambda \in F^\times$ and $g_0\in G_0$ such that $\mu_0(g_0)=\lambda$, define an $F$-endomorphism $[t_1,\dots,t_\ell;g_0,\lambda]$ of $V$ by 
\begin{align}
[t_1,\dots,t_\ell;g_0,\lambda]:\begin{cases}
v_j &\longmapsto t_j v_j \quad (1\leq j \leq \ell),\\
v_j^{*} & \longmapsto \lambda t_j^{-1} v_j^* \quad (1\leq j \leq \ell), \\
w & \longmapsto g_0(w), \quad (w \in W).   
\end{cases}
 \label{ML-L3-f0}
\end{align}
Then, $[t_1,\dots,t_\ell;g_0,\lambda]\in G$ and 
$\mu([t_1,\dots,t_\ell;g_0,\lambda])=\mu_0(g_0)=\lambda$. When $n_{0}=0$,
i.e., $W=(0)$,  we define $[t_1,\dots,t_\ell;\lambda]$ for any 
$t_1,\dots,t_\ell,\lambda \in F^\times$ and suppress $g_0$.) 
  We now recall a result due to Satake \cite[$(E_2)$ \S9]{S1963}.
\begin{lemma}  \label{ML-LL3}
Let $L$ and $M$ be maximal lattices with $\alpha={\rm ord}_p \mu(L)$ and $\beta={\rm ord}_p \mu(M)$. Then we can find a Witt decomposition
\eqref{WittLat} of $L$ in such a way that 
\begin{align}
M=\sum_{j=1}^{\ell}(p^{\gamma_j}O v_j +p^{\beta-\gamma_j}O v_j^*)+M_0
\label{WittLat2}
\end{align}
with $M_0=\{w\in W \mid Q(w)\in p^{\beta}O\}$ and a set of integers $\gamma_j\,(j\in [1,\ell]$ such that  $\gamma_1\geq \dots \geq \gamma_\ell\geq \frac{1}{2}(\beta-\alpha)$.
\end{lemma}

\begin{lemma} \label{ML-L3}
Let $L,M\in \fL_\beta^{\rm max}$; then, there exists $g\in G$ such that $M=g(L)$ if and only if $\mu(M)\mu(L)^{-1}=\mu(g)O\,(\exists g\in G)$.
\end{lemma}
\begin{proof}
If $M=g(L)$ with $g\in G$, then $\mu(M)=\mu(g)\mu(L)$, so that 
$\mu(M)\mu(L)^{-1}=\mu(g)O$. Let us prove the converse. We can fix a
Witt decomposition \eqref{WittDec} such that \eqref{WittLat} and
\eqref{WittLat2} hold simultaneously to define elements \eqref{ML-L3-f0}.
Suppose 
$\mu(M)\mu(L)^{-1}=\mu(g)O$ for some $g\in G$. We have
$\mu(gL)=\mu(g)\mu(L)=\mu(M)$. Hence, to complete the proof, we may suppose
$\mu(L)=\mu(M)$ and find $h\in G$ such that $M=h(L)$. Define 
$h:=[p^{\gamma_1},\dots,p^{\gamma_\ell};{\rm Id}_W, 1]$. Then, 
$h\in \widetilde G$ and $\det h=1$, and, by \eqref{WittLat}, 
\eqref{WittLat2} and \eqref{ML-L3-f0}, we have $M=h(L)$.
\end{proof}
Define $f\in \Z$ by 
\begin{align}
f\Z={\rm ord}_{p}\mu(G).
\label{Def-fmuG}
\end{align}
By Lemma \ref{ML-L3}, $G$ acts on $\fL_\beta^{\rm max}$ transitively  if $f=1$.  

Choose a Witt decomposition \eqref{WittDec}; If $W=(0)$, set $(F^\times)_Q=F^\times$, otherwise, define $(F^\times)_Q$ to be the image of the similitude character $\mu_0$ of $G_0:={\rm GO}(W)^0$, which is a subgroup of $F^\times$ independent of the choice of \eqref{WittDec}.

\begin{lemma}\label{LemFeQ1}
We have ${\rm ord}_p \mu(G)={\rm ord}_p (F^\times)_{Q}$. 
Moreover, $(F^\times)_Q=F^\times$ if $n_0=0,4$, $(F^\times)_Q=(F^\times)^2$ if $n_0=1$ and $(F^\times)_Q={\rm N}_{E/F}(E^\times)$ for some quadratic field extension $E$ of $F$ if $n_0=2,3$. We have $f\in \{1,2\}$. 
\end{lemma}
\begin{proof} Let $P$ be a parabolic subgroup of $G$ whose Levi subgroups is formed by all the elements in \eqref{ML-L3-f0}. Then, we have the Iwasawa decomposition $G=PU$. Suppose $n_0>0$; due to $\mu([t_1,\dots,t_\ell;h,\lambda])=\mu_0(h)$, we have $\mu(P)=\mu_0(G_0)=(F^\times)_Q$. Since $\mu(U)\subset O^\times $, we have ${\rm ord}_p \mu(G)={\rm ord}_p \mu(P)={\rm ord}_p (F^\times)_Q$. Suppose $n_0=0$, then due to $\mu([t_1,\dots,t_\ell;\lambda]=\lambda$, we have $\mu(P)=F^\times$. Hence $\mu(G)=\mu(P)\mu(U)=F^\times$. 

We use Lemma \ref{ANISO-L} to describe $G_0$ and $\mu_0$. If $n=1$, then $G_0=F^\times$ and the similitude norm of $x\in F^\times$ is $x^2$. Hence $\mu(G)=(F^\times)^2$. If $n=2$, then $G_0={\rm GO}(E,u{\rm N}_{E/F})^0=E^{\times}$ and the similitude of $\alpha \in E^\times$ is ${\rm N}_{E/F}(\alpha)$. The norm-image ${\rm N}_{E/F}(E^\times)$ contains a prime element of $F$ if and only if $E/\Q$ is a ramified extension. If $n=4$, then $G_0={\rm GO}(D, u{\rm N}_{D})^0\cong (D^\times \times D^\times)/Z$ with $Z:=\{(z,z^{-1})\mid z\in F^\times\}$ with the action $(\alpha,\beta)\in (D^\times \times  D^\times)/Z$ on $x\in D$ being $(\alpha,\beta)x=\alpha x \beta^{-1}$; the similitude norm of $(\alpha,\beta)\in (D^\times \times D^\times)/Z$ is ${\rm N}_{D}(\alpha \beta)$. Hence $\mu_0(G_0)=F^\times$. If $n=3$, then $G_0$ is isomorphic to the stabilizer of ${\rm GO}(D,u{\rm N}_{D})^0$ of the vector $1\in D$, i.e, $G_0$ is the diagonal of $(D^\times \times D^\times)/Z$, and the similitude of $g=(\alpha,\alpha)\in G_0$ is $\mu_0(g)=({\rm N}_D(\alpha))^2$; thus, $\mu_0(G_0)=(F^\times)^2$. \end{proof}

\subsection{The local index series} \label{sec:LCLIDXS}
Fix $L\in \fL_\beta^{\rm max}$ once and for all. The stabilizer of $L$ in $G$, 
$$
U:=\{g\in G\mid gL=L\},
$$
is a maximal compact subgroup of $G$ ({\it cf}.\cite[\S 9.4, 9.5]{S1963}). 
For any integer $m\geq 0$, consider the set of lattices 
\begin{align}
\fL_{\beta,L}^{\rm {max}}(p^m):=
\{\Lambda \in \fL_\beta^{\rm {max}} \mid \Lambda\subset L, \, \mu(\Lambda)=p^{m}\mu(L)\}
. \label{fLnormPm}
\end{align}

Define a formal power series
$$
\boldsymbol\zeta_{\beta,L}^{(0)}(T):=\sum_{m=0}^{\infty} (\# \fL_{\beta,L}^{\rm max}(p^{fm}))\, T^{m} \quad (\in \Z[[T]])
$$
 for counting the number of maximal lattices contained in $L$ with 
 the index being a power of $p^f$. 
 The series coincides with the $p$-adic zeta function  ({\it cf}. \cite[Theorem 2]{BG2005}). 
 
 \begin{proposition} \label{ML-L5} 
 Let $dg$ denote the Haar measure on $G$ such that $\int_{U}d g=1$. Define 
$$G^{+}:=\{g\in G\mid gL\subset L\}.$$ Then, 
\begin{align*}
\int_{G^+}|\det(g)|_p^{s}\,d g=\boldsymbol\zeta_{\beta,L}^{(0)}(p^{-fns/2}), \quad \Re(s)\gg 0
\end{align*}
as an integral and a series are absolutely convergent for sufficiently large $\Re(s)$. 
 \end{proposition}
 \begin{proof} For $m\geq 0$, set
 $$
 X_L(m):=\{g\in G\mid gL \subset L,\, \mu(g)O=p^{fm}O\}.
 $$
If $g\in G^{+}$, then $\mu(g)\mu(L)=\mu(gL)\subset \mu(L)$, which implies that $\mu(g)\in O$ so that $\mu(g)O=p^{fm}O$ with $fm={\rm ord}_p \mu(g)\geq 0$. Hence, we have the disjoint decomposition
$$G^+=\bigsqcup_{m\geq 0} X_L({m}). 
$$ 
For $g\in X_L(m)$ we have $|\det g|_p=|\mu(g)|_p^{n/2}=p^{-fmn/2}$ by \eqref{ML-f1}. This allows us to write the integral as the sum
$$
\sum_{m=0}^\infty \int_{X_L(m)} d g \times  p^{-sfmn/2}. 
$$
The group $U$ acts on the set $X_L(m)$ by the right-multiplication. By Lemma \ref{ML-L3}, we have a bijection $X_L(m)/U \ni gU \longmapsto g(L)\in \fL_{\beta,L}^{\rm max}(p^{fm})$. Since ${\rm vol}(U)=1$ and since $X_L(m)$ is bi-$U$-invariant, we have 
$$
\int_{X_L(m)}d g=\#(U\backslash X_L(m))=\#(X_L(m)/U)=\# \fL_L^{\rm max}(p^{fm}).
$$
In this way, we get the equality 
$$
\int_{G^+}|\det g|_p^{s}\,d g=\sum_{m=0}^\infty (\# \fL_L^{\rm max}(p^{fm}))\,p^{-fmns/2}
$$
at least formally. The absolute convergence of the series 
$ \sum_{m=0}^{\infty} \#(X_L(m)/U)\,p^{-sfmn/2}$ for large $\Re(s)$ results from the explicit formula of $\#(X_L(m)/U)=\#(U\backslash X_L(m))$ given in \cite[Proposition 1]{HS1983}.     
\end{proof}
 \begin{corollary} \label{ML-L6} 
 Let $\Lambda \in \fL_{\beta,L}^{\rm max}(p^{fm})$. Then, $[L:\Lambda]=p^{fmn/2}$. We have
$$
\boldsymbol\zeta_{\beta, L}^{(0)}(p^{-fns/2})=\sum_{\Lambda \in \fL_L^{{\rm max},(0)}}[L:\Lambda]^{-s},
$$
where $\fL_{\beta,L}^{\rm max,(0)}:=\{\Lambda \in \fL_\beta^{\rm max}\mid \Lambda \subset L,\,{\rm ord}_{p}(\mu(\Lambda)\mu(L)^{-1})\in f\Z\}$.  
 \end{corollary}
\begin{proof} We may set $\Lambda=gL$ with $g\in G^+$. Then,
$$
[\Lambda:L]=|\det g|_p^{-1}=|\mu(g)^{n/2}|_p^{-1}=|p^{fmn/2}|^{-1}_p=p^{fmn/2}$$
as desired. 
\end{proof}
Due to the bijection $X_L(p^m)/U \cong \fL_{\beta,L}^{\rm max}(p^{fm})$ and the equality $\#(U\backslash X_L(m))=\#(X_L(m)/U)$ in the proof of Proposition \ref{ML-L5}, we have the expression $\boldsymbol{\zeta}_{\beta,L}^{(0)}(T)=\sum_{m=0}^{\infty}\#(U\backslash X_L(p^m))T^m$, which is exactly the index function series in \cite[Definition 1]{HS1983}. Let us recall a result by Hina-Sugano (\cite{HS1983}), which concerns an explicit description of the index function series; in that paper, the case $\ell=0$ is handled separately by saying that $\zeta_{\beta,L}^{(0)}(T)=(1-T)^{-1}$ (\cite[(2.1.6)]{HS1983}) is obtained immediately. We quote the result including this case from \cite[Theorem 2]{HS1983}.      
 
\begin{theorem}[Hina-Sugano \cite{HS1983}] \label{ML-L6HS} Set $n_0=n-2\ell$ and  $A:=\tfrac{1}{2}n_0-1$. 
Then, 
we have that 
$$
P_{\beta,L}(T):=\prod_{r=0}^{\ell}(1-(p^{A+\ell-\frac{r-1}{2}})^{fr} T)\times \boldsymbol\zeta_{\beta,L}^{(0)}(T)
$$
is a polynomial in $T$ with constant term $1$ and of degree no greater than $\ell-1$. 
\end{theorem}

\subsection{The numerator of the index function series: the case \texorpdfstring{$f=1$}{}}
\label{sec:NumIDxFtf1}
A description of the numerator of the rational function $\boldsymbol\zeta_{\beta,L}^{(0)}(T)$ is also given in \cite{HS1983}. By Lemma \ref{LemFeQ1}, we have two cases $f=1$ and $f=2$. We first deal with the simpler case $f=1$. Let $t_1,\dots, t_\ell$ be a set of independent variables, where $\ell(\geq 1)$ is the Witt index of $Q$. 
Define 
\begin{align*}
z_{r}&:=p^{r(\ell-r)}\,t_{\ell}\cdots t_{r+1} \quad (r\in [1,\ell-1])
\end{align*}
and set 
\begin{align*}
{\bf q}(t_1,\dots, t_\ell)&:=\prod_{i=0}^{\ell} (1-z_i), \quad 
{\bf p}(t_1,\dots,t_\ell):= \sum_{I\subset [1,\ell-1]} \varphi_{(I)}(p^{-1})\prod_{i\in I}z_i\prod_{j \in [1,\ell]-I} (1-z_j)
, \\
{\bf f}(t_1,\dots, t_\ell)&:=\frac{{\bf p}(t_1,\dots,t_\ell)}{{\bf q}(t_1,\dots,t_{\ell})},
\end{align*}
where, for any subset $I\subset [1,\dots,\ell-1]$, $\varphi_{(I)}(q)$ is a polynomial in a variable $q$ defined as
$$
\varphi_{(I)}(q):=\frac{\varphi_\ell(q)}{\varphi_{i_1-i_0}(q)\cdots \varphi_{i_{r+1}-i_{r}}(q)}
\quad \text{with $\varphi_i(q):=\prod_{j=1}^{i}(q^{j} -1)$}
$$
with $I=(i_1<\dots<i_r)$, $i_0:=0,i_{r+1}:=\ell$. We quote \cite[Proposition 2]{HS1983} and a formula on \cite[Pg.147 line 4]{HS1983} in the following form: 
\begin{theorem}[Hina-Sugano \cite{HS1983}]
When $f=1$, we have 
\begin{align*}
\boldsymbol\zeta_{\beta,L}^{(0)}(T)=(1-T)^{-1}{\bf f}(p^{A+\ell}, \dots,p^{A+2}, p^{A+1}T).   
\end{align*}
\end{theorem}

To make the formula of $\boldsymbol\zeta_{\beta,L}^{(0)}(T)$ more explicit, we introduce the following :

\begin{definition}\label{Def-AW} For $\ell \in \Z_{\geq 0}$, $\varepsilon \in \{+1,-1\}$ and $K\subset [1,\ell-1]$, set
\begin{align}
w_{\ell,K}(q)&:=\sum_{I\subset K} (-1)^{\#(K-I)}\varphi_{(I)}(q)\quad (\in \Q(q)), 
 \label{CK-Def}
\\
a_{\ell}^{(\varepsilon)}(K)&:=\sum_{i\in K} \Bigl(-\tfrac{i(i-\varepsilon)}{2}+\tfrac{\ell(\ell-\varepsilon)}{2}\Bigr) \qquad (\in \Z_{\geq 0})
,\label{aK-Def} \\
b_\ell(K)&:=\sum_{i \in K}(\ell-1-i)\qquad (\in \Z_{\geq 0}). 
\label{bK-Def}
\end{align}
\end{definition}
Note that $w_{\ell,\emptyset}(q)=1$ and $a_{\ell}^{(\varepsilon)}(\emptyset )=b_\ell(\emptyset)=0$; $K=\emptyset$ is the only possible choice of $K$ for $\ell=0,1$. We have 
\begin{align}
a_\ell^{(+1)}(K)=a_{\ell}^{(-1)}(K)-b_\ell(K)-\#(K).
    \label{aab}
\end{align}
  
\begin{lemma} \label{ML-L8} 
Suppose $f=1$. Set $A:=\frac{n_0}{2}-1$. Then, we have 
\begin{align*}
(1-T){\bf q}(p^{A+\ell}, \dots p^{A+2}, p^{A+1}T)=\prod_{r=0}^{\ell}(1-(p^{A+\ell-\frac{r-1}{2}})^{r} T)   
\end{align*}
and 
\begin{align}
P_{\beta,L}(T)&={\bf p}(p^{A+\ell}, \dots p^{A+2}, p^{A+1}T)
=\sum_{K\subset [1,\ell-1]} w_{\ell,K}(p^{-1})\,p^{a^{(-1)}_\ell(K)+A b_\ell(K)}\,(p^{A}T)^{\#(K)}.  
\label{PL-Def}
\end{align}
\end{lemma}
\begin{proof}
A direct computation.     
\end{proof}

\begin{definition}\label{Def-Nftn} 
For any subset $K\subset [1,\ell-1]$, set
$
K':=[1,\ell-1]-K. $.
Let $\mathcal I_K$ denote the set of maximal intervals $I$ of $\Z$ such 
that $I\subset K$. Set
$$
N(K):=\sum_{I \in \mathcal I_K} \frac{\# I(\# I+1)}{2}=\#(K)+\sum_{I\in \mathcal I_K}\frac{\# I(\#I-1)}{2}
$$  
and 
$$
\mu^{(\varepsilon)}(K):=\sum_{k\in K} \frac{k(k-\varepsilon)}{2}+N(K).
$$
\end{definition}
Note that $\sum_{I \in \mathcal I_K}\#(I)=\#(K)$. Write $K=\{k_1<\dots<k_r\}$ and set $k_0=0, k_{r+1}=\ell$; let $a_1,\dots,a_c$ denote the length of maximal segments consisting of $1$'s in the tuple $(k_{\nu+1}-k_{\nu})_{\nu=1}^{r-1}$. Then, $N_\ell(K)=r+\sum_{j=1}^{c}\frac{a_j(a_j+1)}{2}$ ({\it cf} . \cite[Theorem 5.3(2)]{Petrogradsky2007}).
We have the relation:
\begin{align}
\mu^{(\varepsilon)}(K)=N(K)-a_\ell^{(\varepsilon)}(K)+\#(K)\,\tfrac{\ell(\ell-\varepsilon)}{2}. 
\label{muNaell}
\end{align}
Moreover, we have a bound
\begin{align}
N(K)\leq \tfrac{\ell(\ell-1)}{2}, \quad K\subset [1,\ell-1]
  \label{Nell}  
\end{align}
by using the inequality $\tfrac{i(i+1)}{2}+\tfrac{j(j+1)}{2}\leq \tfrac{(i+j)(i+j+1)}{2}$ for any $i,j \in \Z_{\geq 0}$.

We have the following property of the rational functions $w_{\ell,K}(q)\,(K\subset [1,\ell-1])$:

\begin{lemma}\label{ML-L9} Let $K\subset [1,\ell-1]$. 
If $K=\emptyset$, then $w_{\ell,K}(q)=1$. If $K\not=\emptyset$, then $w_{\ell,K}(q)$ is a monic polynomial in $q$ such that $w_{\ell,K}(q)=q^{N(K)}+O(q^{N(K)+1})$.  
\end{lemma}
\begin{proof}
Recall that the $q$-analogue of the integer $k\in \Z_{>0}$ is defined to be 
$[k]_q:=\tfrac{q^{k}-1}{q-1}=1+q+\cdots+q^{k-1}$, which recovers $k$ in the limit $q\rightarrow 1$. For $n\in \Z_{>0}$, the $q$-factorial is defined as $[n]_q!:=\prod_{k=1}^{n}\tfrac{q^k-1}{q-1}=\prod_{j=1}^{n}[k]_q$. With this standard notation, we have $\varphi_j(q)=(q-1)^{j}[j]_q !$, and for any subset $I=\{i_1<\dots<i_r\}\subset [1,\ell-1]$, 
$$
\varphi_{(I)}(q)=\frac{[\ell]_q!}{\prod_{\nu=1}^{r+1}[i_\nu-i_{\nu-1}]_q !}=:\binom{\ell}{i_1-i_0,\dots, i_{r+1}-i_r}_q, \quad \text{with $i_0=0$ and $i_{r+1}:=\ell$}. 
$$
This is the $q$-multinomial coefficient, which is known to be a polynomial in 
$q$. For $K\subset [1,\ell-1]$, we have the expression 
\begin{align}
w_{\ell,K}(q)=(-1)^{\#K} \sum_{I=\{i_1<\dots<i_r\}\subset K}(-1)^{r} 
\binom{\ell}{i_1-i_0,\dots, i_{r+1}-i_r}_q
,
\label{CK-qSer}
\end{align}
which shows that $w_{\ell,K}(q)$ is the same polynomial as the so
called the permutation-descent polynomial that was
originally introduced by \cite{S1976} and  later used in 
\cite{CKK2023} and \cite{Petrogradsky2007} (with the notation 
$w_{\ell,\lambda}(q)$) to study a variant of the sublattice-number-
counting zeta functions on $\Z^{\ell}$. The last claim on the minimal
degree term of $w_{\ell,K}(q)$ is due to \cite[Theorem 5.3(2)]
{Petrogradsky2007}. 
\end{proof}

This is an appropriate moment to introduce the following: 
\begin{definition}\label{Def:PellqT}
Let $T$, $u$ and $q$ be independent variables. For $\ell \in \Z_{\geq 
0}$ and $\varepsilon \in \{+1,-1\}$, define
\begin{align}
P_{\ell}^{(\varepsilon)}(q,T)&:=\sum_{K\subset [1,\ell-1]} w_{\ell,K}(q) q^{-a_{\ell}^{(\varepsilon)}(K)}\,T^{\#(K)} \quad \in \Q[q,q^{-1}][T],\label{Numerator} \\
P_\ell(q,u; T)&:=\sum_{K\subset [1,\ell-1]} w_{\ell,K}(q) q^{-a_{\ell}^{(-1)}(K)}\,u^{b_\ell(K)}\,T^{\#(K)} \quad \in \Q[q,q^{-1},u][T].
\label{NumeratorTotal}
\end{align}
\end{definition}
Note that $P_{\beta,L}(T)=P_{\ell}(p^{-1},p^{A},p^{A}T)$, and 
$$
P^{(-1)}_\ell(q,T)=P_\ell(q,1,T), \quad P^{(+1)}_{\ell}(q,T)=P_\ell(q,q,qT)
$$
due to \eqref{aab}. By Theorem \ref{ML-L6} and Lemma \ref{ML-L8}, we have the formula (which also covers the case $\ell=0$) valid when $f=1$:  
\begin{align}
\boldsymbol\zeta_{\beta,L}^{(0)}(T)=\frac{P_{\ell}(p^{-1}, p^{\frac{n_0}{2}-1}, p^{\frac{n_0}{2}-1}T)}{\prod_{r=0}^{\ell}(1-
p^{\frac{r(n-r-1)}{2}}T)}, 
\label{MainFormula}
\end{align}
which shows that $\boldsymbol\zeta_{\beta,L}^{(0)}(T)$ initially defined relying on a maximal lattice $L\subset V$ actually depends only on the isometry class of the quadratic space $(V,Q)$ over $F$. 

\medskip
\noindent
{\bf Examples} We give a list of $a_\ell^{(\varepsilon)}(K),\,w_{\ell,K}(q)\,(K\subset [1,\ell-1])$ for $\ell=2,3,4$ for later reference: 
\begin{itemize}
    \item ($\ell=2$)
\begin{align}
\begin{array}{|c||c|c|}
\hline
K& \emptyset & \{1\} \\ 
\hline
a_2^{(-1)}(K)& 0 & 2  \\
\hline
b_2(K)& 0 & 0 \\
\hline a_{\ell}^{(+1)}(K) & 0 & 1 \\
\hline 
w_{2,K}(q)& 1 & q \\
\hline
\end{array}
\label{Exp1}  
\end{align}

\item ($\ell=3$). 
\begin{align}
\begin{array}{|c||c|c|c|c|}
\hline
K& \emptyset & \{1\} & \{2\} & \{1,2\} \\ 
\hline
a_3^{(-1)}(K)& 0 & 5 & 3 & 8 \\
\hline 
b_3(K) & 0 & 1 & 0 & 1 \\
\hline 
a_3^{(+1)}(K) & 0 & 3 & 2 & 5 \\
\hline 
w_{3,K}(q)& 1 & q+q^2 & q+q^2 & q^3\\
\hline
\end{array}
\label{Exp2}
\end{align}

\item ($\ell=4$). 
{\tiny
\begin{align}
\begin{array}{|c||c|c|c|c|c|c|c|c|}
\hline
{K} & \emptyset & \{1\} & \{2\} &\{3\}  &\{1,2\} 
&\{1,3\} & \{2,3\} & \{1,2,3\} \\
\hline
a_4^{(-1)}(K)& 0 & 9  & 7 & 4 & 16 & 13 & 11 & 20 \\
\hline
b_4(K) & 0 & 2 & 1 & 0 & 3 & 2 & 1 & 3 \\
\hline
a_4^{(+1)}(K) & 0 & 6 & 5 & 2 & 11 & 9 & 8 & 14 \\
\hline
w_{4,K}(q) & 1  & { q+q^2+q^3} & { q+2q^2+q^3+q^4} & { q+q^2+q^3} & { q^3+q^4+q^5} & { q^2+q^3+2q^4+q^5} & { q^3+q^4+q^5} & q^6\\
\hline
\end{array}
\label{Exp3}
\end{align}}
\end{itemize}

\subsection{The properties of \texorpdfstring{$P_{\ell}^{(\varepsilon)}(q,T)$}{}} 
\label{sec:EFPell}
In these examples, we observe the symmetry: 
\begin{align}
&a_{\ell}^{(\varepsilon)}(K)+a_{\ell}^{(\varepsilon)}(K')=\tfrac{1}{3}{\ell(\ell-1)(\ell+\tfrac{1-3\varepsilon}{4})} ,
\label{EFPell-f0}
\\
&w_{\ell,K}(q^{-1})\times q^{\frac{\ell(\ell-1)}{2}}=w_{\ell,K'}(q) 
\label{EFPell-f1}
\end{align}
for all $K\subset [1,\ell-1]$ and its complement $K':=[1,\ell-1]-K$. 
This is true in general. For the proof we need the following notation. 
\begin{definition} For $\ell \in \Z_{>0}$, let $S_\ell$ denote the
symmetric 
group of $\ell$ letters $\{1,\dots,\ell\}$. For $\pi \in S_\ell$. Let
\begin{align*}
D(\pi)&:=\{i\in [1,\ell] \mid \pi(i)>\pi(i+1)\}, \\
{\rm inv}(\pi)&:=\#\{(i,j) \in [1,\ell]^2\mid i<j, \, \pi(i)>\pi(j)\}.
\end{align*}
\end{definition}

\begin{proposition}\label{EFPell-P1} The relations \eqref{EFPell-f0} and \eqref{EFPell-f1} hold for all $\ell \geq 2$ and for all $K\subset [1,\ell-1]$. 
\end{proposition}
\begin{proof}
The first formula \eqref{EFPell-f0} is elementary. Indeed, 
\begin{align*}
a_\ell^{(\varepsilon)}(K)+a_{\ell}^{(\varepsilon)}(K')=-\sum_{i=1}^{\ell-1} \tfrac{i(i-\varepsilon)}{2} +(\ell-1)\times \tfrac{\ell(\ell-\varepsilon))}{2}=\tfrac{1}{3}\ell(\ell-1)\left(\ell+\tfrac{1-3\varepsilon}{4}\right). 
\end{align*}
For a proof of \eqref{EFPell-f1}, we invoke a formula due to Stanley (\cite{S1976}) for the permutation-descent polynomial :
\begin{align}
w_{\ell,K}(q)=\sum_{\substack {\pi \in S_\ell \\ D(\pi)=K}} q^{{\rm inv}(\pi)}. 
\label{EFPell-f2}
\end{align}
Let $\tau\in S_{\ell}$ be the element defined by $\tau(i)=\ell-i+1\,(i \in [1,\ell])$. Then, $i<j$ if and only if $\tau(i)>\tau(j)$, so that $D(\tau\circ \pi)=[1,\ell]-D(\pi)$ and ${\rm inv}(\tau \circ \pi)=\binom{\ell}{2}-{\rm inv}(\pi)$. By this, 
\begin{align*}
 w_{\ell,K'}(q)=\sum_{D(\pi)=K'}q^{{\rm inv}(\pi)}
&=\sum_{D(\tau \circ \pi)=K'}q^{{\rm inv}(\tau\circ \pi)}\\
 &=\sum_{D(\pi)=K}q^{\binom{\ell}{2}-{\rm inv}(\pi)}
 =q^{\binom{\ell}{2}}w_{\ell,K}(q^{-1})
\end{align*}
as desired. 
\end{proof}
\begin{corollary}\label{EFPell-P2} 
The polynomial $w_{\ell,K}(q)$ is a monic polynomial of degree $\tfrac{\ell(\ell+1)}{2}-N(K')$.    
\end{corollary}
\begin{proof}
 This follows immediately from the last claim of Lemma \ref{ML-L9} and \eqref{EFPell-f1}.    
\end{proof}

Concerning the numbers $a_\ell^{(\varepsilon)}(K)$ and $N(K)$ (Definitions \ref{Def-AW} and \ref{Def-Nftn}), we need a lemma: 
\begin{lemma}\label{EFPell-P3-1} For $d\in [1,\ell-1]$, the minimum value $m_d^{(\varepsilon)}:=\min \{\mu^{(\varepsilon)}(K)\mid K\subset [1,\ell-1], \,\#(K)=d\}$ is attained only at $K=[1,d]$, and 
$$
m_{d}^{(\varepsilon)}=\tfrac{1}{6}d(d+1)\left(d+\tfrac{7-3\varepsilon}{2}\right), \quad d\in [1,\ell-1]. 
$$
\end{lemma}
\begin{proof} We give a proof only for $\varepsilon=-1$ setting $\mu(K)=\mu^{(-1)}(K)$; the case $\varepsilon=+1$ is almost the same. Let $K\subset [1,\ell-1]$ with $\# K=d$. Let $I_1,\dots, I_t$ denote the enumeration of $\mathcal I_K$ such that $I_x\prec I_y$ for $x<y$; then ${\mathfrak p}(K):=(\# I_1, \dots,\# I_t)$ is a partition of $d$. For any partition ${\mathfrak a}=(a_1,\dots a_r)$ of $d$, set $\mathcal K(\fa)$ denote the set of $K\subset [1,\ell-1]$ such that ${\mathfrak p}(K)={\mathfrak a}$. Let $\mu_{\mathfrak a}$ denote the minimal of $\mu(K)\,(K\in \mathcal K(\fa))$. Then, $m_{\ell,d}:=\{\mu(K)\mid K\subset [1,\ell-1],\#(K)=d\}$ is the minimum of $\mu_{\mathfrak a}$ for all partitions $\mathfrak a$ of $d$. When $K$ moves over the set $ \mathcal K({\mathfrak a})$, the number $ N(K)=\sum_{j=1}^{r}\frac{a_j(a_j+1)}{2}$ is constant; for such $K$, the number $s_{\ell}(K):=\sum_{k \in K}\frac{k(k+1)}{2}$ gets minimal when $r$ intervals $I_j\in \mathcal I_K\,(j\in [1,r],\,a_j=\#(I_j))$ should be placed as left as possible in such a way that they are not neighboring each other. The choice is 
$$
I_1=[1,a_1], 
 \quad I_j=[A_{j-1}+j, A_j+j-1]\,(2\leq j \leq r).
 $$
 with $A_j:=\sum_{i=1}^{j}a_i$. Note that $A_r=d$. Let $K_{\mathfrak a}$ be the union of these intervals. Then, since $[1,d+r]-K=\{A_j+j\mid j \in [1,r]\}$, we have
\begin{align*}
s_\ell(K_{\mathfrak a})
&=\sum_{k=1}^{d+r}\frac{k(k+1)}{2}
-\sum_{j=1}^{r}\frac{(A_j+j)(A_j+j+1)}{2}\\
&=\frac{1}{6}(d+r)(d+r+1)(2d+2r+1)-\sum_{j=1}^{r}\frac{(A_j+j)(A_j+j+1)}{2}.
\end{align*}
Hence
{\allowdisplaybreaks
\begin{align}
\mu^\#_{\mathfrak a}:=&\mu_{\mathfrak a}-\frac{1}{6}(d+r)(d+r+1)(2d+2r+1)
+\frac{1}{2}(d+r)(d+r+1)\\
\notag
&=
\frac{(d+r)(d+r+1)}{2}+\sum_{j=1}^{r}\frac{a_j(a_j+1)}{2}
-\sum_{j=1}^{r}\frac{(A_j+j)(A_j+j+1)}{2}
\notag
\\
&=\sum_{j=1}^{r}\frac{a_j(a_j+1)}{2}-\sum_{j=1}^{r-1}\Bigl\{\frac{(A_j+j)(A_j+j-1)}{2}+(A_j+j)\Bigr\}
\notag
\\
&=\sum_{j=2}^{r}\frac{a_j(a_j+1)}{2}-\sum_{j=2}^{r-1}\frac{(A_j+j)(A_j+j-1)}{2}+\sum_{j=1}^{r-1}(A_j+j)
\notag
\\
&=\sum_{j=2}^{r-1}\frac{a_j(a_j+1)}{2}+\frac{a_r(a_r-1)}{2}-\sum_{j=2}^{r-1}\frac{(A_j+j)(A_j+j-1)}{2}+\frac{r(r-1)}{2}+d.
\label{EFPell-P3-1-f0}
\end{align}}It remains to minimize $\mu_{\mathfrak a}$ when $\mathfrak a=(a_1,\dots, a_r)$ moves over all partitions of $d$, i.e., 
$$
a_1,\dots,a_r\in \Z_{>0},\,r\geq 1,\, \sum_{j=1}^{r}a_j=d.
$$
For a fixed $r\in [1,d]$, let $m(d,r)$ denote the minimum of $\mu_{\mathfrak a}$ as $\mathfrak a$ moves over partitions of $d$ by $r$ numbers. The formula in \eqref{EFPell-P3-1-f0} shows that $ \mu^\#_{\mathfrak a}$ is a sum of the quadratic polynomial $\frac{1}{2}a_r(a_r-1)$ added by some quantity that does not depended on $a_r$. Hence, to minimize $\mu_{\mathfrak a}$, we need to choose $a_r=1$ minimizing $\frac{a_r(a_r-1)}{2}$. Then, the remaining variables $(a_1,\dots, a_{r-1})$ should satisfy $\sum_{j=1}^{r-1}a_j=d-1$, or equivalently $A_{r-1}=d-1$. Then, by \eqref{EFPell-P3-1-f0}, $\mu^\#_{(a_1,\dots,a_{r-1},1)}$ equals the quadratic polynomial $\frac{a_{r-1}(a_{r-1}+1)}{2}$ added by a quantity that is independent of $a_{r-1}$. Hence to minimize $\mu_{(a_1,\dots,a_{r-1},1)}$, we need to choose $a_{r-1}=1$ minimizing $\frac{a_{r-1}(a_{r-1}+1)}{2}$. By repeating the same argument, we obtain the unique choice $a_j=1\,(2\leq j\leq r)$, $a_1=d-r$ that minimize $\mu_{\fa}$ for partitions ${\mathfrak a}$ of $d$ by $r$ numbers. Hence $A_j=d-r+j-2\,(1\leq j \leq r)$ and $m(d,r)=\mu_{(d-r+1,1,1,\dots,1)}$, which equals 
\begin{align*}
&
\tfrac{1}{6}(d+r)(d+r+1)(2d+2r+1)
-\tfrac{1}{2}(d+r)(d+r+1)
\\&+
(r-2)-\sum_{j=2}^{r-1}\tfrac{(d-r+2j-1)(d-r+2j-2)}{2}+\tfrac{r(r-1)}{2}+d.
\end{align*}
This simplifies to a degree $3$ polynomial
$$
F(r)=r^3+\tfrac{d}{2}r^2+(\tfrac{5}{6}d-\tfrac{3}{2})r+\tfrac{1}{6}d^3+d^2+\tfrac{17}{6}d-\tfrac{5}{6}.
$$
Since $F'(r)=3r^2+dr+\frac{5}{6}d-\frac{3}{2}>0\,(1\leq r\leq d)$, the function $F(r)$ turns out to be increasing on $[1,d]$ so that it gets minimal when $r=1$. Therefore, we have $m_{\ell,d}=m(d,1)$.   
\end{proof}

\begin{theorem} \label{EFPell-P3}
Let $\ell\geq 2$. We have the following properties of $P_\ell^{(\varepsilon)}(q,T)\,(\varepsilon \in \{+1,-1\})$.  
\begin{itemize}
\item[(i)]
$$
P_\ell^{(\varepsilon)}(q,T)=\sum_{K\subset [1,\ell-1]} w_{\ell,K}(q)\prod_{j\in K} (q^{\frac{j(j-\varepsilon)}{2}-\frac{\ell(\ell-\varepsilon)}{2}}T).
$$
\item[(ii)] The functional equation 
\begin{align*}
P_{\ell}^{(\varepsilon)}(q^{-1},T^{-1})=T^{1-\ell}q^{\frac{1}{3}\ell(\ell-1)(\ell+\frac{1-3\varepsilon}{4})-\frac{\ell(\ell-1)}{2}}\,P_{\ell}^{(\varepsilon)}(q,T).    
\end{align*}
\item[(iii)] The coefficients of $P_\ell^{(\varepsilon)}(q,T)$ viewed as a polynomial in $T$ are identified as follows.   
\begin{align}
P_\ell^{(\varepsilon)}(q,T)=1+\sum_{d=1}^{\ell-1} C_{\ell,d}^{(\varepsilon)}(q^{-1})\,T^{d}
\quad \text{with $C_{\ell,d}^{(\varepsilon)}(q)=\sum_{j \in \Z}b_{\ell,d}^{(\varepsilon)}(j)\,q^{j},
$}
\label{EFPell-f3}
\end{align}
where  
\begin{align}
b_{\ell,d}^{(\varepsilon)}(j):=\#\Bigl\{\pi \in S_\ell \mid \# D(\pi)=d, \, j=-\tfrac{\ell(\ell-\varepsilon)}{2}d+\sum_{\nu \in D(\pi)}\tfrac{\nu(\nu-\varepsilon)}{2}+{\rm inv}(\pi)\Bigr\}.
\label{EFPell-f4}
\end{align}
\item[(iii)] For $d\in [1,\ell-1]$, set 
\begin{align}
j_{+}^{(-1)}(\ell,d)&:=d\Bigl\{-\ell\frac{d-1}{2}+\frac{1}{6}(d+1)(d-4)\Bigr\}, \qquad j_{+}^{(+1)}(\ell,d)=d\Bigl\{-\ell \frac{d-1}{2}+\frac{1}{6}(d+1)(d-1)\Bigr\},  
 \label{j+}
\\ 
j_{-}^{(-1}(\ell,d)&:=d\Bigl\{-\frac{\ell(\ell+1)}{2}+\frac{1}{6}(d+1)(d+5)\Bigr\}, \quad j^{(+1)}_{-}=d\Bigl\{-\frac{\ell(\ell-1)}{2}+\frac{1}{6}(d+1)(d+2)\Bigr\}. 
\label{j-}
\end{align}
The support of the function $j\mapsto b_{\ell,d}^{(\varepsilon)}(j)$ is contained in the interval $[j_{-}^{(\varepsilon)}(\ell,d),j_{+}^{(\varepsilon)}(\ell,d)]$ of $\Z_{<0}$, so that $C_{\ell, d}^{(\varepsilon)}(q^{-1})$ is a polynomial in $q$ with non negative integer coefficients $b_{\ell,d}^{(\varepsilon)}(j)$ of degree $-j_{-}^{(\varepsilon)}(\ell,d)$ without terms of degree less than $-j_{+}^{(\varepsilon)}(\ell,d)$. We have
$b_{\ell,d}^{(\varepsilon)}(j_{+}^{(\varepsilon)}(\ell,d))=b_{\ell,d}^{(\varepsilon)}(j_{-}^{(\varepsilon)}(\ell,d))=1$.
\end{itemize}
\end{theorem}
\begin{proof} The formula in (i) is immediate from \eqref{Numerator}. The functional equation follows from \eqref{Numerator} and Proposition \ref{EFPell-P1}. We have the equality \eqref{EFPell-f3} with 
$$
C_{\ell,d}^{(\varepsilon)}(q):=\sum_{\substack{K\subset [1,\ell-1] \\ \#(K)=d}} w_{\ell,K}(q) q^{-a_\ell^{(\varepsilon)}(K)},
$$
which is surely an element of $\Z[q,q^{-1}]$. By \eqref{EFPell-f2}, we see that the coefficient of $q^j$ is given by \eqref{EFPell-f4}. By Lemma \ref{ML-L9} and Corollary \ref{EFPell-P2}, $w_{\ell,K}(q)q^{-a_{\ell}^{(\varepsilon)}(K)}$ is a sum of monomials $q^{j}$ with $j\in \Z$ such that
\begin{align}N(K)-a_\ell^{(\varepsilon)}(K) \leq j\leq -N(K')+\tfrac{\ell(\ell-1)}{2}-a_\ell^{(\varepsilon)}(K).
\label{EFPell-P2-f0}
\end{align}
Let $j_+^{(\varepsilon)}(\ell,d)$ (resp. $j_{-}^{(\varepsilon)}(\ell,d)$) denote the maximum (resp. minimum) of the function $K\mapsto -N(K')+\tfrac{\ell(\ell-1)}{2}-a_\ell^{(\varepsilon)}(K)$ (resp. $j \mapsto N(K)-a_\ell^{(\varepsilon)}(K)$) defined for subset $K\subset [1,\ell-1]$ with $\#(K)=d$. Moreover, by Lemma \ref{ML-L9}, the coefficients of $q^{j_{\pm}(\ell,d)}$ is the number of $K\subset [1,\ell-1]$ with $\#(K)=d$ that attains the value $j_{\pm}^{(\varepsilon)}(\ell,d)$. 
We first examine $j_{+}(\ell,d)$. By \eqref{EFPell-f0}, we have $-N(K')-a_\ell^{(\varepsilon)}(K)=-N(K')+a_\ell^{(\varepsilon)}(K')-\frac{1}{3}\ell(\ell-1)(\ell+\frac{1-3\varepsilon}{4})$. Hence we need to minimize $N(K')-a_\ell^{(\varepsilon)}(K')$, which equals $\mu^{(\varepsilon)}(K')-\frac{\ell(\ell-\varepsilon)}{2}(\ell-1-d)$ by \eqref{muNaell}.
By Lemma \ref{EFPell-P3-1}, we get
$$
j_{+}^{(\varepsilon)}(\ell,d)=-m_{\ell-1-d}^{(\varepsilon)}+\tfrac{\ell(\ell-\varepsilon)}{2}(\ell-1-d)-\tfrac{1}{3}\ell(\ell-1)(\ell+\tfrac{1-3\varepsilon}{4})+\tfrac{\ell(\ell-1)}{2},
$$
which is simplified as in \eqref{j+} after a computation, and the $j_{+}^{(\varepsilon)}(\ell,d)$ is attained only when $K'=[1,\ell-d]$. Hence, the coefficient of $q^{j_{+}^{(\varepsilon)}(\ell,d)}$ is $1$, i.e., $b_{\ell,d}(j_+^{(\varepsilon)}(\ell,d))=1$. 

The value $j_-^{(\varepsilon)}(\ell,d)$ is identifies as 
$$
j_-^{(\varepsilon)}(\ell,d)=m_{d}^{(\varepsilon)}-\tfrac{\ell(\ell-\varepsilon)}{2}d
$$
due to \eqref{muNaell}, and by Lemma \ref{EFPell-P3-1}, the value is attained only at $K=[1,d]$. Thus, the coefficient of $q^{j_{-}^{(\varepsilon)}(\ell,d)}$ is $1$, i.e., $b_{\ell,d}^{(\varepsilon)}(j_{-}^{(\varepsilon9}(\ell,d))=1$. 
\end{proof}
To get Theorem \ref{MAINTHM1} from this theorem, we only have to make the following definition: 
 \begin{definition} \label{Well-Def}
\begin{align*}
W_\ell^{(\varepsilon)}(X,T)&:=P_\ell^{(\varepsilon)}(X^{-1},T), \qquad W_\ell(X,u,T):=P_\ell(X^{-1},u,T), \\
\alpha_{\ell,d}^{(\varepsilon)}&:=-d^{-1}\,j_{+}^{(\varepsilon)}(\ell,d), \qquad  \beta_{\ell,d}^{(\varepsilon)}:=-d^{-1}\,j_{-}^{(\varepsilon)}(\ell,d). 
\end{align*}
\end{definition}
Let us confirm formula \eqref{Intro-f1}. From Theorem \ref{EFPell-P3}(i) and \cite[Theorem 3.1]{Petrogradsky2007}, the equality $W_\ell^{(\varepsilon)}(p,p^{-z})=f(s_1,\dots,s_{\ell-1})$ holds when 
$$
\tfrac{\nu(\nu-\varepsilon)}{2}-\tfrac{\ell(\ell-\varepsilon)}{2}+z=\sum_{j=1}^{\nu}s_\nu-\nu(\ell-\nu), \quad \nu\in [1,\ell].
$$
This linear system has the unique solution $s_1=z-\frac{(\ell-1)(\ell-1-\varepsilon)}{2}$, $s_j=\ell-j+\tfrac{1-\varepsilon}{2}\,(j \in [2,\ell])$.

\subsection{The numerator of the index function series: the case \texorpdfstring{$f=2$}{}}
\label{sec:IdxFtSrf2}
Let $V$, $\beta$ and $Q$ be as before. In this subsection, we suppose that the Witt index of $(V,Q)$ is $\ell\geq 1$. Recall that $G$ is the connected component of the orthogonal similitude group of $Q$ and $\mu:G\rightarrow F^\times$ the character of similitude. Our natural target of study is the generating series
\begin{align}
\zeta_{\beta,L}(s):=\sum_{\Lambda \in \fL_{\beta,L}^{\rm max}}[L:\Lambda]^{-s}, \quad s\in \C
\label{localGNS}
\end{align}
with $\fL_{\beta,L}^{\rm max}:=\{\Lambda \in \fL_\beta^{\rm max}\mid \Lambda \subset L\}$. As we saw in Corollary \ref{ML-L6}, $\zeta_{\beta,L}(s)=\boldsymbol\zeta_{\beta,L}^{(0)}(p^{-ns/2})$ if $f:=[\Z: {\rm ord}_p \mu(G)]=1$. If $f=2$, the local index-function series $\boldsymbol\zeta^{(0)}_{\beta,L}(T)$ only accounts for sub-lattices of $L$ with the ratio of norms being even powers of $p$. To cover the remaining lattices, i.e., those $\Lambda$ in 
$$
\fL_{\beta,L}^{\rm max,(1)}:=\{\Lambda\in \fL_{\beta,L}^{\rm max}\mid {\rm ord}_p( \mu(\Lambda)\mu(L)^{-1}) \equiv 1 \pmod{2}\},
$$
we introduce a power series
$$
\boldsymbol\zeta_{\beta,L}^{(1)}(T):=\sum_{m=0}^{\infty}\#(\fL_{\beta,L}^{\rm max}
(p^{2m+1}))\,T^{m} \quad (\in \Z[[T]])
$$
with notation as in \eqref{fLnormPm}. At this point we need a lemma to
create a link between the group index $[L:\Lambda]$ and the ratio of the
norms $\mu(\Lambda)/\mu(L)$ for $\Lambda \in \fL_{\beta,L}^{\rm max, 
(1)}$. Hence we need an additional definition. Take a Witt decomposition \eqref{WittLat} of $L$ with $L_0=\{w\in W\mid Q(w)\in \mu(L)\}$; then , $L_0^{[+1]}:=\{w\in W \mid Q(w)\in p \mu(L)\}$ is a sub $O$-lattice of $L_0$ and $L_0/L_0^{[+1]}$ naturally becomes an $O/pO\cong{\mathbb F}_p$-module. Set $\partial:=\dim_{{\mathbb F}_p}(L_0/L_0^{[+1]})$, which is independent of the choice of a Witt decomposition \eqref{WittLat}. 

\begin{lemma} For $M \in \fL_{\beta,L}^{\rm max,(0)}$, we have $[L:M]=|\mu(M)/\mu(L)|_p^{-n/2}$. For $M \in \fL_{\beta,L}^{\rm max, (1)}$, we have
$$
[L:M]=p^{-B}|\mu(M)/\mu(L)|_p^{-n/2} \quad \text{with $B:=-\partial+\tfrac{n_0}{2}$.}
$$
\end{lemma}
\begin{proof} The case $M\in \fL_{\beta,L}^{\rm max, (0)}$ is treated in the proof of Lemma \ref{ML-L6}. Let $M\in \fL_{\beta,L}^{\rm max, (1)}$. Set $\alpha={\rm ord}_p\mu(L)$ and $\beta={\rm ord}_p\mu(M)$. Then we can take \eqref{WittLat} so that $M$ is decomposed as in \eqref{WittLat2} simultaneously; note that $\beta-\alpha\geq \gamma_j \geq 0$ due to $M\subset L$ and ${\rm ord}_p(\mu(M)\mu(L)^{-1})\equiv 1 \pmod{2}$. Set $L_1:=\sum_{j=1}^{\ell}(p^{r_j}Ov_j+p^{r_j'+\alpha}O v_j^*)+L_0^{[+1]}$ with $r_j,r_j\geq 0, r_j+r_j'=1$. Then, $L_1\in \fL_{\beta,L}^{\rm max}$ and $\mu(L_1)=p\mu(L)$ by Lemma \ref{CSTMaxLatt} (here we need $\ell\geq 1$); we also have $M\subset L_1$.
 Moreover, $L/L_1$ is isomorphic to a direct sum of $\ell$ copies of $O/pO\cong {\mathbb F}_p$ and $L_0/L_0^{[+1]}\cong {\mathbb F}_p^{\partial}$, we have $[L:L_1]=p^{\ell+\partial}$. Since ${\rm ord}_p\mu(M)/\mu(L_1)$ is even, we have $[L_1:M]=|\mu(M)/\mu(L_1)|_p^{-n/2}$. Then, we have the desired formula by using the relations $\mu(L_1)=p\mu(L)$, $[L:M]=[L:L_1][L_1:M]$ and $n=2\ell+n_0$.    
\end{proof}
By using this lemma, we easily have an identity of formal series:  
\begin{align}
\zeta_{\beta,L}(s)=\boldsymbol\zeta_{\beta,L}^{(0)}(p^{-ns})+\boldsymbol\zeta_{\beta,L}^{(1)}(p^{-ns})\,p^{-(n/2-B)s}. 
\label{totalZ}
\end{align}

To state our final formula exactly, we need some additional notation. For a power series $F(T)=\sum_{m=0}^{\infty}a_m T^m\in R[[T]]$ with coefficients in a commutative ring $R$, define its even part and the odd part as 
$$
F^{\rm even}(T):=\sum_{m=0}^{\infty}a_{2m}T^{2m}, \quad F^{\rm odd}(T):=\sum_{m=0}^{\infty}a_{2m+1}T^{2m+1}.
$$
Let $X,u$ be indeterminates served as parameters and set $R=\Q[u,X,X^{-1}]$.  
\begin{definition}
For any $\ell \in \Z_{>0}$ and for any subset $J\subset [0,\ell-1]$, define an element ${\mathscr U}_{\ell, J}(X,u;T) \in R[T]$ as 
\begin{align}
{\mathscr U}_{\ell,J}(X,u,T):=u^{b_\ell(J)} X^{a_\ell^{(-1)}(J)}\sum_{ \substack{I_0\subset J \cap [1,\ell-1] \\ 
I_1\subset [1,\ell-1]-J}} w_{\ell,I_0\cup I_1}(X^{-1})\, (X^{a_\ell^{(-1)}(I_1)} u^{b_\ell(I_1)}T^{\#(I_1)})^2,
\label{mathcalWellJ-Def}
\end{align}
where $w_{\ell, K}(q)$ is an in \eqref{CK-qSer}, $a_\ell^{(-1)}(K)$ as in \eqref{aK-Def}, and $b_\ell(K)$ as in \eqref{bK-Def}.
\end{definition}

Let $L$ be a maximal $\fa$-integral lattice in $(V,Q)$ with the Witt decomposition \eqref{WittLat}. Recall $L_0=\{w\in W\mid Q(w)\in \mu(L)\}$; we set $L_0^{[+1]}:=\{w\in W\mid Q(w)\in p\mu(L)\}$. Let $R_0\in \mu(L){\bf Mat}_{n_0}(O)$ be a matrix representing the quadratic form $Q|_{W}$ with respect to an $O$-basis of $L_0$. Let $\alpha$ denote ${\rm ord}_p \mu(L) \pmod{2}$.\footnote{$\mu(L_0)$ equals $\mu(L)$ or $p\mu(L)$.} For convenience we include a classification list of $(R_0,\alpha)$ together with $n_0=\dim W$, $\partial:=\dim_{{\mathbb F}_p}L_0/L_0^{[+1]}$, $A:=n_0/2-1$ and $B:=-\partial+n_0/2$ below, which is a part of \cite[Table 1 (Pg.148)]{HS1983} with $f=2$ (see Lemma \ref{LemFeQ1})\footnote{two cases with $*$, referred to as the unramified cases, are of particular importance when we consider global Euler products (see \S\ref{sec:type2Factor}).} :
\begin{align}
\begin{array}{|c||c|c|c|c|}
\hline
(R_0, \alpha) & n_{0} &\partial & A  & B \\
\hline
 (2s, 0)^*, (2ps,1) & 1 &1 & -1/2  & -1/2  \\
\hline
(2s,-1),(2ps,0)  & 1  & 0 & -1/2 & 1/2 \\
\hline
(S_2,0)^*, (pS_2,1) & 2 & 2 & 0 & -1 \\
\hline 
(S_2,-1), (pS_2,0) & 2 & 0 & 0 & 1 \\
\hline 
(2ps\oplus S_2,0), (2p^2s\oplus pS_2,1) & 3 & 2 & 1/2 & -1/2 \\
\hline 
(2s\oplus pS_2,0), (2ps \oplus p^2 S_2,1) & 3 & 1 & 1/2 & 1/2\\
\hline
\end{array}
\label{TableAB}
\end{align}
Here, $S_2$ is a symmetric matrix of the form
\begin{align*}
S_2&=\left[\begin{smallmatrix}2a_2 & b_2 \\ b_2 & 2 c_2 \end{smallmatrix}\right]\,\quad (a_2,c_2\in O^\times, \,b_2\in O,\,\text{$F(\sqrt{-\det S_2})$ is an unramified field extension of $F$}),
\end{align*}
and $s$ denotes an element of $O^\times$. Note that $\ell$ and $A$ are isometry invariants of the quadratic space $(V,Q)$, whereas $B$ depends on the lattice $L$. Define  
\begin{align}
U^{\bullet}_{\ell,B}(X,u,T):=\sum_{\substack{J \subset [0,\ell-1] \\ \#(J)\equiv \epsilon \pmod{2}}}X^{B \langle J \rangle }  {\mathscr U}_{\ell,J}(X, u, T) T^{\#(J)} \quad (\bullet \in \{{\rm even}, {\rm odd}\})
\label{WlBXuT-Def}
\end{align}
with $\epsilon=0$ if $\bullet ={\rm even}$ and $\epsilon=1$ if $\bullet ={\rm odd}$ and, for $J=\{j_1<\dots<j_t\} \subset [0,\ell-1]$, set  
\begin{align}
\langle J\rangle:=\sum_{\nu=1}^{t}(-1)^{\nu}(j_\nu-\ell).
\label{bracketJ}
\end{align}
It is easy to see $\langle J \rangle \in [0,\ell]$. Moreover, depending on 3 parameters $(\ell,A,B)$, we define 
\begin{align}
W_{\ell,A,B}^{(0)}(X,T)&:=U^{{\rm even}}_{\ell,B}(X,X^{A},X^{A}T)
+U^{\rm odd}_{\ell,B}(X,X^{A}, X^{A}T)\,T, 
 \label{WellB0-Def}
\\
W_{\ell,A,B}^{(1)}(X,T)&:=U^{{\rm even}}_{\ell,-B}(X,X^{A}, X^{A}T)\,T+U^{{\rm odd}}_{\ell,-B}(X,X^{A}, X^{A}T), 
\label{WellB1-Def}
\end{align}
and, with $n=2(\ell+A+1)$ we define
\begin{align}
W_{\ell,A,B}^{\rm total}(X,Y)&:=W_{\ell,A,B}^{(0)}
(X,Y^{n})+W^{(1)}_{\ell,A,B}(X,Y^{n})\,Y^{-2B} 
\label{WellBtotal-Def}
\end{align}
and  call it the {\it{\underline {total ASH}} polynomial}.

The variable $s$ will be related to the indeterminates $T$ and $Y$ by
$Y=p^{-s/2}$ and $T=Y^{n}=p^{-ns/2}$. The following theorem gives us 
an explicit formula of $\zeta_{\beta,L}(s)$. 

\begin{theorem} \label{IdxFtSrf2-T1} For each of cases on table \eqref{TableAB}, we have
\begin{align}
\boldsymbol\zeta^{(0)}_{\beta,L}(T^2)&=\prod_{r=0}^{\ell} (1-p^{2(A+\ell-\frac{r-1}{2})r} T^2)^{-1}\,W^{(0)}_{\ell,A,B}(p,T),   
\label{IdxFtSrf2-T1-f1}
\\
\boldsymbol\zeta_{\beta,L}^{(1)}(T^2)\,T&=\prod_{r=0}^{\ell} (1-p^{2(A+\ell-\frac{r-1}{2})r} T^2)^{-1}\, W_{\ell,A,B}^{(1)}(p,T).
\label{IdxFtSrf2-T1-f2}
\end{align}
Moreover, 
\begin{align}
\zeta_{\beta,L}(s)&=\prod_{r=0}^{\ell} (1-p^{2(A+\ell-\frac{r-1}{2})r} p^{-ns})^{-1}\,W^{\rm total}_{\ell,A,B}(p,p^{-s/2}).
\label{IdxFtSrf2-T2-f3}
\end{align}
 \end{theorem}

\begin{proposition}\label{IdxFtSrf2-T2}
 The total ASH polynomial belongs to $\Z[X,Y]$ with positive coefficients, with constant term $1$
 and with $2\ell n $ as its degree in  $Y$. We have the functional equation 
\begin{align*}
U_{\ell,B}^{\rm even}(X^{-1}, u^{-1}, T^{-1})=X^{-\frac{2}
{3}\ell(\ell^2-1)-(B+1)\ell} u^{-(\ell-1)^2}T^{1-2\ell} U_{\ell,B}^{\rm 
odd}(X,u,T).
\end{align*}
\end{proposition}
\begin{remark} \label{Wtotal-W}
The pair $(A,B)$ is from table \eqref{TableAB}, so that 
$B=0$ does not occur when $f=2$. However we can view $B\in \frac{1}
{2}\Z$ in \eqref{WellBtotal-Def} as a parameter. Then, the value at 
$B=0$ recovers $W_\ell(X,u,T)$ \textup{(}in Definition \ref{Well-Def}\textup{)} as
\begin{align}
W_{\ell,A, B}^{\rm total}(X,Y)|_{B=0}=\prod_{r=0}^{\ell}(1+X^
{-\frac{r(r+1)}{2}+\frac{\ell(\ell+1)}{2}} X^{A(\ell-r-1)}Y^{n})\,W_{\ell}(X,X^{A},X^A Y^{n}).
\label{WellB0}
\end{align}
Actually, $B=-\partial+\frac{n_0}{2}$ is defined including the case $f=1$, and $B=0$ happens if and only if $f=1$ by \cite[Table 1]{HS1983}. Then, \eqref{WellB0} means that \eqref{IdxFtSrf2-T2-f3} reduces to \eqref{MainFormula}  when $f=1$.  
\end{remark}

\subsubsection{{\underline{Proof of Theorem \ref{IdxFtSrf2-T1}: step 1}}}
The formula \eqref{IdxFtSrf2-T1-f1} is obtained by the argument in 
\cite[Pg.147]{HS1983}. Though the series $\boldsymbol{\zeta}_{\beta,L}^{(1)}(T)$ is 
not treated in \cite{HS1983}; the same argument can be easily carried
over. Since the index function series is treated rather briefly as a secondary object in \cite{HS1983}, we include a detailed account. Let $\mu(L)=p^{\alpha}O$ treating the new case. Fixing the Witt decomposition \eqref{WittLat}, define 
\begin{align}
  L^{[+1]}:=\sum_{j=1}^{\ell}(O v_j+ p^{\alpha+1}Ov_j^*)+L_0^{[+1]},
\label{WittLprime}
\end{align}
where $L^{[+1]}_0:=\{w\in W\mid Q(w) \in p^{\alpha+1}O\}$. Since $\ell\geq 1$, by Lemma \ref{CSTMaxLatt}, we have that $L^{[+1]}$ is a maximal $p^{\alpha+1}O$-integral lattice contained in $L$. Fix an $O$-basis $\{w_1,\dots,w_{n_0}\}$ of $L_0$, and set $R_{0}=(\beta(w_i,w_j))_{ij}\in {\bf Mat}_{n_0}(F)$. Then, for any $t\in \Z$, $L_0^{(t)}:=\{z\in {\bf Mat}_{n_0,1}(F)\mid \tfrac{1}{2} {}^t Z R_0 Z \in p^{t}O\}$ is an $O$-lattice in $F^{n_0}:={\bf Mat}_{n_0,1}(F)\cong W$. For example, $L_0=L_0^{(\alpha)}$ and $L_0^{[+1]}=L_0^{(\alpha+1)}$. By the basis $\{v_1, \dots,v_\ell, w_1,\dots,w_{n_0}, v_1^*, \dots,v_\ell^{*}\}$ of $V$, we represent any $F$-endomorphism $g$ of $V$ by an $n\times n$-matrix. Let $P$ be the minimal parabolic subgroup of $G$ consisting of all $g\in G$ whose matrix is of the form 
$$
\varphi_{C,h}(Y,Z):=\left[\begin{matrix} \mu_0(h){}^t C^{-1} & -{}^t({}^t hR_0 Z C^{-1}) & Y \\ 0 & h & Z \\ 0 & 0 & C\end{matrix}\right]
$$
with $h\in G_0$, $C\in {\bf GL}_\ell(F)$, $Y\in {\bf Mat}_{\ell}(F)$ and $Z\in {\bf Mat}_{n_0,\ell}(F)$ such that 
\begin{align*}
{}^t Y\,C+{}^t C\,Y+{}^t Z R_0 Z=0.
\end{align*}
Note that $\mu(\varphi_{C,h}(Y,Z))=\mu_0(h)$. Define 
$$
U:=\{g\in G\mid g(L)=L\}, \quad U':=\{g\in G\mid g(L^{[+1]})=L^{[+1]}\}.
$$
Since $L$ and $L^{[+1]}$ are maximal, we have Iwasawa decompositions $G=PU=PU'$ (for a proof see \cite[\S9, Pg.52]{S1963}). By Lemma \ref{ML-L3}, the group $G$ acts transitively on the set of $\Lambda \in \fL_{\beta}^{\rm max}$ with ${\rm ord}_p(\mu(\Lambda)\mu(L)^{-1})\equiv 1 \pmod{2}$, which in turn implies that the left $U'$-cosets contained in 
$$
X'_L(m):=\{g\in G\mid g(L^{[+1]})\subset L, \, {\rm ord}_p\mu(g)=2m\}, \quad m\in \Z_{\geq 0}
$$
is in a bijective correspondence with $\fL_{\beta,L}^{\rm max}(p^{2m+1})$ (see \eqref{fLnormPm}) by the map $g\mapsto g(L^{[+1]})$. Thus, in the same way as in the proof of Lemma \ref{ML-L5}, we have 
\footnote{The set $X_L'(m)$ is stable by the left translations by $U$; note that $\#(X'_{L}(m)/U')$ may differ from $\#(U\backslash X'_{L}(m))$, where as $\#(X_{L}(m)/U)=\#(U\backslash X_{L}(m))$ in the case of $\boldsymbol\zeta_{\beta,L}^{(0)}(T)$. }
$$
\boldsymbol\zeta_{\beta,L}^{(1)}(T)=\sum_{m=1}^{\infty}\#(X'_{L}(m)/U')\,T^{m}.
$$  
The group $U'\cap P$ acts on the set $X_L'(m)\cap P$ by the right-multiplication so that the natural map $X'_{L}(m)\cap P\rightarrow X_{L}'(m)/U'$ induces an injection $X'_{L}(m)\cap P/U'\cap P\rightarrow X_{L}'(m)/U'$, which is surjective due to the Iwasawa decomposition $G=PU'$. Define a group homomorphism $\mu_{P}:P\rightarrow {\bf GL}_\ell(F) \times G_0$ by $\mu_{P}(\varphi_{C,h}(Y,Z))=(C,h)$. For $C\in {\bf GL}_\ell(F)$ and $h\in G_0$, define 
$$F'(C,h):=\mu_{P}^{-1}(C,h)\cap \{g\in P\mid g(L^{[+1]})\subset L\}.
$$
Let $N_{P}$ be the unipotent radical of $P$, which coincides with the fibre  $\mu_{P}^{-1}(1_\ell, 1_{n_0})$. The group $N_P\cap U'$ acts on the set $F'(C,h)$ by the right-multiplication. Let $\bar F'(C,h)$ denote the quotient set $F'(C,h)/N_P\cap U'$. Fix an element $\varpi \in G_0$ such that $\mu_0(\varpi)\in p^2O^\times$. For ${\bf r}=(r_1,\dots,r_\ell)\in \Z^{\ell}$ with $r_1\leq \dots \leq r_\ell$, let 
$${\mathfrak R}({\bf r}):={\bf GL}_\ell(O) {\rm diag}(p^{r_1},\dots, p^{r_\ell}){\bf GL}_\ell(O), \qquad R({\bf r}):={\mathfrak R}({\bf r})/{\bf GL}_\ell(O).$$

\begin{lemma} \label{IdxFtSrf2-T1-L1}
The set $F'(C,h)$ is non-empty if and only if $\mu_0(h) \in O-\{0\}$ and $C\in {\mathfrak R}({\bf r})$ with $-1\leq r_1 \leq \dots \leq r_\ell\leq {\rm ord}_p\mu_0(h)$, in which case $F'(C,h)$ coincides with the set of points $\varphi_{C,h}(Y,Z)$ such that 
\begin{align}
Y\in &p^{-\alpha-1}{\bf Mat}_{\ell}(O), \quad Z=(z_1,\dots,z_l)\in (L_0^{(-\alpha-2)})^{\oplus \ell}, 
\label{IdxFtSrf2-T1-L1-f000}
\\
&{}^t R_0 Z C^{-1} \in (L_0^{(\alpha)})^{\oplus \ell}, \qquad 
 {}^t YC+{}^t CY=-{}^t Z R_0 Z.
\label{IdxFtSrf2-T1-L1-f0}
\end{align}
When $C={\rm diag}(p^{r_1},\dots,p^{r_\ell})$, these conditions are replaced with 
\begin{align}
Y\in &p^{-\alpha-1}{\bf Mat}_{\ell}(O), \quad Z=(z_1,\dots,z_l)\in \bigoplus_{i=1}^{\ell} L_0^{(r_i-\alpha-1)}, 
\label{IdxFtSrf2-T1-f0000}
\\
&{}^t YC+{}^t CY=-{}^t Z R_0 Z. 
\notag
\end{align}
We have the equality $N_P\cap U'=F'(1_\ell,1_{n_0})$.     
\end{lemma}
\begin{proof} Due to \eqref{WittLat} and \eqref{WittLprime}, the condition $\varphi_{C,h}(Y,Z)(L^{[+1]})\subset L$ is equivalent to the following: 
\begin{align}
&\mu_0(h){}^t C^{-1} \in {\bf Mat}_{\ell}(O), 
 \label{IdxFtSrf2-T1-L1-f1}
 \\
&{}^t h R_0 Z C^{-1} \in (L_0^{(\alpha)})^{\ell \oplus}, 
 \label{IdxFtSrf2-T1-L1-f2}
\\
&h(L_0^{(\alpha+1)})\subset L_0^{(\alpha)}, 
\label{IdxFtSrf2-T1-L1-f3}
\\
&p^{\alpha+1} Y \in {\bf Mat}_{\ell}(O), 
\label{IdxFtSrf2-T1-L1-f4}
\\
&p^{\alpha+1} Z \in (L_0^{(\alpha)})^{\ell \oplus}, 
\label{IdxFtSrf2-T1-L1f5}
\\
&p^{\alpha+1} C\in p^{\alpha} {\bf Mat}_\ell(O).   
\label{IdxFtSrf2-T1-L1-f6}
\end{align}
By \eqref{IdxFtSrf2-T1-L1-f1} and \eqref{IdxFtSrf2-T1-L1-f6}, we have $\mu_0(h)\in p^{-1}O$, which is equivalent to $\mu_0(h) \in O$ due to $f=2$. Then these two conditions are equivalent to $C\in {\mathfrak R}({\bf r})$ with $-1\leq r_1\leq \dots \leq r_\ell \leq {\rm ord}_p\mu_0(h)$. Set $2m:={\rm ord}_p\mu_0(h)$. Then, Since $h(L_0^{(\alpha+1)})=L_{0}^{(\alpha+1+2m)}$, condition \eqref{IdxFtSrf2-T1-L1-f3} is equivalent to $L^{(\alpha+1+2m)}\subset L_0^{(\alpha)}$, which is implied by $2m\geq 0$. Since $ p^{t}L^{(\alpha)}=L^{(\alpha+2t)}$, condition \eqref{IdxFtSrf2-T1-L1f5} is written as $Z\in (L_0^{(-\alpha-2)})^{\ell \oplus}$. Suppose $C={\rm diag}(p^{r_1}, \dots,p^{r_\ell})$; then taking the $(i,i)$-entry of the second equation in \eqref{IdxFtSrf2-T1-L1-f0} yields 
$\tfrac{1}{2}{}^t z_i R_0 z_i=-p^{r_i} y_{ii},
$ whose right-hand side is in $p^{r_i}\times p^{-\alpha-1}O$ due to \eqref{IdxFtSrf2-T1-L1-f4}. Hence $z_i \in L_{0}^{(r_i-\alpha-1)}$ for all $i$. Since $r_i\geq r_1\geq -1$ as shown above, we have $L^{(r_i-\alpha-1)}\subset L^{(-\alpha-2)}$. Hence the second condition in \eqref{IdxFtSrf2-T1-L1-f000} is implied. We have 
\begin{align*}
{}^t h R_0 z_i p^{-r_i}=\mu_0(h) R_0 h^{-1} z_i p^{-r_i}\in  R_0 (p^{2m-r_i}h^{-1}(L_0^{(r_i-\alpha-1)})). 
\end{align*}
Moreover, $p^{2m-r_1}h^{-1}(L_0^{(r_i-\alpha-1)})=L_0^{(4m-2r_i-2m+r_i-\alpha-1)}=L_0^{(2m-r_i-\alpha-1)}$, which is contained in $L_0^{(-\alpha-1)}$ due to $2m-r_i\geq 0$. Since $L_0^{(\alpha+1)}$ is defined to be the set of $z\in F^{n_0}$ such that $\frac{1}{2}{}^t z R_0 z \in p^{\alpha+1}O$, we have $R_0\in p^{\alpha+1}{\bf Mat}_{n_0}(O)$, which implies $R_0(L_{0}^{(-\alpha-1)})\subset L_{0}^{(\alpha+1)}\subset L_0^{(\alpha)}$. Thus, the first condition in \eqref{IdxFtSrf2-T1-L1-f0} is implied. 
\end{proof}

\begin{lemma} \label{IdxFtSrf2-T1-L2} The set $X'_{L}(m)\cap P$ is a disjoint union of its subsets $F'(C,h)$ for $C\in {\mathfrak R}({\bf r})$, $h \in G_0$ with 
$$-1 \leq r_1\leq \cdots\leq r_{\ell}\leq 2m, \quad \mu_{0}(h)\in p^{2m}O^\times.
$$ 
The set $X'_{L}(m)\cap P/U'\cap P$ is a disjoint union of sets $\bar F'(C,\varpi^{m})$ with $C\in R({\bf r})$, $-1\leq r_1\leq \dots \leq r_\ell \leq 2m$. 
\end{lemma}
\begin{proof}
The group $U'\cap P$ is a semi-direct product of $U'\cap N_P=F'(1_\ell,1_{n_0})$ and the group of points $\varphi_{u,v}(0,0)\,(u\in {\bf GL}_\ell(O), v\in O^\times)$. Since $F'(C,h)\,\varphi_{u,v}(0,0)=F'(Cu,hv)$, the claim follows from the definition of $\bar F'(C,h)$.
\end{proof}

\begin{lemma} \label{IdxFtSrf2-T1-L3} $\# \bar F'(u_1 C u_2, \varpi^{m})=\# \bar F'(C,\varpi^m)$ for any $u_1,u_2\in {\bf GL}_{\ell}(O)$.     
\end{lemma}
\begin{proof} ({\it cf}. \cite[Lemma 2]{HS1983}) We have a bijection from $F'(C,h)$ onto $F'(u_1 Cu_2,h)$ defined by $g\mapsto \varphi_{u_1,1}(0,0) g \varphi_{u_2,1}(0,0)$. Since $U'\cap N_P$ is a normal subgroup of $U'\cap P$, by passing to the quotients, the bijection yields a bijection $\bar F'(C,h)\rightarrow \bar F'(u_1 C u_2,h)$.    
\end{proof}

\begin{lemma} \label{IdxFtSrf2-T1-L4} Let $m\in \Z_{\geq 0}$ and $C={\rm diag}(p^{r_1},\dots, p^{r_\ell})$ with $-1\leq r_1\leq \dots \leq r_\ell \leq 2m$. Then,  $$F'(C,\varpi^{m})\ni \varphi_{C,\varpi^m}(Y,Z)\longmapsto Z \in {\mathfrak Z}({\bf r}):=\bigoplus_{i=1}^{\ell} L_0^{(r_i-\alpha-1)}$$
is a surjective map, any of whose fibers has a simply transitive action by the additive group
$$
{\mathfrak Y}({\bf r}):=\{Y=(y_{ij})_{ij} \in p^{-\alpha-1}{\bf Mat}_{\ell}(O) \mid y_{ii}=0\,(i \in [1,\ell]),\, p^{r_j}y_{ij}+p^{r_i}y_{ji}=0\,(i>j)\}.
$$ 
\end{lemma}
\begin{proof} ({\it cf}. \cite[Lemmas 3, 4]{HS1983}). \end{proof} 
For $Z\in {\mathfrak Z}({\bf r})$, fix $Y_Z\in F'(C,\varpi^m)$ such that $\varphi_{C,\varpi^m}(Y_Z,Z)\in F'(C,\varpi^m)$. Define
\begin{align*}
&Z'({\bf r}):=\bigoplus_{i=1}^{\ell}L_0^{(r_i-\alpha-1)}/L_0^{(2m-\alpha-1)},  \\
&Y'({\bf r}):=\{Y=(y_{ij})_{ij}\mid y_{ii}=0, y_{ji}\in p^{-\alpha-1}O/p^{2m-r_j-\alpha-1}O,\,y_{ij}=-p^{r_j-r_i}y_{ji}\,(i<j)\}.
\end{align*}

For any ${\bf r}=(r_j)_{j=1}^{\ell} \in \Z^{\ell}$, set
$$
s({\bf r}):=\{i \in [1,\ell]\mid r_i \equiv 1 \pmod{2}\}, \quad \varepsilon({\bf r}):=\# s({\bf r}). 
$$

\begin{lemma}\label{IdxFtSrf2-T1-L5} Let $m\in \Z_{\geq 0}$ and $C={\rm diag}(p^{r_1},\dots, p^{r_\ell})$ with $-1\leq r_1\leq \dots \leq r_\ell \leq 2m$. The points $(Y_Z+Y,Z)$ with $Y\in Y'({\bf r}), Z\in Z'({\bf r})$ comprise of a complete set of representatives of the quotient set $\bar F'(C,\varpi^{m})$. We have
\begin{align*}
\# Z'({\bf r})=p^{\frac{n_0}{2}\sum_{i=1}^{\ell}(2m-r_i)+(\partial-\frac{n_0}{2})\varepsilon({\bf r})}, \quad 
 \#Y'({\bf r})= p^{\sum_{1\leq i<j\leq \ell}(2m-r_j)},  
\end{align*}
where $\partial:=\dim_{{\mathbb F}_p} L_0^{(\alpha)}/L_0^{(\alpha+1)}$. 
\end{lemma}
\begin{proof} Consider three neighboring lattices $L_{0}^{(t+2)}\subset L_{0}^{(t+1)}\subset L_0^{(t)}$ ($t\in \Z)$ in $W$. Since $pL_0^{(j)}=L_0^{(j+2)}$, the sequence $j\mapsto \partial_{j}:=\dim_{{\mathbb F}_p} L_0^{(j)}/L_0^{(j+1)}$ is $2$-periodic. Since $[L^{(t)}_0:L_{0}^{(t+2)}]=\#({\mathbb F}_p^{n_0})=p^{n_0}$ and $[L_0^{(t)}:L_0^{(t+1)}]=p^{\partial_t}$, we have $[L_0^{(t+1)}:L^{(t+2)}_0]=p^{n_0-\partial_t}$, which implies $\partial_{t+1}=n_0-\partial_t$. From these observations, we have the formula for $\#(Z'({\bf r}))$ easily. The formula for $\#(Y'({\bf r}))$ is evident. \end{proof}

\begin{proposition} \label{IdxFtSrf2-T1-L6} For $m\in \Z_{\geq 0}$, we have
\begin{align*}
\# (X_L'(m)/U')=\sum_{0\leq r_2\leq \dots \leq r_{\ell}\leq 2m+1}p^{\varrho({\bf r})+A\sigma({\bf r})+B' \varepsilon({\bf r})}\,\#(R({\bf r})),
\end{align*}
where $\varrho({\bf r}):=\sum_{j=1}^{\ell}(\ell-j+1)r_j$, ${\sigma}({\bf r}):=\sum_{j=1}^{\ell} r_j$ and $B':=\partial-\frac{n_0}{2}$.     
\end{proposition}
\begin{proof} By Lemmas \ref{IdxFtSrf2-T1-L2}, \ref{IdxFtSrf2-T1-L3}, and \ref{IdxFtSrf2-T1-L5},  
\begin{align*}
\#(X_L'(m)/U')&=\#(X_L'(m)\cap P/U'\cap P) \\
&=\sum_{C\in R({\bf r})} \#(\bar F'(C,\varpi^m)) \\
&=\#(\bar F'({\rm diag}(p^{r_1},\dots p^{r_\ell}),\varpi^{m}) \times 
\#(R({\bf r}))=\#(Z'({\bf r}))\times \#(Y'({\bf r})) \times \#(R({\bf r})).     
\end{align*}
Define $\widetilde {\bf r}=(\tilde r_j)_{j=1}^{\ell}$ as $\tilde r_j:=2m-r_{\ell-j+1}\,(j\in [1,\ell])$. Then, $0\leq \tilde r_1\leq  \dots \leq \tilde r_{\ell}\leq 2m+1$, and 
\begin{align*}
 &\varepsilon(\widetilde{\bf r})=\varepsilon({\bf r}), \\
&\sum_{i=1}^{\ell}(2m-r_i)=\sum_{i=1}^{\ell}\tilde r_i=\sigma(\widetilde{\bf r}), \qquad  \sum_{1\leq i<j \leq \ell} (2m-r_j)=\sum_{1\leq i<j\leq \ell}\tilde r_{i}=\varrho(\widetilde {\bf r}).
\end{align*}
Having these, by Lemma \ref{IdxFtSrf2-T1-L5}, we get
$$
\#(X'_L(m)/U')=\sum_{{\bf r}} \#(Z'({\bf r}))\,\#(Y'({\bf r}))\,\#(R({\bf r)})=\sum_{0\leq \tilde r_1\leq \dots \leq \tilde r_{\ell} \leq 2m+1} p^{(\frac{n_0}{2}-1)\sigma(\widetilde{\bf r})+(\partial-\frac{n_0}{2})\varepsilon(\widetilde{\bf r})+\varrho(\widetilde{\bf r})}\#(R({\bf r})). 
$$
To complete the proof, it suffices to confirm the equality $\#(R({\bf r}))=\# (R(\widetilde {\bf r}))$. Set ${\bf r}'=(r_{2\ell-j+1})_{j}$ and $\underline{(2m)}:=(2m,\dots,2m)$ so that $\widetilde r={\underline{(2m)}}-{\bf r}'$. Then $D(\widetilde {\bf r})=p^{2m}\,D({\bf r}')=p^{2m}w_0 D({\bf r}) w_0^{-1}$, where $D({\bf r})$ denotes ${\rm diag}(p^{r_1},\dots,p^{r_\ell})$ and $w_0\in {\bf GL}_{\ell}(O)$ is the longest permutation matrix of $\ell$ letters. Thus ${\mathfrak R}(\widetilde {\bf r})=p^{2m}\,{\mathfrak R}({\bf r}')$. \end{proof}

For any subset $I\subset [1,\ell]$, define two formal power series $\boldsymbol\zeta_I^{(i)}(t_1,\dots,t_\ell;T)\in \Z[[t_1,\dots, t_\ell]][[T]]$ ($i=0,1)$ by 
$$
\boldsymbol\zeta_I^{(i)}(t_1,\dots,t_\ell; T):=\sum_{m=0}^\infty \Bigl(\sum_{\substack{0\leq r_1\leq \dots \leq r_{\ell}\leq 2m+i \\ s({\bf r})=I}} t_1^{r_1}\dots t^{r_\ell}_\ell\#(R({\bf r})\Bigr)\,T^{m} \qquad (i=0,1). 
$$
(Remark: The same notation is used in \cite{HS1983} for a different object.) Let $B:=-\partial+\frac{n_0}{2}$ for $L$ and $B':=-\partial'+\frac{n_0}{2}$ for $L^{[+1]}$ be as in the table \eqref{TableAB}. Note that $\partial+\partial'=n_0$, so that $B'=\partial-\frac{n_0}{2}=-B$.

\begin{proposition} \label{ADR-L2}
We have that $\boldsymbol\zeta_{\beta,L}^{(0)}(T)$ and $\boldsymbol\zeta_{\beta,L}^{(1)}(T)$ are obtained from 
$$
\sum_{I\subset [1,\ell]} p^{B\#(I)}\boldsymbol\zeta_{I}^{(1)}(t_1,\dots, t_\ell;T), \qquad 
\sum_{I\subset [1,\ell]} p^{B'\#(I)}\boldsymbol\zeta_{I}^{(1)}(t_1,\dots, t_\ell;T)
$$
respectively, by the specialization $t_j=p^{A+\ell-j+1}\,(j \in [1,\ell])$ with $A:=\frac{n_0}{2}-1$.     
\end{proposition}
\begin{proof}
The statement for our $\boldsymbol\zeta_{\beta,L}^{(0)}(T)$ is essentially proved and is stated on \cite[Pg.147(line 7)]{HS1983}. The formula for $\boldsymbol\zeta_{\beta,L}^{(1)}(T)$ is a direct consequence of Proposition \ref{IdxFtSrf2-T1-L6}, because the specialization of $t_1^{r_1}\cdots t_\ell^{r_\ell}$ at $t_j=p^{A+\ell-j+1}$ turns out to be $p^{(\frac{n_0}{2}-1) \sigma({\bf r})+\varrho({\bf r})}$.    
\end{proof}
The power series 
$$
{\bf f}(t_1,\dots,t_\ell):=\sum_{0\leq r_1\leq \dots \leq r_\ell} t_1^{r_1}\cdots t_\ell^{r_\ell}\,\#(R({\bf r}))\quad (\in \Z[[t_1,\dots,t_\ell]])
$$
is the multi-index function series of ${\bf GL}_\ell(F)$ studied by Andrianov. By an easy calculation,
\begin{align}
\frac{{\bf f}(t_1,\dots,t_{\ell-1}, t_\ell T)}{1-T}=\sum_{m=0}^{\infty} \Bigl(\sum_{0\leq r_1\leq \dots \leq r_\ell\leq m} t_1^{r_1}\cdots t_\ell^{r_\ell} \#(R({\bf r}))\Bigr)\,T^{m}.
 \label{ADR-f}
\end{align}
For a power series $g(t)=g(t_1,\dots,t_\ell)\in \Z[[T]][[t_1,\dots,t_\ell]]$ and a subset $I\subset [1,\ell]$, define $g_{I}(t)\in \Z[[T]][[t_1,\dots,t_\ell]]$ as
$$
g_I(t):=2^{-\ell}\sum_{\substack{I_0 \subset I \\
I_1 \subset [1,\ell]-I}} (-1)^{\# I_0} g(t_{I_0 \cup I_1}')
$$
where $t_{J}'=(t_1', \dots,t_\ell')$ with $t_j'=-t_j$ if $j\in J$ and $t_j':=t_j$ if $j\not\in J$. By \cite[Lemma 5]{HS1983}, for a monomial $g_{\bf r}(t):=t_1^{r_1}\cdots t_\ell^{r_\ell}$ with ${\bf r}=(r_1,\dots,r_\ell)$ and for any $I\subset [1,\ell]$, 
$$
(g_{\bf r})_{I}(t)=\begin{cases} g_{\bf r}(t)\quad &(s({\bf r})=I), \\
0 \quad &(s({\bf r})\not=I).
\end{cases}
$$
From this and \eqref{ADR-f}, we get
\begin{align}
\frac{{\bf f}_I(t_1,\dots,t_{\ell-1}, t_\ell T)}{1-T}=\sum_{m=0}^{\infty} \Bigl(\sum_{\substack {0\leq r_1\leq \dots \leq r_\ell\leq m \\ s({\bf r})=I}} t_1^{r_1}\cdots t_\ell^{r_\ell} \#(R({\bf r}))\Bigr)\,T^{m},
 \label{ADR-fI}
\end{align}
which means that $\boldsymbol\zeta_{I}^{(1)}(t_1,\dots, t_\ell:T^2)T$ is the odd part of this formal series in $T$. 

\begin{definition} \label{fdHK-Def}
Fix an integer $\ell \in \Z_{>0}$. For $H\subset [0,\ell-1]$,$K\subset [1,\ell-1]$, their symmetric difference is defined as $H\triangle K:=(H\cup K)-(H\cap K)$. Let $H\triangle K=\{j_1<\dots<j_t\}$; then, define 
$$
\fd(H,K)=\bigcup_{\substack{1\leq \nu \leq t \\ \nu \equiv 1\pmod{2}}}[j_\nu+1, j_{\nu+1}] \quad \text{with $j_{t+1}:=\ell$}. 
$$ 
\end{definition}
Note that $\fd(H,K) \subset [1,\ell]$.

\begin{lemma} \label{ADR-L1}
For any $I\subset [1,\ell]$, 
\begin{align*}
&{\bf f}_{I}(t_1,\dots,t_{\ell-1}, t_\ell T)\times \prod_{r=0}^{\ell-1}(1- p^{2r(\ell-r)}(t_{r+1}\cdots t_\ell)^2 T^2) 
\\
&=\sum_{\substack{K \subset [1,\ell-1] \\ H\subset [0,\ell-1]}} w_{\ell,K}(p^{-1}) p^{a_\ell(K)+a_\ell(H)+A(b_\ell(H)+b_\ell(K))} \delta_{I,\fd(H,K)}\,(p^A T)^{\#(K)+\#(H)}   
\end{align*}
and 
$$\prod_{r=0}^{\ell-1}(1- p^{2r(\ell-r)}(t_{r+1}\cdots t_\ell)^2 T^2)=
\prod_{r=1}^{\ell}(1- p^{2(A+\ell-\frac{r-1}{2})\ell}T^2)
$$
when $t_j=p^{A+\ell-j+1}\,(j \in [1,\ell])$.     
\end{lemma}
\begin{proof} Recall the notation from \eqref{sec:NumIDxFtf1}. For $i\in [0,\ell-1]$, $z_i:=p^{i(\ell-i)} t_{i+1}\cdots t_{\ell}$. 
By \cite[Proposition 2]{HS1983}, 
$$
{\bf f}(t_1,\dots, t_\ell)\times \prod_{r=0}^{\ell-1}(1-z_r)
=\sum_{I\subset [1,\ell-1]} \varphi_{(I)}(p^{-1})\,\prod_{i\in I} z_i \prod_{j \in [1,\ell-1]-I}(1-z_j). 
$$
The product on the right-hand side expands to 
\begin{align*}
\sum_{K\subset [1,\ell-1]}w_{\ell,K}(p^{-1}) \prod_{k\in K}  z_k.   
\end{align*}
We also have $ \prod_{r=0}^{\ell-1}(1+z_r)=\sum_{H\subset [0,\ell-1]} \prod_{h\in H} z_h$. Hence, 
\begin{align*}
{\bf f}(t)\times  \prod_{r=0}^{\ell-1}(1- z_r^2 )=\sum_{\substack{K\subset[1,\ell-1] \\ H\subset[0,\ell-1]}} w_{\ell,K}(p^{-1}) \prod_{h\in H}z_{h} \prod_{k \in K} z_k.     
\end{align*}
The monomial $z_j$ substituted with $t=t'_{I_0\cup I_1}$ ($I_0\subset I,I_1\subset [1,\ell]-I_0)$ equals $(-1)^{\# (I_0\cup I_1)\cap [j+1,\ell]} \,z_j$. If $G(t)\in \C[[t_1,\dots,t_\ell]]$ is even in each variables $t_j\,(j\in [1,\ell])$, then obviously
$$
(FG)_{I}(t)=F_I(t)\,G(t), \quad I\subset [1,\ell].
$$
Noting these observations, we get  
\begin{align*}
{\bf f}_{I}(t)\times  \prod_{r=1}^{\ell}(1- z_r^2)=2^{-\ell}\sum_{\substack{I_0\subset I \\I_1\subset [1,\ell]-I_0}}(-1)^{\#(I_0)}
\sum_{\substack{K\subset[1,\ell-1] \\ H\subset[0,\ell-1]}} w_{\ell,K}(p^{-1}) \varepsilon_{I_0\cup I_1, K}\varepsilon_{I_0\cup I_1, H}
\prod_{h\in H}z_{h} \prod_{k \in K} z_k,    
\end{align*}
where $\varepsilon_{J,H}:=(-1)^{\#\{(h,j)\mid j\in J, h \in H, h<j\}}$. We compute the sum
\begin{align*}
\Delta(I;H,K):=2^{-\ell}\sum_{\substack{I_0\subset I \\I_1\subset [1,\ell]-I_0}}(-1)^{\#(I_0)}
\varepsilon_{I_0\cup I_1, K}\varepsilon_{I_0\cup I_1, H}.    
\end{align*}
For any $J\subset [1,\ell]$, we have
\begin{align*}
\varepsilon_{J,K}\varepsilon_{J,H}
&=\prod_{h\in H} \varepsilon_{J,\{h\}} \prod_{k \in K}\varepsilon_{J, \{k\}}=\prod_{i \in H\triangle K}\varepsilon_{J,\{i\}}= \varepsilon_{J,H\triangle K}
=\prod_{j \in J}(-1)^{r(j)}   
\end{align*}
with ${\bf r}=(r(j))_{j}$, $r(j):=\#(H\triangle K\cap [1,j-1])$ for $j\in [1,\ell]$. Then, we apply \cite[Lemma 5]{HS1983} to have
\begin{align*}
\Delta(I,H,K)=2^{-\ell}\sum_{\substack{I_0\subset I \\I_1\subset [1,\ell]-I_0}}(-1)^{\#(I_0)} \prod_{j\in I_0\cup I_1}(-1)^{r(j)}= \delta_{I,s({\bf r})}.    
\end{align*}
It is easy to see $s({\bf r})=\fd(H,K)$. Hence,
$\Delta(I;H,K)=\delta_{I,\fd(H,K)}$. Therefore, 
\begin{align*}
{\bf f}_I(t)\times \prod_{r=0}^{\ell-1}(1-z_r^2)&=
\sum_{\substack{K\subset[1,\ell-1] \\ H\subset[0,\ell-1]}} w_{\ell,K}(p^{-1}) 
\Delta(I;H,K) \prod_{h\in H}z_{h} \prod_{k \in K} z_k
\\
&=
\sum_{\substack{K\subset[1,\ell-1] \\ H\subset[0,\ell-1]}} w_{\ell,K}(p^{-1}) \delta_{I,\fd(H,K)}\prod_{h\in H}z_{h} \prod_{k \in K} z_k. 
\end{align*}
To complete the proof it remains to note that the specialization at $t_j=p^{A+\ell-j+1}\,(j\in [1,\ell-1]),\,t_\ell=p^{A+1}T$ of $\prod_{h\in H} z_j \prod_{k \in K} z_k$ equals 
$$p^{a_\ell^{(-1)}(K)+a_\ell^{(-1)}(H)+A(b_\ell(H)+b_\ell(K))+A(\#(K)+\#(H))}T^{\#(H)+\#(K)}.$$ \end{proof}
Define $P_{\ell,I}(q,u;T)\in \Z[q,q^{-1},u][T]$ as
\begin{align}
P_{\ell,I}(q,u; T):=\sum_{\substack{K\subset [1,\ell-1] \\ H\subset [0,\ell-1]}}\delta_{I,\fd(H,K)}\, w_{\ell,K}(q)q^{-a_\ell^{(-1)}(K)-a_\ell^{(-1)}(H)} u^{b_\ell(H)+b_\ell(K)}\,T^{\#(H)+\#(K)}.   
\label{PlIquT-Def}
\end{align}
Then, by \eqref{ADR-fI} and Lemma \ref{ADR-L1}, 
\begin{align*}
\sum_{m=0}^{\infty} \Bigl(\sum_{\substack {0\leq r_1\leq \dots \leq r_\ell\leq m \\ s({\bf r})=I}} t_1^{r_1}\cdots t_\ell^{r_\ell} \#(R({\bf r}))\Bigr)\,T^{m} 
=\frac{1+T}{\prod_{r=0}^{\ell} (1-p^{2(A+\ell-\frac{r-1}{2})r}T^2)}\times P_{\ell,I}(p^{-1}, p^{A}, p^{A}T).    
\end{align*}
Note that the denominator includes the factor $1-T^2=(1+T)(1-T)$. The left hand side equals $\boldsymbol\zeta_{I}^{(0)}(t_1,\dots,t_\ell; T^2)+\boldsymbol\zeta_{I}^{(1)}(t_1,\dots,t_\ell;T^2)T$. By comparing the even and the odd part, we obtain the following. 
\begin{proposition} \label{ADR-Main}
Let $t_j=p^{A+\ell-j+1}\,(j \in [1,\ell])$. Then, 
\begin{align*}
\boldsymbol\zeta_{I}^{(0)}(t_1,\dots, t_\ell; T^2)&=
\prod_{r=0}^{\ell} (1-p^{2(A+\ell-\frac{r-1}{2})r}T^2)^{-1} \times \{P^{\rm even}_{\ell,I}(p^{-1}, p^{A}, p^AT)+P^{\rm odd}_{\ell,I}(p^{-1},p^{A},p^{A}T)T\},\\
\boldsymbol\zeta_{I}^{(1)}(t_1,\dots, t_\ell; T^2) T&= \prod_{r=0}^{\ell} (1-p^{2(A+\ell-\frac{r-1}{2})r}T^2)^{-1} \times \{P_{\ell,I}^{\rm odd}(p^{-1},p^A,p^A T)+P^{\rm even}_{\ell,I}(p^{-1}, p^A, p^A T)T\}.  
\end{align*}
\end{proposition}

In the same way as in Lemma \ref{ML-L5}, we have the integral representation: 
\begin{lemma} \label{Remarkf2}
$$
\int_{G^{++}}|\det g|_p^{s}\,d\mu'(g)=\boldsymbol\zeta_{\beta,L}^{(1)}(p^{-ns}), \quad \Re(s)\gg 0,
$$
where $G^{++}:=\{g\in G\mid g(L^{[+1]})\subset L\}$ than $G^{+}$ and $\mu'$ the Haar measure on $G$ normalized by $\mu'(U')=1$.
\end{lemma}
The measure $\mu'$ is different from $\mu$ in Lemma \ref{ML-L5}. With an element $\delta_0\in {\bf GL}_{n}(F)$ such that $L^{[+1]}=\delta_0(L)$, we have $G^{++}=G^{+}\delta_0^{-1}\cap G$. 

\subsubsection{{\underline{Proof of Theorem \ref{IdxFtSrf2-T1}: step 2}}}\label{sec:PrpPlI}
In this step, we make a link between $P_{\ell,I}(q,u;T)$ and ${\mathscr U}_{\ell,J}(X,u;T)$ to complete the proof of Theorem \ref{IdxFtSrf2-T1} and prove Proposition \ref{IdxFtSrf2-T2} on the way.

For any $H\subset [0,\ell-1]$, set $H^\dagger=[0,\ell-1]-H$. For $K\subset [1,\ell-1]$, set $K':=[1,\ell-1]-K$, so that $K^{\dagger}=\{0\}\cup K'$.
\begin{lemma} \label{PrpPlI-L1}
     Let $H\subset [0,\ell-1]$ and $K\subset [1,\ell-1]$. Then, $$\fd(H^\dagger, K')=[1,\ell]-\fd(H, K).$$
\end{lemma}
\begin{proof} It is easy to see that 
$$
H^\dagger \triangle K'=\begin{cases} (H\triangle K)\cup\{0\} \quad &(0 \not \in H), \\
(H\triangle K)-\{0\} \quad &(0 \in H).
\end{cases}
$$
Let $j_1<\dots<j_t$ be the enumeration of elements of $H\triangle K$. Then, since $0\not\in K$, we have that $0\in H$ if and only if $j_1=0$. 
$$
H^\dagger \triangle K'=\begin{cases} \{0,j_1,\dots,j_{t}\}, \quad &(0\not\in H), \\
\{j_2, \dots,j_{t}\} \quad &(0 \in H).
\end{cases}
$$
Having this observation, the claim follows from Definition \ref{fdHK-Def}.
\end{proof}

\begin{proposition} \label{PrpPlI-L2}
For any $I\subset [1,\ell]$, we have the following properties of $P_{\ell,I}(q,u,T)$: 
\begin{itemize}
\item[(i)] The functional equation
\begin{align*}
P_{\ell,[1,\ell]-I}(q^{-1},u^{-1},T^{-1})=q^{\frac{2}{3}\ell(\ell-1)(\ell+1)+\ell}u^{-(\ell-1)^2} T^{1-2\ell}\,P_{\ell,I}(q,u,T).
\end{align*}
\item[(ii)] We have
$$
\sum_{I\subset [1,\ell]}P_{\ell, I}(q,u,T)= \prod_{r=0}^{\ell-1}(1+q^{\frac{r(r+1)}{2}-\frac{\ell(\ell+1)}{2}} u^{\ell-r-1}T) \times P_{\ell}(q,u,T). 
$$
\end{itemize}
\end{proposition}
\begin{proof}
(i)  For $H\subset [0,\ell-1]$ and $K\subset [1,\ell-1]$, we easily confirm the relation $H^\dagger \triangle K^\dagger=H\triangle K$. By \eqref{aK-Def} and \eqref{bK-Def}, 
\begin{align*}
&a_\ell^{(-1)}(K)+a_\ell^{(-1)}(K')=\tfrac{\ell(\ell+1)(\ell-1)}{3},\quad a_\ell^{(-1)}(H)+a_\ell^{(-1)}(H^\dagger)=\tfrac{\ell(\ell-1)(\ell+1)}{3}+\tfrac{\ell(\ell+1)}{2}, \\
&b_\ell(K)+b_\ell(K')=\tfrac{(\ell-1)(\ell-2)}{2}, \quad b_\ell(H)+b_\ell(H^\dagger)=\tfrac{\ell(\ell-1)}{2}. 
\end{align*}
Using these and \eqref{PlIquT-Def}, together with Lemma \ref{PrpPlI-L1} and \ref{EFPell-f1} on the way, we have\begin{align*}
&P_{\ell,[1,\ell]-I}(q^{-1},u^{-1},T^{-1})
\\
&=\sum_{\substack{K\subset [1,\ell-1] \\ H\subset [0,\ell-1]}}  \delta_{[1,\ell]-I,\fd(H,K)}\,w_{\ell,K}(q^{-1})q^{a_\ell^{(-1)}(K)+a_\ell^{(-1)}(H)} u^{-b_\ell(H)-b_\ell(K)}\,T^{-\#(H)-\#(K)}\\
&=\sum_{\substack{K\subset [1,\ell-1] \\ H\subset [0,\ell-1]}}  \delta_{[1,\ell]-I,[1,\ell]-\fd(H^\dagger,K')} w_{\ell,K'}(q)q^{-\frac{\ell(\ell-1)}{2}}\,q^{\frac{2}{3}\ell(\ell+1)(\ell-1)+\frac{\ell(\ell+1)}{2}-a_\ell^{(-1)}(K')-a_\ell^{(-1)}(H^\dagger)}\\
&\qquad \quad \times u^{-\frac{\ell(\ell-1)}{2}+b_\ell(H^\dagger)-\frac{1}{2}(\ell-1)(\ell-2)+b_\ell(K)}\,T^{1-2l+\#(H^\dagger)+\#(K')}
\\
&=q^{\frac{2}{3}\ell(\ell-1)(\ell+1)+\frac{1}{2}\ell(\ell+1)-\frac{1}{2}\ell(\ell-1)}u^{-(\ell-1)^2} T^{1-2\ell} \\
&\qquad \quad \times \sum_{\substack{K\subset [1,\ell-1] \\ H\subset [0,\ell-1]}}  \delta_{I,\fd(H^\dagger,K')} w_{\ell,K'}(q)
\,q^{-a_\ell^{(-1)}(K')-a_\ell^{(-1)}(H^\dagger)}
u^{b_\ell(H^\dagger)+b_\ell(K')}\,T^{\#(H^\dagger)+\#(K')}. 
\end{align*}
Since $(H,K)\mapsto (H^\dagger,K')$ is an involution of the index set, the last sum is $P_{\ell,I}(q,u,T)$.

(ii) By \eqref{PlIquT-Def} and \eqref{Numerator}, the sum $
\sum_{I\subset [1,\ell]} P_{\ell,I}(q,u,T)$ equals  
\begin{align*}& \sum_{\substack{K\subset [1,\ell-1] \\ H\subset [0,\ell-1]}} w_{\ell,K}(q)
\,q^{-a_\ell^{(-1)}(K)-a_\ell^{(-1)}(H)}
u^{b_\ell(H)+b_\ell(K)}\,T^{\#(H)+\#(K)} \\
&=\Bigl(\sum_{H \subset [0,\ell-1]} q^{-a_\ell^{(-1)}(H)} u^{b_\ell(H)} T^{\#(H)}\Bigr)
\times \Bigl( \sum_{K\subset [1,\ell-1]} w_{\ell,K}(q)q^{-a_\ell^{(-1)}(K)} u^{b_\ell(K)} T^{\#(K)}\Bigr) \\
&=\prod_{h=0}^{\ell-1}\bigl(1+q^{-a_{\ell}^{(-1)}(\{h\})} u^{b_\ell(\{h\})} T\bigr)\times P_{\ell}(q,,u,T). 
\end{align*}
\end{proof}

\begin{definition}\label{ftntl-Def}
\begin{itemize}
    \item[(a)] set $\emptyset ^{\flat}:=\emptyset$. For $\emptyset \not=I\subset [1,\ell]$, let $[a_\nu,b_\nu]\,(1\leq \nu\leq t)$ be the set of maximal intervals of $\Z$ contained in $I$; then define $I^\flat:=\{a_\nu-1,b_\nu\mid \nu\in [1,t]\}\cap [0,\ell-1]$. 
    \item[(b)] Set $\emptyset ^\natural:=\emptyset$. For $\emptyset\not=H=\{h_1<\dots < h_{t}\} \subset[0,\ell-1]$, let $H^\natural$ be the union of intervals $[h_{\nu}+1,h_{\nu+1}]$ with $h_{t+1}:=\ell$ for $1\leq \nu \leq t$, $\nu \equiv 1 \pmod{2}$. 
\end{itemize}
\end{definition}
The following is clear from Definitions \ref{ftntl-Def} and \ref{fdHK-Def}. 
\begin{lemma} \label{PrpPlI-L3}
The map $I\mapsto I^\flat$ is a bijection from the power set of $[1,\ell]$ onto the power set of $[0,\ell-1]$, whose inverse is $H\mapsto H^\natural$. We have $\#(I)=\langle I^\flat \rangle$. For $H\subset [0,\ell-1]$ and $K\subset [1,\ell-1]$, we have $\fd(H,K)=(H\triangle K)^\sharp$ and $H\triangle K=\fd(H,K)^\flat$. 
\end{lemma}

\begin{proposition} \label{PrpPlI-L4}
For any $I\subset [1,\ell]$, we have the formula
\begin{align}
P_{\ell,I}(q,u,T)&=T^{\#(I^\flat)}\,{\mathscr U}_{\ell, I^\flat}(q^{-1}, u;T). \label{PrpPlI-L4-f0}
\end{align}
\end{proposition}

\begin{proof} We may set $H=I_2\cup I_1$, $K=I_0\cup I_1$ with three mutually disjoint sets $I_0,I_1\subset [1,\ell-1]$, $I_2\subset [0,\ell-1]$; then $H\triangle K=I_0\cup I_2$. In the summand of \eqref{PlIquT-Def}, the delta-symbol $\delta_{I,\fd(H,K)}$ is $1$ if and only if $I=\fd(H,K)$, which happens if and only if $H\triangle K=I^{\flat}$ by the previous lemma. Thus, we may write the sum in \eqref{PlIquT-Def} as a sum over $(I_0,I_1,I_2)$ such that $I_0\cup I_2=I^\flat$, which in turn becomes a sum over $I_0\subset I^\flat\cap [1,\ell-1]$ and $I_1\subset [1,\ell-1]-I^{\flat}$ as claimed.  
\end{proof}

\subsubsection{{\underline{Completion of proof}}} \label{sec:PROOFofT1}
Theorem \ref{IdxFtSrf2-T1} follows from Proposition \ref{ADR-Main}, \ref{PrpPlI-L4} and Proposition \ref{ADR-L2} readily by noting that $B+B'=0$ from table \eqref{TableAB} because the parity of ${\rm ord}_{p}\mu(L)$ and ${\rm ord}_{p}(\mu(L^{[+1]}))$ are different. The relation $B+B'=0$ is also proved from $\partial_{t}+\partial_{t+1}=n_0$ noted in the proof of Lemma \ref{IdxFtSrf2-T1-L5}.

Let us prove Proposition \ref{IdxFtSrf2-T2}. By Proposition \ref{PrpPlI-L4}, we have 
\begin{align}
U_{\ell,B}^{\bullet}(X,u;T)=\Bigl(\sum_{I\subset [1,\ell]}X^{B\#(I)} P_{\ell,I}(X^{-1},u;T)\Bigr)^{\bullet}
\label{UPexp}
\end{align}
for $\bullet\in \{{\rm even}, {\rm odd}\}$, from \eqref{WlBXuT-Def} by changing the running index in \eqref{WlBXuT-Def} from $I$ to $J=I^\flat$ using Lemma \ref{PrpPlI-L3}. Then, the functional equations are immediate from this expression and Proposition \ref{PrpPlI-L2}(i). From the above expression and \eqref{PlIquT-Def}, we see that $U_{\ell,B}^{\bullet}(X,X^{A}, X^{A}T)$ is a sum of terms $w_{\ell,K}(X^{-1}) X^{j(H,K)}T^{\#(H)+\#(K)}$ for $H\subset [0,\ell-1]$, $K\subset [1,\ell-1]$ with 
\begin{align}
j(H,K):=B\#(\fd(H,K))+a_\ell^{(-1)}(H)+a_{\ell}^{(-1)}(K)+A(b_\ell(H)+b_\ell(K))+A(\#(H)+\#(K)).
 \label{jHK}
 \end{align}
Recall that $w_{\ell, K}(X^{-1})$ is a polynomial in $X^{-1}$ with positive integer coefficients, whose degree is $N(K')-\frac{\ell(\ell-1)}{2}$ by Lemma \ref{ML-L9} and Proposition \ref{EFPell-P1}. It suffices to confirm that $j(H,K)\in \Z$ and $j(H,K)-{N}(K')+\frac{\ell(\ell-1)}{2} \geq 0$. We may set $H=I_2\cup I_1$, $K=I_0\cup I_1$ so that $H\triangle K= I_0 \cup I_2$ as in the proof of Proposition \ref{PrpPlI-L4}. If we set $H\triangle K=\{i_1<\dots<i_t\}$ and $i_{t+1}=\ell$, then by \eqref{bracketJ}, $\#(\fd(H,K))\leq \sum_{\nu=1}^{t}(-1)^{\nu}(\ell-i_\nu)$, i.e., 
\begin{align}
\#(\fd(H,K))\leq b_{\ell}(H\triangle K)+\#(H\triangle K).
\label{dHKb}
\end{align}
We also have the congruence $\# \fd(H,K) \equiv b_{\ell}(H\triangle K)+\#(H\triangle K)\pmod{2}$. Noting that $2A,2B\in \Z$ and $A+B,A-B \in \Z$ from table \eqref{TableAB}, we have the following mod $2$ relation:  
\begin{align*}
2j(H,K)&\equiv 2B b_{\ell}(H\triangle K)+2B\#(H\triangle K)+2A( b_\ell(H)+b_\ell(K))+2A(\#(H)+\#(K)) \\
&\equiv 2B(b_{\ell}(I_0)+b_\ell(I_2))+2B(\#(I_0)+\#(I_2))
+2A(b_{\ell}(I_0)+b_\ell(I_2)+2b_\ell(I_1)) \\
&\qquad +2A(\#(I_2)+\#(I_0)) \equiv 0.
\end{align*}
This shows $j(H,K)\in \Z$. By \eqref{Nell}, $j(H,K)-N(K')+\tfrac{\ell(\ell-1)}{2}\geq j(H,K)$. Set $B_-:=\min(0,B)$. Then, by \eqref{jHK}, \eqref{Nell} and $a_\ell(J)\geq b_{\ell}(J)+\#(J)$ for all $J\subset [0,\ell-1]$, we have
\begin{align*}
&j(H,K)\\
&\geq B_{-}(b_\ell(H\triangle K)+\#(H\triangle K))+a_\ell(K)+a_\ell(H)+A(b_\ell(H)+b_\ell(K))+A(\#(H)+\#(K))\\
&\geq  B_{-}(b_\ell(H\triangle K)+\#(H\triangle K))+(b_\ell(K)+\#(K))+b_\ell(H)+\#(H) \\
&\qquad +A(b_\ell(H)+b_\ell(K))+A(\#(H)+\#(K)).
\end{align*}
It is easy to check that the last formula is non-negative for $(B_-,A)=(-1/2,-1/2)$ and $(B_-,A)=(-1,0)$. By Table \eqref{TableAB}, this is sufficient to imply the non-negativity of $j(H,K)-N(K')+\frac{\ell(\ell-1)}{2}$. The formula \eqref{WellB0} follows from Proposition \ref{PrpPlI-L2}(ii) and Proposition \ref{PrpPlI-L4} readily. 

\subsection{Examples} \label{sec:OddORTExample}
We have the following explicit formula of $U_{\ell,B}^{\bullet}(X,u,T)$ for $\ell=1,2,3$ computed by \eqref{mathcalWellJ-Def} and \eqref{WlBXuT-Def}: 
\begin{align*}
\begin{array}{|c||c|c|}
\hline
\ell & U^{{\rm even}}_{\ell,B}(X,u;T) & U^{\rm odd}_{\ell,B}(X,u;T) \\
\hline
1 & 1 & X^{B+1} T \\
\hline 
2 & 1+\{X^3+X^{4+B}(1+X)u\}T^2 & \{X^{B+1}(1+X)+X^{2B+3}u\}T+X^{2B+6}u T^3\\
\hline
\end{array}
\end{align*}
and for $\ell=3$, 
\begin{align*}
U^{\rm even}_{3,B}(X,u;T)
&=1+\Bigl\{X^4(1+X)(1+u^2X^4)+
u^{2}X^{B+7}(1+X)(uX^2+X^{B}) \\
& \quad +u^{2}X^{B+9}(uX^2+X^B)
+uX^{B+5}(1+X)(1+X+X^2)\Bigr\} T^{2} \\
&\qquad +\Bigl\{u^2X^{13}+u^3 X^{B+14} +u^3X^{B+15}(1+X)+u^{4}X^{2B+16}(1+X+X^2)\Bigr\}T^{4}, \\
U^{\rm odd}_{3,B}(X,u;T)&=\Bigl\{u^2X^{3B+6}+uX^{2B+5}+uX^{2B+3}(1+X)+X^{B+1}(1+X+X^2)\Bigr\} T \\
&\quad +\Bigl\{u^2 X^{3B+10}(1+X)(1+u^2X^4)+u X^{B+9}(1+X)(X^B+uX^{2}) \\
&\quad +uX^{B+8}(X^B+uX^2)
+u^{3}X^{2B+11}(1+X)(1+X+X^2)\Bigr\}T^{3}
+u^{4} X^{3B+19} T^{5}.
\end{align*}
By these formulas we can check the functional equation in Proposition \ref{IdxFtSrf2-T2}. For example, $$U^{\rm even}_{3,B}(X^{-1},u^{-1};T^{-1})=u^{-4}X^{-3B-19}U^{\rm odd}_{3,B}(X,u;T).$$ 
Let $V=F^{2\ell+1}$, the space of column vectors of dimension $2\ell+1$, and 
$$
((2)\oplus \beta_{\ell,\ell})({\bf x},{\bf y}):={}^t {\bf x} \left[\begin{smallmatrix} {} & {} & J_2 \\ {} & {2} & {} \\ 
J_2 & {} & {} \end{smallmatrix}
\right] {\bf y}, \quad {\bf x},{\bf y}\in V.
$$
The standard $O$-lattice $L=O^{2\ell+1} \subset F^{2\ell+1}$ is a maximal $O$-integral lattice; table \eqref{TableAB} tells us  $A=B=\frac{-1}{2}$. By Theorem \ref{IdxFtSrf2-T1}, the series $Z_{(2)\oplus \beta_{\ell,\ell},L}(s)$ for $\ell=1,2$, respectively, are given as 
\begin{align*}   
&{(1-p^{-3s})^{-1}(1-p^{1-3s})^{-1}}(1+(1+p)p^{-2s}+p^{-3s}),\\
&(1-p^{-5s})^{-1}(1-p^{3-5s})^{-1}(1-p^{4-5s})^{-1}\{1+[4]_p p^{-3s}+[3]_p(1+p^2)p^{-5s} 
+[4]_p p^{2-8s}+p^{3-10s}),
\end{align*}
where $[n]_p:=1+p+\dots+p^{n-1}$. 
\section{The \texorpdfstring{${\rm GO}(n)$}{}-case: Euler products}\label{sec:GlobalGO}
In this section, we prove Theorems \ref{THM1}, \ref{Thm2(EvenGO)} and \ref{THM2Odd} in a general setting explained in \S\ref{sec:GlobalCTN}. Theorem \ref{MAINTHM2} for $\varepsilon=+1$ is a special case of Theorem \ref{THM1} as is spelled out in \S\ref{sec:EXPLITZETA} (1) followed by explicit examples. In the first 2 subsections, we settle technical issues related to the convergence of Euler products.

\subsection{Preliminary lemmas}
\label{sec:CerTnWAB}
This subsection is a preliminary to \S\ref{sec:type2Factor}. 
Let $\R_{\geq 0}[X^{ \pm 1/2},Y]$ denote the set of all those $F(X,Y)\in \R[X^{\pm 1/2},Y]$ whose coefficients are non-negative real numbers; it is obvious that $\R_{\geq 0}[X^{\pm 1/2},Y]$ is a cone in the $\R$-vector space $\R[X^{1/2},Y]$. For two elements $F(X,Y)=\sum_{ij}f_{ij}X^iY^j$ and $G(X,Y)=\sum_{ij}g_{ij}X^{i}Y^{j}$ from $\R_{\geq 0}[X^{\pm 1/2},Y]$, we write $F(X,Y)\prec G(X,Y)$ if $f_{ij}\leq g_{ij}$ for all $i,j$. By writing $F(X,Y)=\sum_{d\geq 0}F_d(X)Y^{d}$ with $F_d(X)\in \R[X^{1/2},X^{-1/2}]$, set ${\mathfrak b}_d(F):={\rm deg}(F_d)$ for $d\geq 0$, where the degree of $f(X)=\sum_{n \in \frac{1}{2}\Z_{\geq n_0}} f_n X^{n}\in \R[X^{1/2},X^{-1/2}]-\{0\}$ is defined as ${\rm deg}f(X):=\max\{n\in \frac{1}{2}\Z\mid f_{n}\not=0\}$ with the convention $\max \emptyset =-\infty$. The following is obvious:  
\begin{align}
    (\forall F(X,Y),\,\forall G(X,Y)\in \R_{\geq 0}[X^{\pm 1/2},Y])\,(\forall d\in \Z_{\geq 0})(\, F(X,Y)\prec G(X,Y) \, \Longrightarrow \, \fb_d(F)\leq \fb_d(G)\,). 
\label{CerTnWAB-f0}
\end{align}
Moreover, 
\begin{align}
\fb_d(FG)=\max\{\fb_i(F)+\fb_j(G)\mid i,j\in \Z_{\geq 0}, i+j=d\}.  
\label{CerTnWAB-f1}
\end{align}
In this and the next subsections, $\ell \in \Z_{>0}$ and $(A,B)$ denotes $(0,-1)$ or $(-1/2,-1/2)$; we set $n:=2(\ell+A+1)$. In the following, the negativity of $B$ is important. From Definitions \ref{Intro-Def3} and \ref{Def-AW},  
\begin{align}
\fb_{ne}\bigl(\prod_{r=0}^{\ell-1}(1+p^{\frac{\ell(\ell+1)}{2}-\frac{r(r+1)}{2}+A(\ell-r)}Y^{n})\bigr)&=\max\{a_\ell^{(-1)}(I)+A b_{\ell}(I)+A\#(I) \mid I\subset [0,\ell-1], \,\#(I)=e\}, 
 \label{CerTnWAB-f2}
\\
\fb_{nd}(W_\ell^{(\varepsilon)}(X,Y^n))&=d\times \beta_{\ell,d}^{(\varepsilon)} \quad (d\in [0,\ell-1])
\label{CerTnWAB-f3}
\end{align}
Set
$$
\beta(A;d):=\fb_{nd}(P_\ell(X^{-1},X^A,X^AY^n)). 
$$
Then, 
\begin{lemma} \label{CerTnWAB-L0} For $d\in [0,\ell-1]$, 
$$
\beta(A;d)=\begin{cases} d\beta_{\ell,d}^{(-1)} \quad &(A=0), \\
\tfrac{1}{2} d(\beta_{\ell,d}^{(-1)}+\beta_{\ell,d}^{(+1)})\quad &(A=-1/2).
\end{cases}
$$  
\end{lemma}
\begin{proof}
If $A=0$, then $P_{\ell}(X^{-1},1,T)=W_{\ell}^{(-1)}(X,T)$, so that the value $\beta(0;d)$ is given by \eqref{CerTnWAB-f3}. When $A=-1/2$, then 
$$
P_\ell(X^{-1},X^{-1/2},X^{-1/2}T)=\sum_{K \subset [1,\ell-1]}w_{\ell,K}(X^{-1})X^{a_{\ell}^{(-1)}(K)-\frac{1}{2}b_\ell(K)-\frac{1}{2}\#(K)}T^{\#(K)}.
$$
We have $a_{\ell}^{(-1)}(K)-\tfrac{1}{2}b_\ell(K)-\frac{1}{2}\#(K)=\tfrac{1}{2}(a_{\ell}^{(-1)}(K)+a_{\ell}^{(+1)}(K))$. Hence, by the same argument as in the proof of Theorem \ref{EFPell-P3}, the value $\fb_{nd}(P_{\ell}(X^{-1},X^{-1/2},X^{-1/2}Y^n))$ is identified with the maximum of 
$$
-N(K)+\tfrac{1}{2}(a_{\ell}^{(-1)}(K)+a_{\ell}^{(+1)}(K))
$$
for $K\subset [0,\ell]$, $\#(K)=d$. Since the maximum of $-N(K)+a_{\ell}^{(\varepsilon)}(K)$ is $d\beta_{\ell,d}^{(\varepsilon)}$ attained at $K=[1,d]$ both for $\varepsilon=+1$ and for $\varepsilon=-1$ (see Lemma \ref{EFPell-P3-1}), we have the conclusion. 
\end{proof}

For polynomials in \eqref{WlBXuT-Def}, we have the following inequalities. 
\begin{lemma}\label{CerTnWAB-L1} 
\begin{align*}
& \fb_{nd}(U_{\ell,\pm B}^{\bullet}(X,X^A, X^AY^{n})) \leq \Psi_\pm(A;d)
\end{align*}
with 
\begin{align}
\Psi_\pm (A;d)&:=d\left\{\tfrac{1}{2}\ell(\ell+1)-\tfrac{1}{6}(d^2-10)\right\}+\delta_{\pm 1, -1}\ell \quad (A=0,B=-1), 
 \label{CerTnWAB-L1-f0}
\\
\Psi_{\pm}(A;d)&:=d\left\{\tfrac{1}{2}\ell^2 -\tfrac{1}{12}(d+7)(2d+1)\right\}+\tfrac{1}{2}\delta_{\pm 1,-1} \ell \quad (A=B=-1/2) ,
\notag
\end{align}
where $\bullet\in\{{\rm even},{\rm odd}\}$ being the parity of $d$.     \end{lemma}
\begin{proof} In the defining formula \eqref{PlIquT-Def}, if we neglect the delta symbol $\delta_{I,\fd(H,K)}$, the expression is factored to a product of the polynomials $\sum_{H\subset [0,\ell-1]}q^{-a_\ell^{(-1)}(H)}u^{b_\ell(H)} T^{\#(H)}=\prod_{r=0}^{\ell-1}(1+q^{-a_\ell^{(-1)}(\{r\})}u^{b_\ell(\{r\})}T)$ and $\sum_{K\subset [1,\ell-1]}w_{\ell,K}(q)q^{-a_\ell^{(-1)}(K)} u^{b_\ell(K)}T^{\#(K)}=P_\ell(q,u;T)$. Since all polynomials involved are of positive coefficients, we have 
\begin{align*}
 P_{\ell,I}(X^{-1},X^{A},X^{A}Y^n)\prec \prod_{r=0}^{\ell-1}(1+X^{a_\ell(\{r\})+Ab_\ell(\{r\})+A}Y^{n})\,P_{\ell}^{(-1)}(X^{-1},X^{A}, X^AY^n).
\end{align*}
Combining this with \eqref{UPexp}, we obtain 
\begin{align*}
U_{\ell,-B}^{\bullet}(X,X^{A},X^{A}Y^n) &\prec X^{|B|\ell}\,\sum_{I\subset [0,\ell-1]} P_{\ell,I}(X^{-1},X^{A},X^{A}Y^n)
\\
&\prec X^{|B|\ell} \prod_{r=0}^{\ell-1}(1+X^{a_\ell(\{r\})+Ab_\ell(\{r\})+A}Y^{n})\,P_{\ell}(X^{-1},X^A, X^AY^n).
\end{align*}
For $U_{\ell,B}^{\bullet}$, we may omit $X^{|B|\ell}$. Then, by using \eqref{CerTnWAB-f0}, \eqref{CerTnWAB-f1}, \eqref{CerTnWAB-f2} and \eqref{CerTnWAB-f3}, we have 
\begin{align*}
&\fb_{dn}(U_{\ell,\pm B}^{\bullet}(X,X^{A},X^{A} Y^n))
\\
&\leq \delta_{\pm1, -1} |B|\ell+\fb_{nd}\bigl( \prod_{r=0}^{\ell-1}(1+X^{a_\ell(\{r\})+Ab_\ell(\{r\})+A}Y^{n})\,P_{\ell}(X^{-1},X^A, X^AY^n)\bigr)
\\
&=\delta_{\pm 1,-1}|B|\ell+\max\{a_\ell^{(-1)}(H)+A b_\ell(H)+A\#(H)+d\beta_{\ell,e}^{(-1)} \mid H\subset [0,\ell-1],\,d=e+\#(H)\}=:\Psi(d). 
\end{align*}
It is easy to see that the maximum
$$
\psi(e):=\max\{a_\ell^{(-1)}(H)+A b_\ell(H)+A\#(H)\mid H \subset [0,\ell-1],\,\#(H)=e\}
$$
is attained at $H=[0,e]$ and 
\begin{align*}
    \psi(e)=\tfrac{\ell(\ell+1)}{2}e-\tfrac{1}{6}e(e+1)(e+2)+\tfrac{A}{2}e(2\ell-e+1).    
\end{align*}
Hence, $
\Psi(d)=|B|\delta_{\pm 1, -1}\ell+\max_{e\in[0,d]}\Bigl\{\beta(A;e)+\psi(d-e)\Bigr\}
$. By Lemma \ref{CerTnWAB-L0}, this is the maximum of a quadratic polynomial in $e$, so that it is easy to evaluate the value as in \eqref{CerTnWAB-L1-f0}. 
\end{proof}
Recall the total ASH polynomial in \eqref{WellBtotal-Def}. 
\begin{corollary}\label{CerTnWAB-C1}
Let $d\geq 1$. Then, we have $\fb_{d}(W_{\ell,A,B}^{\rm total}(X,Y))=-\infty$ unless $d\equiv 0\pmod{n}$ and $d\geq 2n$ or $d\equiv -2B \pmod{n}$ and $d\geq n-2B$, so that the minimum value of such $d$ is $n-2B$. We have
$$
\fb_{n-2B}(W_{\ell,A,B}^{\rm total}(X,Y))=\tfrac{1}{2}\ell(\ell+1)+(A-B)\ell. 
$$
For all $e\geq 2$, both $\fb_{en}(W_{\ell,A,B}^{\rm total}(X,Y)))$ and $\fb_{en-2B}(W_{\ell,A,B}^{\rm total}(X,Y))$ are no greater than $\Psi_{-}(A;e)$. We have a better bound $\fb_{2n}(W_{\ell,A,B}^{\rm total}(X,Y))\leq \Psi_{+}(A;e)$.   
\end{corollary}
\begin{proof}
We abbreviate $U_{\ell,B}^{\bullet}(X,X^A,X^AY^n)$ to $U_{\ell,B}^{\bullet}$. Then, 
$$
W_{\ell,A,B}^{\rm total}(X,Y)=U_{\ell,B}^{\rm even}+U_{\ell,-B}^{\rm odd} Y^{-2B}+U_{\ell,-B}^{\rm even}Y^{n-2B}+U_{\ell,B}^{\rm odd} Y^n.
$$
Since $U_{\ell,B}^{\rm even }$ is an even polynomial in $Y^{n}$ and $U_{\ell,B}^{\rm odd}$ is an odd polynomial in $Y^n$, the first assertion is evident from this. Then, the second lowest degree term in $Y$ of $W_{\ell,A,B}^{\rm total}(X,Y)$ is $C(X)\,Y^{n-2B}$, which arises from the constant term $1$ of $U^{\rm even}_{\ell,-B}$ and the lowest term of $U_{\ell,-B}^{\rm odd}\,Y^{-2B}$ in the above expression. Set $\kappa_r:=\frac{r(n-r-1)}{2}-Br$ for $r\in [1,\ell]$. We have
\begin{align*}
C(X)&=1+\sum_{j=0}^{\ell-1}\sum_{K\subset [1,\ell-1]\cap \{j\}}
w_{\ell,K}(X^{-1}) X^{\frac{\ell(\ell+1)}{2}-\frac{j(j+1)}{2}+(A-B)(\ell-j)} 
=1+\Bigl(\sum_{r=1}^{\ell}X^{\kappa_r}\Bigr)\,(1+{\mathcal O}(X^{-1})). 
\end{align*}
The exponent $\kappa_r$ gets its maximum at $r=\ell$, so that $\fb_{n-2B}(W_{\ell,A,B}^{\rm total}(X,Y))=\kappa_{\ell}$.  
To prove the last statement, we only have to note that the term involving $Y^{2n}$ (the third lowest degree in $Y$) only arises from $U_{\ell,B}^{\rm even}$ for which a smaller majorant $\Psi_{+}$ can be invoked. 
\end{proof}

\subsection{Convergence of certain Euler product}\label{sec:type2Factor}
Let $\ell \in \Z_{>0}$. Let ${\mathbb P}$ denote the set of all prime numbers. We are concerned with the convergence region of the Euler product
\begin{align}
{\mathfrak E}_{\ell,A,B}^{\rm total}(s):=\prod_{p \in {\mathbb P}}
{W_{\ell, A,B}^{\rm total}(p,p^{-s/2})}
\label{Type2Factor-f1}
\end{align}for $(A,B)=(0,-1)$, referred to as the even case, and for $(A,B)=(-\frac{1}{2}, -\frac{1}{2})$ referred to as the odd case; these correspond to the unramified cases in Table \eqref{TableAB}. The Euler product \eqref{Type2Factor-f1} covers the almost all Euler factors arising from a global maximal lattice in the next section.

\begin{lemma} \label{type2Factor-L1}
Set $\kappa:=\frac{\ell(n-\ell-1)}{2}-B\ell$. Then, for some $\eta>0$,  
\begin{align*}
(1-p^{ \kappa-(n-2B)s/2})\, {W_{\ell,A,B}^{\rm total}(p,p^{-s/2})}=1+{\mathcal O}(p^{-1-\eta})    
\end{align*}
for all prime numbers $p$ and for all $\Re(s)\geq \frac{2}{n-2B}(\kappa+1-\epsilon)$ for some $\epsilon>0$ with the implied constant independent of $p$.     
\end{lemma}
\begin{proof} Set $W_{\ell,A,B}^{\rm toal}(X,Y)=1+\sum_{d\geq 1}C_d(X)Y^{d}$; then, $C_d(X)=0$ unless $d\equiv 0,-2B \pmod{n}$ and smallest among such $d$ is $n-2B$; moreover, $C_{n-2B}(X)=1+X^{\kappa}+{\mathcal O}(X^{\kappa-1})$ from the proof of Corollary \ref{CerTnWAB-C1}. Thus, by multiplying $(1-X^{\kappa}Y^{n-2B})$, we can cancel the term $X^{\kappa}Y^{n-2B}$, which is the largest degree-in-$X$ term of the smallest degree $n-2B$ in $Y$. Thus, it suffices to show that, for all degree $d>n-2B$ in $Y$, the term $C_{d}(p)p^{-sd/2}$ is bounded by $p^{-1-\epsilon}$ for some $\epsilon>0$ when $(n-2B)\Re(s)/2>\kappa$. We have the bound $C_{d}(p)\ll p^{\fb_{d}}$ for $p$ with $\fb_{d}=\fb_d(W_{\ell,A,B}^{\rm total}(X,Y))$. Hence, by Lemma \ref{CerTnWAB-L1}, it suffices to show that, for some $\epsilon>0$, 
\begin{align*}
\text{$\Psi_{-}(A;e)-\tfrac{d\Re(s)}{2}<-1$ when $(n-2B)\tfrac{\Re(s)}{2}>\kappa+1-\epsilon$.}
\end{align*}
for $d=en,en-2B$ unless $d=2n$, in which case we may replace $\Psi_{-}$ with $\Psi_{+}$.
\begin{itemize}
\item Let $A=0,B=-1$, so that $n=2\ell+2$ and $\kappa=\frac{1}{2}\ell(\ell+3)$. Let $d=en$ with $e>2$. Then
\begin{align*}
\Psi_{-}(0;d)-\tfrac{d\Re(s)}{2}&=e\{\tfrac{\ell(\ell+1)}{2}-\tfrac{1}{6}(e^2-10)-\tfrac{n\Re(s)}{2} \}
\\
&<e\left\{ \tfrac{\ell(\ell+1)}{2}-\tfrac{1}{6}(e^2-10)-\tfrac{n}{2}\left(\tfrac{\ell(\ell+3)}{n+2}+\tfrac{2(1-\epsilon)}{n+2} \right)\right\}+\ell
\\
&=e\left\{ \tfrac{-\ell(\ell+1)}{2(\ell+2)}-\tfrac{1}{6}(e^2-10)-\tfrac{(\ell+1)(1-\epsilon)}{\ell+2}\right\}+\ell<-1.
\end{align*}
When $e=2$, by replacing $\Psi_{-}$ with $\Psi_{+}$, we ends up with the same formula as above without the last summand $\ell$, so that $\Psi_{+}(0;2)-\frac{n\Re(s)}{2}<-1$. Let $d=ne-2B$, Then $\Psi_{-}(0;e)-\frac{d\Re(s)}{2}=(\Psi_{-}(0;e)-\frac{ne \Re(s)}{2})+B\Re(s)$, which means we repeat the same estimation as in the case $d=ne$ with an extra negative term $B\Re(s)$. Hence, $\Psi_{-}(0;e)-\frac{d\Re(s)}{2}<0$ in this case. 
\item Let $A=B=-1/2$, so that $n=2\ell+1$ and $\kappa=\tfrac{\ell(\ell+1)}{2}$. Let $d=en$ with $e\geq 2$. Then 
\begin{align*}
\Psi_{-}(0;d)-\tfrac{d\Re(s)}{2}&=e\{\tfrac{\ell^2}{2}-\tfrac{1}{12}(e+7)(2e+1)-\tfrac{n\Re(s)}{2} \} \\  
&<e\left\{ \tfrac{\ell^2}{2}-\tfrac{1}{12}(e+7)(2e+1)
-\tfrac{n}{2}\left(\tfrac{\ell}{2}+\tfrac{1-\epsilon}{\ell+1}\right)
\right\}+\tfrac{1}{2}\ell
\\
&=e\left\{ \tfrac{-\ell}{4}-\tfrac{1}{12}(e+7)(2e+1)
-\tfrac{2\ell+1}{2}\tfrac{1-\epsilon}{\ell+1}
\right\}+\tfrac{1}{2}\ell<-1.
\end{align*}
The case $n=en-2B$ is similar. 
\end{itemize}
\end{proof}

Set 
$$
b_o:=\tfrac{2}{n-2B}\left(\tfrac{\ell(n-2B-\ell-1)}{2}+1\right)=\begin{cases} \frac{1}{\ell+2}(\frac{\ell(\ell+3)}{2}+1) \quad &(B=-1), \\ 
\frac{1}{\ell+1}(\frac{\ell(\ell+1)}{2}+1) \quad &(B=-1/2).
\end{cases}
$$

\begin{proposition} \label{type2Factor-P1}
The Euler product \eqref{Type2Factor-f1} converges absolutely and locally uniformly on $\Re(s)>b_o$. It has a meromorphic continuation to a half-plane containing the point $b_o$ admitting $b_o$ a simple pole. The point $b_o$ is the abscissa of convergence of \eqref{Type2Factor-f1}.
\end{proposition}
\begin{proof} By the comparison of the $p$-factors, we have
\begin{align}
{\mathfrak E}_{\ell,A,B}^{\rm total}(s)=
\prod_{p}\zeta_p\left(\tfrac{n-2B}{2}s-\kappa\right)\times \prod_{p}\Bigl\{ ({1-p^{\kappa-\frac{n-2B}{2}s}) W_{\ell,A,B}^{\rm total}(p,p^{-s/2})}\Bigr\}. 
\label{type2Factor-P1-f0}
\end{align}
The first Euler product converges for $\Re(s)>b_o$ and has a meromorphic continuation to $\C$ admitting a simple pole at $s=b_o$. By Lemma \ref{type2Factor-L1} the second Euler product converges on the half-plane $\Re(s)>b'_o$ with $b'_o<b_o$. Since $\mathfrak E_{\ell,A,B}^{\rm total}(s)$ is a Dirichlet series of positive coefficient by Proposition \ref{IdxFtSrf2-T2}, the abscissa of convergence is indeed $b_o$ by Landau's lemma. \end{proof}

\subsection{Global counting of maximal lattices} \label{sec:GlobalCTN}
 Let $V$ be a finite dimensional $\Q$-vector space and 
 $\beta:V \times V \rightarrow {\mathbb Q}$ be a non-degenerate symmetric bilinear form. For $p\in {\mathbb P}$, set $V_p:=V \otimes_{\Q}\Q_p$. Then, by extension of scalars, $\beta$ yields a symmetric $\Q_p$-bilinear form $\beta_p:V_p \times V_p \rightarrow \Q_p$. Endowed with the quadratic form $Q_p(v):=\frac{1}{2}\beta_p(v,v)$, the $\Q_p$-vector space $V_p$ is viewed as a quadratic space. Fix a Witt decomposition $V_p=T_p\oplus T^*_p\oplus W_p$ as in \eqref{WittDec} of $V_p$ for each $p$; let $\ell_p$ denote the Witt index of $(V_p,Q_p)$ and define $n_{0,p}$ 
  by $n_{0,p}:={\rm dim}_{\Q_p} W_{0,p}$. Let $f_p\in \{1,2\}$ be the integer defined as in \eqref{Def-fmuG} for the quadratic space $(V_p,Q_p)$.  
 Let $L\in \fL_V$. The number $D_L:=\det((\beta(v_i,v_j))_{ij})\in \Z$ with $\{v_j\}_j$ being any $\Z$-basis of $L$ is called the discriminant of $L$, which is independent of the choice of $\{v_j\}_j$. Let $\widehat L:=\{v\in V\mid \beta(v,L)\subset \Z\}$ be the dual lattice of $L$. For a prime $p$, set $L_p:=L\otimes_\Z \Z_p$, which is a $\Z_p$-lattice in $V_p$ so that we can form its dual lattice $\widehat{L_p}$ in $V_p$ with respect to $\beta_p$. Note that $(\widehat L)_{p}=\widehat{L_p}$. We have $\widehat L_p=L_p$ if and only if $D_L\in \Z_p^\times$. By lemma \ref{ML-L2}, if $D_L\in \Z^\times_p$, then $\mu(L_p)=\Z_p$; this is the case for almost all $p\in {\mathbb P}$. Noting this, for $L \in \fL_V$, we define the norm $\mu(L)$ of $L$ to be a positive rational number determined by $\mu(L){\mathbb Z}_p=\mu(L_p)$ for all $p\in {\mathbb P}$.
 
 \begin{definition}
A lattice $L\in \fL_V$ is said to be maximal if $L$ is a maximal element of $\{M\in \fL_V \mid \mu(M)=\mu(L)\}$. Equivalently, $L$ is maximal if for each $p\in {\mathbb P}$, the $\Z_p$-lattice $L_p$ is a maximal $\mu(L_p)$-integral lattice in $V_p=V\otimes_\Q \Q_p$. 
Let $\fL_\beta^{\rm max}$ denote the set of all the maximal lattices in $V$.
\end{definition} The next lemma shows that $\fL_\beta^{\rm max}$ is an infinite set. 
\begin{lemma} For any $\Lambda\in \fL_V$, there exists $L\in \fL_\beta^{\rm max}$ such that $\Lambda\subset L$ and $\mu(L)=\mu(\Lambda)$. 
\end{lemma}
\begin{proof} Let $S:=\{p\in {\mathbb P} \mid D_L\not \in \Z_p^\times\}$; then $S$ is a finite set. If $p\not\in S$, then $\widehat \Lambda_p=\Lambda_p$, and $\Lambda_p$ is a maximal lattice in $V_p$ by Lemma \ref{ML-L2}. If $p\in S$, by Lemma \ref{LM-L00}, we can find a maximal lattice $L_p$ in $V_p$ containing $\Lambda_p$ and $\mu(L_p)=\mu(\Lambda_p)$. Then, $L:=\bigcap_{p\in {\mathbb P}}(V\cap L_p)$ with $L_p=\Lambda_p\,(p\not\in S)$ is a maximal lattice in $V$ containing $\Lambda$.    
\end{proof}

Fix $L\in \fL_\beta^{\rm max}$ and define 
$$
\fL_{\beta,L}^{\rm max}:=\{\Lambda \in \fL_\beta^{\rm max} \mid \Lambda \subset L\}.
$$
Then, we consider the (formal) Dirichelt series
\begin{align}
\zeta_{\beta,L}(s):=\sum_{\Lambda \in \fL_{\beta,L}^{\rm max}} [L:\Lambda]^{-s}, \quad s\in \C. 
\label{ZetaLattice}
\end{align} 
From \S\ref{sec:IdxFtSrf2}, recall the local counterpart $\zeta_{\beta_p,L_p}(s)$ given as in \eqref{localGNS}; it is related to the local index functions series $\boldsymbol\zeta_{\beta_p,L_p}^{(i)}(T)$ as $\zeta_{\beta_p,L_p}(s)=\boldsymbol\zeta_{\beta_p,L_p}^{(0)}(p^{-ns/2})$ (Corollary \ref{ML-L6}) if $f_p=1$ and as in \eqref{totalZ} if $f_p=2$. 
\begin{lemma} \label{ML-L7} We have the formal Euler product decomposition
$$
\zeta_{\beta,L}(s)=\prod_{p\in {\mathbb P}} \zeta_{\beta_p,L_p}(s). 
$$
\end{lemma}
\begin{proof} The set $\fL_{\beta,L}^{\rm max}$ is in bijective correspondence with the set of all systems $\{\Lambda(p)\}_{p\in {\mathbb P}}$ such that $\Lambda(p)\in \fL_{\beta_p, L_p}^{\rm max}$ with $\Lambda(p)=L_p$ for almost all $p\in {\mathbb P}$ by the map $\Lambda \mapsto \{\Lambda_p\}_{p \in {\mathbb P}}$ and its inverse $\{\Lambda(p)\}_{p<\infty} \mapsto \Lambda:=\bigcap_{p\in {\mathbb P}}(V\cap \Lambda(p))$. Since $[L:\Lambda]=\prod_{p\in {\mathbb P}}[L_p:\Lambda_p]$ for any $\Lambda \in \fL_L
$, we have the Euler product decomposition $\zeta_{\beta,L}=\prod_{p\in {\mathbb P}}\zeta_{\beta_p,L_p}(s)$. 
\end{proof}

Now, we move on to showing the convergence of the series \eqref{ZetaLattice} on a half-plane and to studying its analytic nature. For that we need a lemma. We say that a prime $p$ is unramified if $D_L\in \Z^\times_p$; this is the case for almost all $p$. Let $(\Z_p^\times)^2:=\{u^2\mid u \in \Z_p^\times\}$. For $L\in \fL_{\beta}^{\rm max}$, by choosing a Witt decomposition \eqref{WittLat} of $L_p$ for each $p$ and set  $\partial_{L_p}:=\dim_{{\mathbb F}_p}((L_{p})_0/(L_{p})_0^{[+1]})$ with $(L_{p})_0^{[+1]}:=\{w\in W_p\mid Q_p(w)\in p\mu(L_p)\}$.

 \begin{lemma} \label{UMDLTT}
 Suppose $D_L\in \Z_p^\times$. 
 \begin{itemize}
     \item[(i)] If $n$ is odd, then $\ell_p=(n-1)/2$, $n_{0,p}=1$, $\partial_{L_p}=1$ and $f_p=2$. 
     \item[(ii)]
     If $n$ is even and $(-1)^{n/2} D_L\in (\Z_p^\times)^2$, then $\ell_p=n/2$, $n_{0,p}=0$, $\partial_{L_p}=0$ and $f_p=1$, If $n$ is even and $(-1)^{n/2} D_L\not\in (\Z_p^\times)^2$, then $\ell_p=n/2-1$, $n_{0,p}=2$, $\partial_{L_p}=2$ and $f_p=2$.
   \end{itemize}  
 \end{lemma}
 \begin{proof} Since $D_L\in \Z_p^\times$, we have $\widehat L_p=L_p$, which implies $\widehat {L_{0,p}}=L_{0,p}$, where $L_{0,p}$ is from the Witt decomposition \eqref{WittDec} of $L_p$,
 $\mu(L)\in \Z_p^\times$ and $D_{L_{0,p}}\in \Z_p^\times$. By the classification of the maximal integral lattices in an anisotropic quadratic spaces (\S 28 \cite{Sh2010}), we have the following 3 possibilities of $L_{0,p}$:
 \begin{itemize}
 \item $L_{0,p}=(0)$, so that $n_{0,p}=\partial_{L_p}=0$ and $f_p=1$.  
  \item $L_{0,p}=\Z_p w_1$ with $2^{-1}\beta(w_1,w_1)\in \Z_{p}^\times$.
   This falls in the first case in the first line of table \eqref{TableAB}, so that $f_p=2$ and $\partial_{L_p}=1$. 
  \item  there is an isometry $(L_{0,p}, \beta)\cong (O_{E}, u{\rm tr}_{E/\Q_p})$ with $u\in \Z_p^\times$ and  $E$ being the unramified quadratic extension of $\Q_p$ and ${\rm tr}_{E/\Q_p}$ the trace map. If $p\not=2$, then $O_{E}=\Z_p\oplus \sqrt{c}\Z_p$ with $c\in \Z_p^\times-(\Z_p^\times)^2$; if $p=2$, then $O_{E}=\Z_2\oplus \frac{1-\sqrt{5}}{2}\Z_p$. This falls in the first case of line 3 in table \eqref{TableAB}, so that $f_p=2$ and $\partial_{L_p}=2$.   
 \end{itemize}
 If $n\,(=2\ell_p+n_{0,p})$ is odd, we necessarily have case (i) so that $n_{0,p}=1$ and $\ell_p=(n-1)/2$. Suppose $n$ is even. 
 The Witt decomposition \eqref{WittDec} of $L_p$ also implies 
 $$D_{L}(\Z_p^\times)^2=(-\mu(L)^2)^{\ell_p}D_{L_0}(\Z_p^\times)^2=(-1)^{\ell_p}D_{L_0}(\Z_p^\times)^2.$$ 
Thus, $D_{L} \in (\Z_p^\times)^2$ is equivalent to $(-1)^{\ell_p} D_{L_0} \in (\Z_p^\times)^2$. If $n_{0,p}=0$, then $D_{L_{0,p}}=1$ and $D_{L}(\Z_p^\times)^2=(-1)^{n/2}(\Z_p^\times)^2$, which means $(-1)^{n/2}D_L \in (\Z_p^\times)^2$. If $n_{0,p}=2$, then $D_{L}(\Z_p^\times)^2=(-1)^{n/2-1}\,(-D_E)(\Z_p^\times)^2=(-1)^{n/2}D_E(\Z_p^\times)^2$ with $D_E$ the discriminant of the unramified quadratic extension $E/\Q_p$, which means $(-1)^{n/2}D_{L}\not\in (\Z_p^\times)^2$.      
 \end{proof}

\subsubsection{{\underline{Euler product \textup{(}Even degree case\textup{)}}}} \label{sec:O(EvenDegree)} 
Suppose $n=\dim_\Q(V)$ is even. Let ${\rm I}(L)$ (resp. ${\rm S}(L)$) denote the set of primes $p$ such that $D_L\in \Z_p^\times$ and $n_{0,p}=2$ (resp. $D_L\in \Z_p^\times$ and $n_{0,p}=0$); by Lemma \ref{UMDLTT}, ${\rm I}(L)=\{p\in {\mathbb P} \mid (-1)^{n/2}D_L\in \Z_p^\times-(\Z_p^\times)^2\}$ and ${\rm S}(L)=\{p\in {\mathbb P} \mid (-1)^{n/2}D_L\in (\Z_p^\times)^2\}$. 
Let ${\bf k}$ be the discriminant field of the quadratic space $(V,Q)$ (see \cite[Pg.119]{Sh2010}). Since ${\bf k}=\Q(\sqrt{(-1)^{n/2}D_L})$, the set ${\rm I}(L)$ is infinite if and only if ${\bf k}$ is a quadratic field, in which case ${\rm I}(L)$ contains all primes inert in ${\bf k}$. For any finite set $S$ of containing all prime factors of $D_L$, consider the Euler products 
\begin{align*}
E_{n}^{S}(s)&:={\mathfrak E}_{\frac{n}{2}-1,0,-1}^{S(L)\cup S}(s)
 \times  \prod_{p \in {\rm S}(L)-S}
W_{\frac{n}{2}}^{(+)}(p,p^{-\frac{ns}{2}-1}),
\end{align*}
where the first factor is the partial Euler product obtained from ${\mathfrak E}_{\frac{n}{2}-1, 0,-1}(s)$ (see \eqref{Type2Factor-f1}) by removing its local factors over $p\in {\rm S}(L)\cup S$. Let $\zeta^S_{\beta,L}(s)$ be the Euler product of the local factors $\zeta_{\beta_p,L_p}(s)$ over all $p\in S$.  
Below, we will obtain the Euler product expression
\begin{align}
\zeta^{S}_{\beta,L}(s)&=\prod_{r=0}^{\frac{n}{2}}\zeta^{S}\left(\tfrac{ns}{2}-\tfrac{r(n-r-1)}{2}\right) \label{THM1-f00}\\
&\quad \times \prod_{p\in {\rm I}(L)-S}\Bigl\{\prod_{r=0}^{\frac{n}{2}-1}(1+p^{-\frac{ns}{2}+\frac{1}{2}r(n-r-1)})^{-1} \times  
(1-p^{-\frac{ns}{2}+\frac{1}{8}n(n-2)})\Bigr\} \times E_n^{S}(s).
 \notag
\end{align} 
on a half-plane, where $\zeta^S(z)$ denotes the Euler product of $(1-p^{-z})^{-1}$ over $p\not\in S$. 

First we deal with the case ${\bf k}=\Q$. Set 
\begin{align} 
\sigma_n:=\tfrac{2}{n}\left(\tfrac{n(n-2)}{8}+1\right). 
\label{RghtMstPl}
\end{align}
\begin{theorem} \label{THM1} Suppose that $n$ is even and that ${\bf k}=\Q$. Then, the Euler product $E_{n}^{S}(s)$ converges absolutely on $\Re(s)>\sigma$ with some $0<\sigma<\sigma_n$. The function $\zeta_{\beta,L}(s)$ is meromorphic on $\Re(s)>\sigma$ and its right-most pole is identified with $\sigma_n$, which is a double pole.  
Let $A_{\beta,L}$ be the leading Laurent coefficient of $\zeta_{\beta,L}(s)$ at $s=\sigma_n$, i.e., $\zeta_{\beta,L}(s)=A_{\beta,L}(s-\sigma_n)^{-2}(1+{\mathcal O}(s-\sigma_n))$. Then $A_{\beta,L}$ is a positive real number. 
\end{theorem}
\begin{proof} Since $I(L)$ is a finite set, we only need to show the absolute convergence of $\prod_{p  \in {\rm S}(L)-S}W_{\frac{n}{2}}^{(+1)}(p,p^{-\frac{ns}{2}-1})$, which follows from Theorem \ref{dSW-Intro}. 
Below is a direct argument free from inputs from\cite{dSW2008}. 
By Theorem \ref{EFPell-P3} (iii), the Euler product converges absolutely for $\Re(s)>\max_{d\in [1,\ell-1]} \tfrac{2}{n}\beta_{n/2,d}^{(+1)}$. The maximum is attained at
$d=1$, which yields the value $\frac{2}{n}(\frac{\ell(\ell-1)}{2}-1)
$. Since $\sigma_n=\frac{2}{n}(\frac{n(n-2)}{8}+1)$ is larger than this number, we are done. The poles of the Riemann zeta factors in \eqref{THM1-f00} are 
 $$
 s=\tfrac{2}{n}\{ \tfrac{r(n-r-1)}{2}+1\}\,(0\leq r\leq n/2),
 $$
the largest of which is $\sigma_n$ corresponding to $r=\frac{n}{2},\frac{n}{2}-1$. 
 By the first claim of Proposition \ref{IdxFtSrf2-T2}, the functions $s\mapsto W_{\ell_p,A_p,B_p}^{\rm total}(p,,p^{-s/2})$ for each $p<\infty$ are positive for $s\in \R$. Hence, the poles of $\zeta^S_{\beta,L}(s)$ coming from the Riemann zeta factors are indeed poles of $\zeta_{\beta,L}(s)$. The last claim results from the explicit description of $A_{\beta,L}$ in terms of $W_{\ell,0,-1}^{\rm toral}(p,p^{-\sigma_n})$ and $W_{\ell}^{(+1)}(p,p^{-\sigma_n n/2-1})$, which are positive. \end{proof}

\begin{corollary} \label{Asymptotic(EvenSP)}
Suppose $n$ is even and ${\bf k}=\Q$. Then, We have the asymptotic formula
\begin{align*}
{\rm N}_{\beta,L}(X):=\#\{\Lambda \in \fL_{\beta,L}^{\rm max} \mid [L:\Lambda]<X\} \sim \sigma_n^{-1}{A_{\beta,L}}\, {X^{\sigma_n}}{\log X}, \quad X \rightarrow \infty.    
\end{align*}   
\end{corollary}
Suppose that ${\bf k}$ is a quadratic field.  Let ${\rm I}({\bf k})$ denote the set of all primes inert in ${\bf k}$. Set
$$
\sigma_n':=\tfrac{2}{n+2}\left(\tfrac{(n-2)(n+4)}{8}+1\right),
$$
which is larger than $\sigma_n$. 

\begin{lemma}\label{Even(GO)-L1}
The partial Euler product $\eta_{{\bf k}}(s):=\prod_{p \in {\rm I}({\bf k})}(1-p^{-z})^{-2}$ is absolutely convergent on $\Re(z)>1$. The square $\eta_{\bf k}(s)^2$ has a meromorphic continuation to $\Re(z)>1/2$ having a simple pole $z=1$. We have ${\rm Res}_{z=1}(\eta_{\bf k}(z)^2)>0$. 
\end{lemma}
\begin{proof}
Up to a finite Euler factors, $\eta_{\bf k}(s)$ equals to $
\zeta_{{\bf k}}(z)^{-1}\zeta(z)^2\times \prod_{p\in {\rm I}({\bf k})}(1-p^{-2z})$. From this, it has a meromorphic continuation to a half-plane $\Re(z)>1/2$. Since ${\rm Res}_{z=1}\zeta_{{\bf k}}(z)>0$ and the all the Euler factors involved are positive at $z=1$, we have the last assertion. 
\end{proof}

\begin{theorem}\label{Thm2(EvenGO)} Suppose that $n$ is even and that ${\bf k}$ is a quadratic extension. Then, the Euler product ${\mathfrak E}_{\frac{n}{2}-1,0,-1}^{{\rm S}(L)\cup S}(s)$ converges absolutely on the half plane $
\Re(s)>\sigma_n'$. The squared function $({\mathfrak E}_{\frac{n}{2}-1,0,-1}^{{\rm S}(L)\cup S}(s))^{2}$, as well as $\zeta_{\beta,L}(s)^{2}$, has a meromorphic continuation to a half-plane containing $\sigma_n'$, admitting $\sigma_n'$ as a unique simple pole with ${\rm Res}_{s=\sigma_n'}(\zeta_{\beta,L}(s)^2)>0$.     
\end{theorem}
\begin{proof} The first assertion results from Proposition \ref{type2Factor-P1}. We have ${\rm I}(L)\subset {\rm I}({\bf k})$ and ${\rm I}({\bf k})-{\rm I}(L)$ is a finite set, so that it suffices to consider the Euler product ${\mathfrak E}_{\frac{n}{2}-1,0,-1}^{S({\bf k})\cup S}(s)$, which is a partial Euler product of \eqref{type2Factor-P1-f0} over $p\in {\rm I}({\bf k})-S$. From the proof of Lemma \ref{type2Factor-L1}, the second partial Euler product in \eqref{type2Factor-P1-f0} is absolutely convergent on $\Re(s)>\sigma_n'$. The first partial Euler product in \eqref{type2Factor-P1-f0} is $\eta_{{\bf k}}(z)$ with $z=\frac{s-2B}{2}s-\kappa$, whose behavior at $z=1$ is given by Lemma \ref{Even(GO)-L1}. By the same reasoning as in the proof of Theorem \ref{THM1}, the Euler product $\prod_{p \in {\rm S}(L)-S}W_{\frac{n}{2}}^{(+1)}(p,p^{-ns/2-1})$ is absolutely convergent on $\Re(s)>\sigma_n$, which is larger than $\Re(s)>\sigma_n'$. \end{proof}

\begin{corollary} \label{Asymptotic(EvenNS)}
Suppose $n$ is even and ${\bf k}$ is a quadratic field. Then, We have the asymptotic formula
\begin{align*}
{\rm N}_{\beta,L}(X):=\#\{\Lambda \in \fL_{\beta,L}^{\rm max} \mid [L:\Lambda]<X\} \sim \frac{A_{\beta,L}}{\sigma_n'\sqrt{\pi}} \frac{X^{\sigma_n'}}{(\log X)^{1/2}}, \quad X \rightarrow \infty    
\end{align*}
with $A_{\beta,L}:=\sqrt{{\rm Res}_{s=\sigma_n'}(Z_{\beta,L}(s)^2)}$.   \end{corollary}
\begin{proof} This follows from Theorem \ref{Thm2(EvenGO)} by applying \cite[Theorem 2]{Kable2008}.   
\end{proof}

\subsubsection{{\underline{Euler product\textup{(}Odd degree case\textup{)}}}}\label{sec:ZetaOdd}

Suppose $n=\dim_\Q(V)$ is odd. For any finite set $S$ of primes containing all the prime factors of $D_L$, let $E^{S}_n(s)$ be the partial Euler product obtained from ${\mathfrak E}_{\frac{n-1}{2}, -\frac{1}{2}, -\frac{1}{2}}(s)$ by removing Euler factor over $S$. Set 
$$
\sigma_n:=\tfrac{n^2+1}{n+1}.
$$

\begin{theorem} \label{THM2Odd}
The Euler product $E_{n}^S(s)$ converges absolutely on a half plane $\Re(s)>\sigma$ and admits a meromorphic continuation to a half plain containing $\sigma_n$, admitting the simple pole at $s=\sigma_n$. We have   
\begin{align}
\zeta^{S}_{\beta,L}(s)&=\prod_{r=0}^{\frac{n-1}{2}}\zeta^{S}\left({ns}-{r(n-r-1)}\right)\times E_n^{S}(s), 
 \label{THM2Odd-f0}
\\
E_{n}^{S}(s)&={\zeta^{S}\left(\tfrac{(n+1)s-n^2+1}{2}\right)}
\times \prod_{p\in{\mathbb P}- S} \bigl\{(1-p^{-\frac{(n+1)s-n^2+1}{2}})W_{\frac{n-1}{2},-\frac{1}{2}, -\frac{1}{2}}^{\rm total}(p,p^{-\frac{s}{2}})\bigr\}
 \label{THM2Odd-f2}
\end{align} 
for $\Re(s)>\sigma$. The function $\zeta_{\beta,L}(s)$ has a simple pole at $\sigma_n$, which is right-most. The residue $A_{\beta,L}:={\rm Res}_{s=\frac{n^2+1}{n+1}}\zeta_{\beta,L}(s)$ is a positive real number. 
\end{theorem}
\begin{proof} This is a direct consequence of Proposition \ref{type2Factor-P1}. Note that $b_o$ for $n=2\ell+1$ equals $\sigma_n$. 
\end{proof}

\begin{remark} \label{NaBdl} The right-most pole $\sigma_n$ of $E_n^S(s)$ exceeds the right-most pole $\frac{\ell^2+1}{n}$ of the Riemann zeta factors in \eqref{THM2Odd-f0}. Examples suggest that $\frac{n^2-1}{n+1}$ would be the natural boundary of $E_n^{S}(s)$. 
\end{remark}
As a corollary to Theorem \ref{THM2Odd}, by the Tauberian theorem, we have the asymptotic formula 
\begin{corollary} \label{AsymptoticOdd}
\begin{align}
{\rm N}_{\beta,L}(X):=\#\{\Lambda \in \fL_{\beta,L}^{\rm max}\mid [L:\Lambda]<X\}\sim A_{\beta,L} X^{\frac{n^2+1}{n+1}}, \quad X\rightarrow +\infty.
 \notag
 \end{align}
\end{corollary}

\subsection{Explicit Zeta Functions}\label{sec:EXPLITZETA}
\begin{itemize}
\item[(1)]({\underline{${\rm D}_\ell$-case}})
Let $V=\Q^{2\ell}$, the space of column vectors of dimension $2\ell$, and 
$$
\beta_{\ell,\ell}(x,y):={}^t{\bf x} \left[\begin{smallmatrix} 0 & J_\ell \\
J_\ell & 0 \end{smallmatrix}\right] {\bf y}, \quad x,y\in V.
$$
The standard lattice $\Z^{2\ell}$ is even unimodular, in particular maximal $\Z$-integral. We have $(\ell_{p},n_{0,p})=(\ell,0)$ for all $p\in {\mathbb P}$. We write $Z({\rm D}_\ell;s)$ for $\zeta_{\beta_{\ell,\ell},\Z^{2\ell}}(s)$, because this coincides with the group-zeta function $\prod_{p\in {\mathbb P}}Z_{G(\Q_p),\rho}(s)$ for the inclusion $\rho:G:={\rm GO}^0(\beta_{\ell,\ell})\hookrightarrow {\bf GL}_{2\ell}$. By Theorem \ref{THM1} and its proof, 
$$
Z({\rm D}_{\ell};s)=\prod_{r=0}^{\ell} \zeta\left(\tfrac{ns}{2}+\tfrac{r(r-n+1)}{2}\right)\times \prod_{p\in {\mathbb P}}W^{(+1)}_\ell(p,p^{-\ell s-1}),\quad \Re(s)\gg 0 
$$
The right-most pole is $\sigma_n=\frac{1}{\ell}(\frac{\ell(\ell-1)}{2}+1)$, which is a double pole and the leading coefficient in the Laurent series expansion at $\sigma_n$ is 
$$
\prod_{r=1}^{\ell}\zeta\left(\tfrac{r^2-r+2}{2}\right) \times \prod_{p\in {\mathbb P}}W^{(+1)}(p, p^{-\frac{\ell(\ell-1)}{2}-2}).
$$
This is a positive real number because $W_{\ell}^{(+1)}(p,T)$ has positive integer coefficients. 

\begin{example}
\begin{enumerate} 
\item[(i)] Let $\ell=2$. By \eqref{Exp1} and Theorem \ref{THM1}, 
$$
Z({\rm D}_2; s)=\zeta(2s)\zeta(2s-1)^2\prod_{p}(1+p^{-2s}).
$$
This agrees with Example 2 of \cite{BG2005} where the canonical orthogonal form is used. 
\item[(ii)] Let $\ell=3$. By \eqref{Exp2} and Theorem \ref{THM1}, 
$$
Z({\rm D}_3; s)=\zeta(3s)\zeta(3s-2)\zeta(3s-3)^2\prod_{p} (1+(1+2p+p^2)p^{-3s}+p^{-6s}).
$$
\item[(iii)] Let $\ell=4$. Then by \eqref{Exp3} and Theorem \ref{THM1}, 
\begin{align}
Z({\rm D}_4; s)&=\zeta(4s)\zeta(4s-3)\zeta(4s-5)\zeta(4s-6)^2\times \prod_{p}1+w_+(p)p^{-4s}+w_+(p^{-1})p^{-8s+8}+p^{-12s+8})
 \label{Zeta-D8}
 \end{align}
with
\begin{align}
w_+(p)&:=1+2p+2p^2+3p^3+2p^4+p^5.
 \label{Coff-Ap+}
\end{align}
\end{enumerate}
\end{example}
\item[(2)]({\underline{${\rm B}_\ell$-case}}) Let $V=\Q^{2\ell+1}$, the space of column vectors of dimension $2\ell+1$, and $((2)\oplus \beta_{\ell,\ell})(x,y):={}^t {\bf x} \left[\begin{smallmatrix} {} & {} & J_2 \\ {} & {2} & {} \\ 
J_2 & {} & {} \end{smallmatrix}\right]{\bf y}$ for ${\bf x},{\bf y}\in V$. It turns out that
the standard $\Z$-lattice $L=\Z^{2\ell+1} \subset \Q^{2\ell+1}$ is a maximal 
$\Z$-integral lattice. For an odd prime $p<\infty$, $\Z_p$ is unimodular, whereas the dyadic
completion $L_2$ is not $(2)\oplus \beta_{\ell,\ell}$-polarized. We have
$(\ell_p,n_{0,p})=(\ell,1)$ for all $p<\infty$. We write $\zeta({\rm B}_\ell;s)$ for 
$\zeta_{(2)\oplus \beta_{\ell,\ell},\Z^{2\ell+1}}(s)$. Note that $\zeta({\rm B}_\ell;s)$ is different from the group zeta function $Z({\rm B}_\ell;s):=\prod_{p\in {\mathbb P}}Z_{G(\Q_p),\rho}(s)$ for $\rho:G:={\rm GO}^0((2)\oplus \beta_{\ell,\ell})\hookrightarrow {\bf GL}_{2\ell+1}$. By Theorem \ref{IdxFtSrf2-T1}, 
\begin{align*}
Z({\rm B}_\ell;s)&=
\prod_{r=0}^{\frac{n-1}{2}}\zeta\left({ns}-{r(n-r+1)}\right) \times \prod_{p\in {\mathbb P}} W_{\frac{n-1}{2},-1/2,-1/2}^{(0)}(p,p^{-ns}),
\\
\zeta({\rm B}_\ell;s)&= 
\prod_{r=0}^{\frac{n-1}{2}}\zeta\left({ns}-{r(n-r+1)}\right) \times \prod_{p \in {\mathbb P}}W_{\frac{n-1}{2},-1/2,-1/2}^{\rm total}(p,p^{-1/2}, p^{-s/2}).
\end{align*}
The Euler products are absolutely convergent on a half plane (By Theorem \ref{THM2Odd}). 
\begin{example} \label{ExampleOddSym}
By \S\ref{sec:OddORTExample},
\begin{align*}
\zeta({\rm B}_1;s)=\zeta({3s})\zeta(3s-1)&\prod_{p}(1+[2]_pp^{-2s}+p^{-3s}), \\
\zeta({\rm B}_2;s)=\zeta(5s)\zeta(5s-3)\zeta(5s-4)&\prod_{p}(1+[4]_pp^{-3s}
+[3]_p(1+p^2)p^{-5s}+[4]_pp^{2-8s}+p^{3-10s})
\end{align*}
with $[n]_p=1+p+\cdots+p^{n-1}$. 
\end{example}
\item[(3)] ({\underline{non-split ${\rm GO}(6)$-case}}) Let $\beta({\bf x}, {\bf y}):={}^t{\bf x} \left[\begin{smallmatrix} 0 & 0 & J_2 \\ 0 & 2 1_2 & 0 \\ J_2 & 0 & 0 \end{smallmatrix} \right]{\bf y}$ for ${\bf x}, {\bf y}\in \Q^6$; then the lattice $\Z^{6}$ is a maximal lattice such that $\mu(\Z^6)=1$. The discriminant field is ${\bf k}=\Q(i)$ and  
\begin{align*}
(\ell_p,n_{0,p};f_p;A_p,B_p)=
\begin{cases}
(3,0;1;-1,0) \quad &(\text{if $p\equiv 1 \pmod{4}$}), \\ 
(2,2;2;0,-1) \quad &(\text{if $p\equiv 3 \pmod{4}$}),  \\
(2,2;1;0,0) \quad &(\text{if $p=2$}).
\end{cases}
\end{align*}
Thus, $\zeta_{\beta,\Z^6}(s)$ equals
\begin{align*}
&\zeta_2(3s)\zeta_2(3s-3)^2(1+2^{1-3s})
\\
&\times \prod_{p\equiv 1\pmod{4}}\zeta_p(3s)\zeta_p(3s-2)\zeta_{p}(3s-3)^2\Bigl(1+(1+p)^2p^{-3s}+p^{-6s}\Bigr) \\
&\times \prod_{p\equiv 3\pmod{4}} \zeta_{p}(6s)\zeta_p(6s-4)\zeta_p(6s-6)\Bigl(1+(1+p^2)(1+p^3)p^{-4s}+(1+p)(1+2p^3)p^{-6s} \\
&\qquad \qquad \qquad +(1+p^2)(1+p^3)p^{3-10s}+p^{4-12s}\Bigr). 
\end{align*}
The first Euler product converges on $\Re(s)>4/3$ and the second one converges on $\Re(s)>3/2=\sigma'_{6}$. The square $(\zeta_{\beta,\Z^6}(s))^2$ has a meromorphic continuation to a half-plane $\Re(s)>3/2-\epsilon$ admitting the simple pole $s=3/2$ as a unique singularity.

\end{itemize}

\section{The \texorpdfstring{${\rm GSp}(n)$}{}-case}\label{sec:Symp}

In this section, Theorem \ref{MAINTHM2} for $\varepsilon=-1$ is restated as Theorem \ref{THM2} with some further explanation followed by concrete examples. In \cite{HS1983}, all classical groups are handled with a unified notation and the theorems and proofs are so described 
to cover all cases including the symmetric and the alternating cases. 
Similarly these two cases could have been treated simultaneously here. Instead, for a clearer presentation, we separate the sections for the two cases, though the explanation for the alternating case is brief
and without all the proofs since these are basically the same as for the orthogonal case. Since the arithmetic of alternating forms tends to be much simpler than that of the quadratic forms, this makes sense. 

Let $\beta:V\times V\rightarrow \Q$ be a non-degenerate alternating form (i.e., $\beta(v,w)=-\beta(w,v))$ on a $\Q$-vector space $V$ of dimension $n$, i.e., $\beta(v,v)=0\,(\forall\,v\in V)$). Then, $n=\dim V$ is even, so that we may set $n=2\ell$ with $\ell \in \Z_{>0}$. There exists a $\Q$-basis $\{v_i,v_i^*\}_{i=1}^{\ell}$ of $V$ such that 
\begin{align*}
&\beta(v_i,v_j)=\beta(v_i^*,v_j^*)=0, \quad \beta(v_i,v_j^*)=\delta_{ij}\quad (i,j \in [1,\ell]). 
\end{align*}
Such a basis is called a symplectic basis. The dual lattice of $L\in 
\fL_V$ is defined as $\widehat L:=\{v\in V\mid \beta(v,L)\subset \Z\}$. 
Then, $L\in \fL_V$ is called  $\beta$-polarized if $\widehat L=c L$ for 
some $c\in \Q^\times$. The norm $\mu(L)$ of $L\in \fL_V$ is defined as 
a unique positive rational number such that $\mu(L)\Z_p$ is the 
fractional ideal in $\Q_p$ generated by the values $\beta(v,w)\,(v,w\in 
L_p)$. Then a $\Z$-lattice $L$ is referred to as being maximal if it is
maximal among all lattices with the same norm $\mu(L)$. A lattice $L$ 
is maximal if and only if there exists a symplectic basis $v_i,v_i^{*}\,
(1\leq i \leq \ell)$ of $V$ and $c\in \Q^\times$ such that 
$L=\sum_{i=1}^{\ell}(\Z v_i +c\Z v_i^*) $ 
with $c=\mu(L)$. Hence a maximal lattice $L$ satisfies the relation 
$\widehat L=\mu(L)^{-1}L$, i.e., $L$ is $\beta$-polarized and 
conversely, a $\beta$-polarized lattice is maximal ({\it cf}. Lemmas 
\ref{ML-L1}, \ref{ML-L2}). Thus, there is no difference between the
maximal lattices and the $\beta$-polarized lattices in the alternating 
case.   

Let $\fL_\beta^{\rm max}$ denote the set of all the maximal lattices. 
Fix $L\in \fL_\beta^{\rm max}$, and let $\fL_{\beta,L}^{\rm max}:
=\{\Lambda \in \fL_{\beta}^{\rm max}\mid \Lambda \subset L\}$. Arguing 
the same way as in \S\ref{sec:ORTH}, we have the following result on 
the Dirichlet series 
$$
\zeta_{\beta,L}(s):=\sum_{\Lambda \in \fL_{\beta,L}^{\rm max}} 
[L:\Lambda]^{-s},
$$
with $W_{\ell}^{(-1)}(X,T)$ as in Definition \ref{Well-Def}.
\begin{theorem} \label{THM2}
The Euler product
$$
 E_{\ell}(s):=\prod_{p\in {\mathbb P}} W_{\ell}^{(-1)}(p, p^{-\ell s})
$$
converges absolutely on $\Re(s)>\sigma$ with some $\sigma>0$. We have  
$$\zeta_{\beta,L}(s)=\prod_{r=0}^{\ell} \zeta\left(\tfrac{ns}{2}+\tfrac{r(r-2\ell-1)}{2}\right)\times E_\ell(s), \quad \Re(s)>\sigma.
$$   
The right-most pole of $\zeta_{\beta,L}(s)\,(\Re(s)>\sigma)$ is $s_0=\frac{1}{\ell}(\frac{\ell(\ell+1)}{2}+1)$, which is simple. We have 
$$
{\rm Res}_{s=s_0}\zeta_{\beta,L}(s)=\prod_{r=1}^{\ell} \zeta\left(\tfrac{r^2+r+2}{2}\right)\times \prod_{p \in {\mathbb P}} W_\ell^{(-1)}(p,p^{-\frac{\ell^2+\ell+2}{2}})>0. 
$$
\end{theorem}
The function $\zeta_{\beta,L}(s)$  coincides with the group-zeta 
function $Z({\rm C}_\ell;s):=\prod_{p\in {\mathbb P}}Z_{{\rm Sp}_{2\ell}
(\Q_p),\rho}(s)$. In particular it is independent of the lattice $L$.

\subsection{Explicit Zeta functions (\texorpdfstring{${\rm C}_\ell$}{}-case)} \label{sec:EXPLICITZETA2}

Let $J_{\ell}$ be as above. Then 
$$
\alpha_{\ell}(x,y)={}^t {\bf x} \left[\begin{smallmatrix} 0 & J_\ell \\ -J_\ell & 0 \end{smallmatrix} \right]{\bf y}, \quad {\bf x},\,{\bf y} \in \Q^{2\ell}
$$
is the standard symplectic space of dimension $2\ell$, and $L_{\ell}=\Z^{2\ell}$ is a self-dual lattice. Theorem \ref{THM2}, together with \eqref{Exp1}, \eqref{Exp2} and \eqref{Exp3}, yields the following: 
\begin{example} \label{ExampleSymplectic}
    
\begin{itemize}
\item[\textup{(1)}] If $\ell=2$, then
$$
Z({\rm C}_2; s)=\zeta(2s)\zeta(2s-2)\zeta(2s-3) \prod_{p}(1+p^{-2s+1}).
$$
The same formula is obtained in Example 1 of \cite{BG2005} using the canonical symplectic form.
\item[\textup{(2)}] $\ell=3$, then
$$
Z({\rm C}_3; s)=\zeta(3s)\zeta(3s-3)\zeta(3s-5)\zeta(3s-6) \prod_{p}(1+(p+p^2+p^3+p^4)p^{-3s}+p^{-6s+5}).
$$
This agrees with the formula in \cite[Pg.48]{dSL1996}.

\item[\textup{(3)}] If $\ell=4$, then 
\begin{align*}
Z({\rm C}_4;s)&=\zeta(4s)\zeta(4s-4)\zeta(4s-7)\zeta(4s-9)\zeta(4s-10) \\
&\quad \times \prod_{p}(1+w_{-}(p)p^{-4s+1}+w_{-}(p^{-1})p^{-8s+13} +p^{-12s+14})
\end{align*}
with
\begin{align}
w_{-}(p)&:=1+p+2p^2+p^3+2p^4+2p^5+p^6+p^7.
 \notag
\end{align}
\end{itemize}
The last Euler $p$-factor of $Z({\rm C}_4;s)$ is factorized as
\begin{align}
(1+p^{5}X)\{1+(p+p^2+2p^3+p^4+p^5+2p^6+p^7+p^8)X+p^{9}X^2\}
\label{C4Euler}
\end{align}
with $X=p^{-4s}$. Vankov \cite[Theorem 1]{Vankov2011} obtained an explicit expression of the numerator, using computers, for the full Hecke series of ${\bf GSp}_8$. Its specialization in \cite[Proposition 5]{Vankov2011} 
factors into polynomials of degree one or two in $X=p^{-4s}$ containing the factors of \eqref{C4Euler}, as is expected \textup{(\cite[Pg.140]{HS1983})}.
\end{example}

\section{Remarks and further problems}

For a general quadratic space $(V,\beta)$, for almost all primes $p$, the form $\beta$ is represented over $\Q_p$ by the matrix of type ${\rm D}_\ell$ or of type ${\rm B}_\ell$ in \eqref{labelCDB}, which means that the group zeta function $Z_{\fG,\rho}(s)$, up to a finite number of Euler $p$-factors, coincides with $Z({\rm D}_\ell,s)$ or $Z({\rm B}_\ell;s)$. We computed the exceptional $p$-factors in Theorem \ref{IdxFtSrf2-T1}. For $Z({\rm D}_\ell,s)$ and $Z({\rm B}_\ell,s)$, their natural boundaries are determined in \cite[\S6]{dSW2008}. 
 When $\dim(V)=2\ell+1$, for each square-free integer $N>0$, we can consider the Dirichlet series, say $\zeta^{(N)}_{\beta,L}(s)$, that counts the number of maximal lattices $\Lambda \in \fL_{\beta}^{\rm max}$ such that $\Lambda \subset L$ and $\mu(\Lambda)/\mu(L)=Nm^2\,(\exists\,m\in \Z_{>0})$. This series $\zeta_{\beta,L}^{(N)}(s)$, up to a finite number of Euler factors at primes dividing $ND_{L}$, coincides with the group zeta-function $Z({\rm B}_\ell;s)$ spelled out in \S\ref{sec:EXPLITZETA} (2). By \cite[Corollary 6.4]{dSW2008}, it has a simple pole at $\sigma:=\frac{2}{n+1}(\frac{n-1}{2})^2$, so that the Tauberian theorem yields an asymptotic 
$$
\sum_{m\leq X} \#\{\Lambda \in \fL_{L,\beta}^{\rm max} \mid \mu(\Lambda)/\mu(L)=Nm^2\}\sim \sigma^{-1} A^{(N)}_{\beta,L} \, X^{\sigma}.
$$
Our result gives a very explicit description of the residue $A_{\beta,L}^{(N)}:={\rm Res}_{s=\sigma}\zeta_{\beta,L}^{(N)}(s)$. We should note that $\sigma$ is smaller than the pole $\frac{n^2+1}{n+1}$ of the full counting series $\zeta_{\beta,L}(s)$. 

For non-split even orthogonal case or the odd orthogonal case, the counting series $\zeta_{\beta,L}(s)$ is different from the group zeta-function. Given an explicit nature of our $W_{\ell,A,B}^{\rm total}(X,Y)$, the theory in \cite[\S5]{dSW2008} may be applied to $W_{\ell,A,B}^{\rm total}(X,Y)$ to yield the natural boundary of $\zeta_{\beta,L}(s)$ ({\it cf}. Remark \eqref{NaBdl}).

 The theory of maximal lattices in a finite dimensional vector space over a quadratic field $E$ with a non-degenerate hermitian form is fully developed by Shimura \cite{Shimura1964}. The same method pursued in this paper also works on the counting problem of the hermitian maximal lattices. We will discuss this case elsewhere. In \cite{BG2005}, another kind of Dirichlet series that counts numbers of double cosets ${\bf Sp}_{2\ell}(\Z)\backslash {\bf Sp}_{2\ell}(\Q)/{\bf Sp}_{2\ell}(\Z)$ with given determinant is mentioned; the orthogonal group analogue may be interesting.


\end{document}